\documentclass[11pt,letterpaper]{amsart}

\makeatletter
\AtBeginDocument{
  \def\@lbibitem[#1]#2{\@bibitem{#2}}
}
\makeatother

\usepackage[
    top=1in, bottom=1in, inner=1in, outer=1in,
    marginparwidth=2cm %necessary to avoid breaking fixme notes.
    ]{geometry}

\usepackage{amssymb,latexsym, amsmath, amsxtra, mathrsfs, bm}
\usepackage[dvips]{graphics}
\usepackage[T1]{fontenc}

\usepackage[all]{xy}
\usepackage{xcolor}
\usepackage{subcaption}
\usepackage{verbatim}
\usepackage[abs]{overpic}
\usepackage[hidelinks]{hyperref}
\usepackage{mathtools}
\usepackage{enumitem}

\DeclareMathAlphabet{\mathbbold}{U}{bbold}{m}{n}

\allowdisplaybreaks[3]

\makeatletter
\def\th@plain{%
	\thm@notefont{}% same as heading font
	\itshape % body font
}
\def\th@definition{%
	\thm@notefont{}% same as heading font
	\normalfont % body font
}
\makeatother
\theoremstyle{plain}
        \newtheorem{theorem}{Theorem}[section]
        \newtheorem*{theorem*}{Theorem}
        \newtheorem{thmx}{Theorem}

        \newtheorem{lemma}[theorem]{Lemma}
        \newtheorem{prop}[theorem]{Proposition}

\theoremstyle{definition}
        \newtheorem{definition}[theorem]{Definition}
        \newtheorem{rem}[theorem]{Remark}

\theoremstyle{remark}

        \newtheorem*{notation}{Notation}

\numberwithin{equation}{section}
\numberwithin{theorem}{section}
\numberwithin{table}{section}
\numberwithin{figure}{section}

\renewcommand{\le}{\leqslant}
\renewcommand{\leq}{\leqslant}
\renewcommand{\ge}{\geqslant}

\newcommand{\supp}{\operatorname{supp}}

\newcommand{\R}{\mathbb{R}}
    
\newcommand{\C}{\mathbb{C}}      

\newcommand{\Q}{\mathbb{Q}}
\newcommand{\N}{\mathbb{N}}      
\newcommand{\Z}{\mathbb{Z}}      
\newcommand{\T}{\mathbb{T}}     % Torus

\providecommand{\abs}[1]{\lvert#1\rvert}
\providecommand{\Absbig}[1]{\bigl\lvert#1\bigr\rvert}
\providecommand{\Absbigg}[1]{\biggl\lvert#1\biggr\rvert}

\providecommand{\norm}[1]{\|#1\|}
\providecommand{\Normbig}[1]{\bigl\|#1\bigr\|}
\providecommand{\Normbigg}[1]{\biggl\|#1\biggr\|}
\providecommand{\NormBig}[1]{\Bigl\|#1\Bigr\|}

\renewcommand{\:}{\colon}

\newcommand{\I}{\mathbf{i}}
\newcommand{\wnorm}[1]{\norm{#1}_{A}}
\newcommand{\wNormbig}[1]{\Normbig{#1}_{A}}

\newcommand{\transpose}[1]{#1^t}
\newcommand{\transposeparen}[1]{\bigl(#1^t\bigr)}

\newcommand{\CCC}{C}

\newcommand{\Holder}[1] {\CCC^{0,#1}}

\renewcommand{\=}{\coloneqq}

\newcommand{\wh}{\widehat}

\newcommand{\cB}{\mathcal{B}}

\newcommand{\cD}{\mathcal{D}}
\newcommand{\cE}{\mathcal{E}}
\newcommand{\cF}{\mathcal{F}}

\newcommand{\cI}{\mathcal{I}}

\newcommand{\cM}{\mathcal{M}}

\newcommand{\cO}{\mathcal{O}}
\newcommand{\cP}{\mathcal{P}}
\newcommand{\cQ}{\mathcal{Q}}

\newcommand{\cS}{\mathcal{S}}

\newcommand{\sO}{\mathscr{O}}

\newcommand{\oB}{\overline{B}}

\newcommand{\Var}{\operatorname{Var}}

\begin{document}

\title[Explicit exposure of Haar measure]
{Explicit exposure of Haar measure}

\makeatletter
% Remove \uppercase and \author@andify so \author prints its argument as-is
\renewcommand{\@setauthors}{%
  \begingroup
  \trivlist
  \centering\normalsize \@topsep30\p@\relax
  \advance\@topsep by -\baselineskip
  \item\relax
  \authors
  \vspace{14pt}
  \endtrivlist
  \endgroup
}
\makeatother

\author[Yixuan~Huang, Oliver~Jenkinson, Zhiqiang~Li, and Hang~Zhao]{%
  \begin{tabular}{c}Yixuan~Huang\\[-3pt]\Small Peking University
  \end{tabular}\;
  \begin{tabular}{c}Oliver~Jenkinson\\[-3pt]\Small Queen Mary University of London
  \end{tabular}\;
  \begin{tabular}{c}Zhiqiang~Li\\[-3pt]\Small Peking University  \end{tabular}\;
  \begin{tabular}{c}Hang~Zhao\\[-3pt]\Small Neijiang Normal University  \end{tabular}
}

\subjclass[2020]{Primary: 37A99; Secondary: 22D40, 37A05, 37A20, 37A25, 37A46, 37E10, 43A25.}

\keywords{Ergodic optimization, maximizing measure, Haar measure, Lebesgue measure, exposed point, Wiener algebra, P-recursive sequence}

\begin{abstract}
We solve the problem, posed some twenty years ago, of explicitly constructing a
continuous function whose unique maximizing measure for the doubling map is Lebesgue measure.
This is a special case of a more general result: given a nontrivial compact metrizable abelian group and a continuous endomorphism for which normalised Haar measure is ergodic, we explicitly construct a continuous function on the group for which Haar measure is the unique invariant maximizing measure.
The function is the uniform limit of a recursively defined sequence of
trigonometric polynomials with rational coefficients; every parameter of the recursion is
given by an explicit formula, every step is exact, and the rate of convergence is explicit.
For expanding circle endomorphisms, ergodic toral endomorphisms, and the one-sided full
shift, the construction can moreover be arranged so that the function is given by a Fourier series that converges uniformly in the
classical frequency order, each of whose coefficients is rational and computable exactly by
a finite procedure.

We also show that no continuous function whose unique maximizing measure is Haar measure
has absolutely summable Fourier coefficients, so that no such function lies in the Wiener
algebra. For expanding maps of the circle, the Fourier coefficients of a continuous function
with a unique maximizing measure of full support admit no closed form, in the sense that
they satisfy no linear recurrence with polynomial coefficients; this follows from a result
of independent interest, that a continuous function on the circle whose Fourier
coefficients satisfy such a recurrence is H\"older continuous.
\end{abstract}

\maketitle

\setcounter{tocdepth}{1}
\thispagestyle{empty}
\vspace{-18pt}
\tableofcontents

\section{Introduction}

\subsection{Ergodic optimization and the inverse problem}
\label{ss:intro background}

For any measure $m$ that is ergodic for a topological dynamical system $T\:X\to X$ (a continuous self-map of a compact metric space $(X,d)$), it is known (see \cite{Je06b}) that there exist continuous functions $f\:X\to\R$ such that $\int\! f\, \mathrm{d}m$ is \emph{strictly} larger than
$\int\! f\, \mathrm{d}\mu$ for all other $T$-invariant probability measures $\mu$.
In other words, every ergodic measure is the unique $f$-maximizing measure for some continuous $f$. However, the argument of \cite{Je06b} does not exhibit such an $f$
\emph{explicitly}.

If $m$ is \emph{strictly ergodic}, i.e.,~if $\supp m$ carries no invariant measure
other than $m$, then explicit examples of such continuous $f$ can be readily found, by
arranging that $f$ attains its maximum value precisely on $\supp m$, for example by
setting $f(x)\= -d(x,\supp m)$.
If $m$ is not strictly ergodic, then this construction is unavailable, and the
problem of exhibiting an explicit continuous $f$ whose unique maximizing measure is $m$
is considerably more delicate.
To make the question concrete, it was posed as Problem 3.9 in \cite{Je06a} (see also the discussion around \cite[Corollary 1]{Je06b}) in the particular case of the doubling map $T(x)=2x \pmod 1$ on the circle $X=\R/\Z$, with $m$ Lebesgue measure; in this form, the problem has been publicised in lectures and problem sessions over the past two decades\footnote{For example, the problem was discussed at the 2004 CIRM \'Ecole Pluri-th\'ematique de Th\'eorie Ergodique
that led to \cite{Je06a}, 
and at problem sessions during the 2006 BIRS workshop on Measurable Dynamics: Theory \& Applications (\url{https://www.birs.ca/workshops/2006/06w5079/report06w5079.pdf}), and the 2011 BIRS workshop on Ergodic Optimization (\url{https://www.birs.ca/workshops/2011/11w5039/report11w5039.pdf}).}.

Ergodic optimization has developed considerably since
\cite{Je06a}; for surveys see \cite{Boc18, Je19}.
Problem~3.9 was one of several problems posed in \cite{Je06a}, and it may be useful
to record the status of the others.
The question of \emph{zero temperature limits}, whether the equilibrium states of
$tf$ converge as $t\to+\infty$ \cite[footnote~16]{Je06a}, was answered negatively by
Chazottes and Hochman \cite{CH10}. The question of \emph{typical periodic
optimization}, whether a generic Lipschitz function has a periodic maximizing
measure \cite[Conjecture~4.11]{Je06a}, was answered affirmatively for expanding maps
by Contreras \cite{Co16}, and subsequently in wider settings
(uniformly hyperbolic systems in \cite{HLMXZ25}, expanding Thurston maps in \cite{LZ25},
symbolic dynamics in \cite{HJXZ26}, the $C^r$ category in \cite{GSZ25}, and 
joint optimization allowing variation of the system in \cite{HHJL25, HHJL26}). The question of
\emph{analytic positive entropy}, whether a positive entropy measure can be the
unique maximizing measure for a real analytic function \cite[Problem~3.12]{Je06a},
was answered negatively by Gao and Shen \cite{GS24}: for a real analytic expanding
circle map, every maximizing measure of a real analytic function has zero entropy,
unless the function is cohomologous to a constant (this also settles
\cite[Problem~3.11]{Je06a} negatively).
The present article concerns the remaining question of \emph{explicit exposure}:
Problem~3.9 above, together with its generalisation \cite[Problem~4.4]{Je06a},
which asks, for a transitive hyperbolic map with local product structure, for an explicit continuous function
with a unique maximizing measure of full support.

For the doubling map, any continuous $f$ whose unique maximizing measure is Lebesgue
measure $m$ is necessarily rather irregular: $f$ cannot be H\"older continuous, nor, more generally, Dini continuous
(i.e.,~admit a modulus of continuity $\omega_f$ that satisfies $\int_0^{1/2} \! \omega_f(s)\, \frac{\mathrm{d}s}{s}< +\infty$), nor 
can it belong to the
class of
Walters functions (cf.~\cite{Bou01, Walt78}).
This is because, for functions $f$ in these regularity classes, it is known to be possible to modify $f$, by adding a continuous coboundary (which has zero integral for every invariant measure), 
so that the resulting function
attains its maximum at every point of the support of every $f$-maximizing measure
(see e.g.~\cite{Bou00,Bou01,Bou11, CLT01, CG95, LT03, Sa99}).
Consequently, if a measure of full
support such as $m$ were $f$-maximizing, then every invariant measure would be
$f$-maximizing, precluding uniqueness.
This obstruction, and the results described above in the real analytic, Lipschitz,
and $C^r$ categories, are instances of a common phenomenon: regularity of the
function constrains the complexity (and in particular the entropy) of its maximizing measures.
By contrast, Theorems~\ref{thm:circle}, \ref{thm:torus}, \ref{thm:shift}, and~\ref{thm:main} below
construct, explicitly, continuous functions whose unique maximizing measure is Haar measure $m_G$ (which has full
support and maximal entropy, cf.~\cite{Bow71}) 
for every continuous endomorphism $T$ of a compact metrizable abelian group $G$ such that $m_G$ is ergodic.
In particular, the
case $d=2$ of Theorem~\ref{thm:circle} solves \cite[Problem~3.9]{Je06a},
and Theorems~\ref{thm:circle}--\ref{thm:main} answer \cite[Problem~4.4]{Je06a} for
expanding circle endomorphisms, hyperbolic toral automorphisms, and full shifts.

\subsection{Setting and main results}
\label{ss:intro main}

Throughout the article, $G$ denotes a nontrivial compact metrizable abelian group, $T\:G\to G$ a continuous endomorphism, and $m_G$ normalised Haar measure on $G$, which 
we assume to be
ergodic for $T$.
We write $\N\=\{1,\,2,\,\ldots\}$, $\N_0\=\N\cup\{0\}$, and
$\T\=\R/\Z$ for the circle group, identified with
$\{z\in\C:\abs z=1\}$ in the usual way,
and $m_\T$ for normalised Haar--Lebesgue measure on $\T$;
for $H$ a compact abelian group, $\wh H$ denotes its dual
(character) group. 
We write $\Gamma\=\wh G$, and $A\=\wh T$ for the dual endomorphism of
$\Gamma$, given by $A\gamma\=\gamma\circ T$.
Further notation and basic facts will be given in
Section~\ref{s:separation}.

For $f\in C(G,\R)$, the quantity
\begin{equation*}
Q(T,f)\=\sup\biggl\{\int \!f\,\mathrm{d}\mu:
\mu\in\cM(G,T)\biggr\}
\end{equation*}
is called the \emph{maximum ergodic average} of $f$, where
$\cM(G,T)$ denotes the space of $T$-invariant Borel probability
measures on $G$. A measure $\mu\in\cM(G,T)$ is called
\emph{$f$-maximizing} if $\int\!f\,\mathrm{d}\mu=Q(T,f)$, and the
collection of all $f$-maximizing measures is denoted by
$\cM_{\max}(f)$.

Ergodic optimization studies $\cM_{\max}(f)$ for a prescribed
function $f$. Here we consider the inverse problem: given
$\mu\in\cM(G,T)$, construct $f\in C(G,\R)$ such that
\begin{equation}\label{eq:exposed point}
\cM_{\max}(f)=\{\mu\}.\end{equation}
In this case, we say that $f$ \emph{exposes} $\mu$, or that $f$ is an \emph{exposing function} for $\mu$, since \eqref{eq:exposed point} asserts that $\mu$ is an exposed point
(cf.~\cite[p.~119]{Al71})
of the convex set $\cM(G,T)$,
in view of the fact that each $f\in C(G,\R)$ induces a
continuous linear functional $\nu\mapsto\int\! f\,\mathrm d\nu$ on the space of
signed Borel measures on $G$ endowed with the weak$^*$ topology.

We first record three classical settings, namely expanding circle
endomorphisms, ergodic toral endomorphisms, and the one-sided full shift, in
which the exposing function can be given by a Fourier (or Walsh)
expansion, uniformly convergent in the classical frequency order,
with every coefficient rational and determined exactly by a finite
procedure.
These are special cases of the general recursive
construction of Theorem~\ref{thm:main} below, whose defining sequence
can, after relocating frequencies, be reassembled in this way;
see Section~\ref{s:Fourier series}. 
In fact, for every $G$ and every enumeration of $\wh G\smallsetminus\{1\}$, the same
procedure gives an exposing function whose Fourier series converges uniformly when summed
in that order (Proposition~\ref{p:abstract uniform fourier}); by
Theorem~\ref{thm:wiener obstruction} below, it never converges absolutely.
We begin with 
expanding circle endomorphisms, 
the doubling map of \cite[Problem~3.9]{Je06a}
being the case $d=2$.

\begin{thmx}[Fourier exposure for expanding circle endomorphisms]
\label{thm:circle}
For an integer $d$ with $\abs{d}\ge 2$, let $T\colon\T\to\T$ be given by
$T(x)=d\, x\pmod 1$. Then there exists a function
$\widetilde F\in C(\T,\R)$ of the form
\begin{equation*}
\widetilde F(x)=\sum_{k=1}^{+\infty}c_k\cos(2\pi kx),
\quad c_k\in\Q,
\end{equation*}
such that the Fourier partial sums converge uniformly to $\widetilde F$ on $\T$, and Lebesgue measure is the unique
$T$-invariant maximizing measure for $\widetilde F$.
Moreover, each coefficient $c_k$ is determined exactly by a finite
procedure.
\end{thmx}

The second setting consists of the toral endomorphisms
$T_B(x)=Bx\pmod{\Z^m}$ given by an $m\times m$ integer matrix $B$ with $\det B\neq0$ for
which Lebesgue measure is ergodic, i.e., for which no eigenvalue of $B$ is a root of unity
(cf.~\cite{Hal43, Ro49, SR88}). For these, Theorem~\ref{thm:torus} is a
higher-dimensional analogue of Theorem~\ref{thm:circle}.

For Fourier series on $\T^m\=(\R/\Z)^m$, we use the standard square partial sums; see \cite[Definition~3.2.3, p.~184]{Gr14}. Namely, for
$F(x)\=\sum_{k\in\Z^m}c_k\exp(2\pi\I\langle k,x\rangle)$ and $R\in \N_0$, the $R$-th square partial sum is
\begin{equation}
\label{eq:T^d partial sum}
\Sigma_RF(x)\=\sum_{\norm{k}_\infty\le R}
c_k\exp(2\pi\I\langle k,x\rangle).
\end{equation}

\begin{thmx}[Fourier exposure for ergodic toral endomorphisms]
\label{thm:torus}
For an integer $m\ge2$, suppose that $B$
is an $m\times m$ integer matrix with
$\det B\neq0$ such that
the toral endomorphism $T_B\colon\T^m\to\T^m$ given by $T_B(x)=Bx\pmod{\Z^m}$ is ergodic with respect to Lebesgue measure.
Then there exists a function $\widetilde F\in C(\T^m,\R)$ of the form
\begin{equation*}
\widetilde F(x)=
\sum_{k\in\Z^m\smallsetminus\{0\}}
c_k\exp(2\pi\I\langle k,x\rangle),
\quad c_k\in\Q,
\end{equation*}
such that the square partial sums of the Fourier series converge
uniformly to $\widetilde F$ on $\T^m$, and Lebesgue measure is the unique $T_B$-invariant maximizing measure for $\widetilde F$.
Moreover, each coefficient $c_k$ is determined exactly by a finite
procedure.
\end{thmx}

The construction applies equally to symbolic dynamics, for any
alphabet $\Z/q\Z$ with an integer $q\ge2$ and for either one-sided or two-sided shifts;
we present here
the case $X\=(\Z/2\Z)^{\N_0}$, i.e.,~the one-sided full shift on two symbols. 
We use the standard Walsh ordering of
$\wh X$; see \cite[p.~99]{Fo16}. For $k\in\N_0$, write
\begin{equation}
\label{eq:walsh ordering}
k=\sum_{j=0}^{+\infty}\varepsilon_j(k)2^j,\quad
\varepsilon_j(k)\in\{0,\,1\}.
\end{equation}
For $x=(x_j)_{j\in\N_0}\in X$, the $k$-th Walsh function is defined by $w_k(x)\=
(-1)^{\sum_{j=0}^{+\infty}\varepsilon_j(k)x_j}$.
Then $\wh X=\{w_k:k\in\N_0\}$ and $w_0=1$.
The functions $w_k$ originate in Walsh's system of orthonormal
functions on $[0,1]$ \cite{Wals23}, introduced as a uniformly
bounded alternative to the earlier orthonormal system of
Haar~\cite{Haa10}\footnote{Note that Haar measure itself, the subject of this article, was introduced some two decades later, in \cite{Haa33}.};
their identification with the characters of $X$, in the form used here, goes back to Fine~\cite{Fi49}.

\begin{thmx}[Fourier exposure for the one-sided full shift]
\label{thm:shift}
Let $T\colon X\to X$ be the one-sided
shift on $X\=(\Z/2\Z)^{\N_0}$. Then there exists a function $\widetilde F\in C(X,\R)$ of the form
\begin{equation*}
\widetilde F(x)=\sum_{k=1}^{+\infty}c_k w_k(x),
\quad c_k\in\Q,
\end{equation*}
such that the partial sums of the Walsh--Fourier series converge uniformly to $\widetilde F$ on $X$, and the $(1/2,1/2)$-Bernoulli measure is the unique $T$-invariant maximizing measure for $\widetilde F$. Moreover, each coefficient $c_k$ is determined exactly by a finite procedure.
\end{thmx} 

\begin{rem}
\label{rem:exact coefficients}
The coefficients in Theorems~\ref{thm:circle}, \ref{thm:torus}, and~\ref{thm:shift}
have a property different from, and in a precise sense stronger
than, the computability of $\widetilde F$. To say that $c_k$ is
\emph{determined exactly by a finite procedure} is to say that there is
an algorithm which, given $k$, returns the \emph{exact} value of $c_k$
after finitely many 
operations in the sense of Remark~\ref{rem: computable} below, rather than an approximation
to it. Such a statement is meaningful only because each $c_k$ is
rational; no analogue
is to be expected for the values
$\widetilde F(x)$, which are
limits of infinite series and can in general only be approximated.
Exactness of the $c_k$,
together with the uniform convergence of the partial sums in the
prescribed order, implies in particular that $\widetilde F$ is
computable in the sense of Remark~\ref{rem: computable}
below: the series is summed directly, with every term exact and with an
explicit bound on the truncation error at each stage (see the proof of
Proposition~\ref{p:abstract uniform fourier}).
\end{rem}

The three theorems above are in fact special cases, obtained by relocating frequencies, of
a single construction that works for an arbitrary compact metrizable abelian group:

\begin{thmx}[Uniformly convergent trigonometric exposure of Haar measure]
\label{thm:main}
Let $G$ be a nontrivial compact metrizable abelian group and
$T\:G\to G$ a continuous endomorphism for which the
normalised Haar measure $m_G$ is ergodic.
There is a sequence $\{F_n\}_{n\in\N_0}$ of real-valued trigonometric polynomials on $G$
with rational Fourier coefficients, produced by an explicit recursion whose parameters are given by explicit formulas depending only on
the step index $n$, such that $F_n$ converges, as $n\to +\infty$,
uniformly to a function $F\in C(G,\R)$ for which $m_G$ is the
unique $T$-invariant maximizing measure. More precisely, 
\begin{equation*}\int\!F\,\mathrm{d}m_G=0
\quad\text{ and }\quad
\int\!F\,\mathrm{d}\nu<0
\quad\text{ for all }
\nu\in\cM(G,T)\smallsetminus\{m_G\}.\end{equation*}
\end{thmx}

The parameters of the recursion are the same for every $G$ and $T$; only the characters
from which the trigonometric polynomials $F_n$ are built depend on $G$ and $T$. The
construction is also quantitative: Proposition~\ref{p:exposure bound} gives an explicit
lower bound for $\int\!F\,\mathrm{d}m_G-\int\!F\,\mathrm{d}\nu$ in terms of the separation
function $D(\nu)$ of Definition~\ref{def:separation function D}.

\begin{rem}\label{rem: computable}

We use \emph{computable} in the sense of computable analysis
(see e.g.~\cite{PER89}). Informally, a real number is computable if
an algorithm can produce a rational approximation to it within any
prescribed error, and a function is computable if a single algorithm
does this uniformly in the input and the requested accuracy, using
sufficiently accurate approximations of the input.

For an abstract compact metrizable abelian group, computability is not
an intrinsic notion: it can only be discussed after fixing a concrete
effective representation of the group and the objects involved.
Statements about computability should therefore be understood relative to a fixed
presentation of $G$, its dual group $\Gamma$, and the dual endomorphism
$A$.

Relative to such a presentation, we say that a quantity is produced by finitely many
\emph{operations} if it is obtained from the input by finitely many applications of the
following, each of which is exact:
\begin{enumerate}[label=\rm{(O\arabic*)}]
\smallskip
\item\label{op:arith} the field operations of $\Q$, and the ceiling of a rational;
\smallskip
\item\label{op:char} multiplication and inversion in $\Gamma$, application of $A=\wh T$
to a character, and the test of equality of two characters;
\smallskip
\item\label{op:enum} passage from $r\in\N$ to the character $\gamma_r$ of the enumeration
\eqref{eq:enumeration};
\smallskip
\item\label{op:poly} formation of $P+Q$, $PQ$, $P\circ T^j$, and $R\circ P$, for $P$,
$Q\in\cF_\Q$ (the real-valued trigonometric polynomials with rational Fourier
coefficients; see Section~\ref{s:separation}), $j\in\N_0$, and $R\in\Q[t]$, carried out on the finite lists of coefficients
representing the arguments (Lemma~\ref{lem:rational fourier closure});
\smallskip
\item\label{op:haar} Haar integration of an element of $\cF_\Q$, which returns the
coefficient of the trivial character.
\end{enumerate}
The recursion of Subsection~\ref{ss:inductive construction} uses finitely many of these
operations at each step.

Consequently, given $N\in \N$, an algorithm produces the complete finite Fourier expansion of $F_N$, and $F_N$ can be evaluated
at any point to any prescribed accuracy by a finite computation. Moreover, Proposition~\ref{p:increment estimate} gives the explicit bound $\norm{F-F_N}_\infty<\frac{16}{105}\delta_N$. Since $\delta_N$ is explicitly defined, this bound gives a computable rate of convergence: from a target accuracy $\varepsilon>0$, one can compute an $N$ such that $\norm{F-F_N}_\infty<\varepsilon/2$, and then evaluate $F_N$ to within $\varepsilon/2$. It follows that $F$ is computable in the above sense; see e.g.~\cite{PER89}.
\end{rem}

A function $f\in C(G,\R)$ is called a \emph{coboundary} if $f=u-u\circ T$ for some $u\in C(G,\R)$; see e.g.
\cite[Section~1]{Boc18}.
Every coboundary has zero integral with respect to every measure in $\cM(G,T)$. In general, the space of coboundaries need not be closed in $C(G,\R)$. Following
\cite[Section~3]{BJ02}, an element of its uniform closure is called a \emph{weak coboundary}. Equivalently,
$f\in C(G,\R)$ is a weak coboundary if and only if
\begin{equation*}
\int\!f\,\mathrm d\mu=0
\quad\text{for all }\mu\in\cM(G,T);
\end{equation*}
see \cite[Lemma~3]{BJ02}. Thus every coboundary is a weak coboundary, although the converse need not hold.

Recall that a \emph{character} of $G$ is a continuous group homomorphism $\gamma\colon G\to\T$, and that the characters form the dual group $\Gamma=\wh G$
under pointwise multiplication;
see \cite[Section~4.1]{Fo16}, and Section~\ref{s:separation} for
further notation. 
For $f\in C(G,\C)$ and $\gamma\in \wh G$, its Fourier coefficient at $\gamma$ is $\wh f(\gamma)=\int\!f \, \overline{\gamma}\,\mathrm dm_G$. We denote by $A(G)$ the \emph{Wiener algebra} (or \emph{Fourier algebra}, cf.~\cite[p.~132]{Fo16}) of $G$, defined as
\begin{equation*}
A(G)\=\Bigl\{f\in C(G,\C):
\sum_{\gamma\in \wh G}
\Absbig{\wh f(\gamma)}<+\infty\Bigr\}.
\end{equation*}

\begin{thmx}[Exposing functions cannot lie in the Wiener algebra]
\label{thm:wiener obstruction}
Let $G$ be a nontrivial compact metrizable abelian group and
$T\:G\to G$ a continuous endomorphism for which the
normalised Haar measure $m_G$ is ergodic.
Suppose $F \in C(G,\R)$ belongs to the Wiener algebra $A(G)$.
If $m_G$ is a $T$-invariant maximizing measure for $F$, then $F-Q(T,F)$ is a weak coboundary,
and therefore every measure in $\cM(G,T)$
is $F$-maximizing.
\end{thmx}

Since $G$ is nontrivial, $\cM(G,T)$ is not a singleton (it contains
$\delta_{e_G}\ne m_G$, as $T(e_G)=e_G$),
so Theorem~\ref{thm:wiener obstruction} shows that
no continuous function whose unique maximizing measure is $m_G$ lies in $A(G)$. In
particular, the Fourier series in Theorems~\ref{thm:circle}, \ref{thm:torus},
and~\ref{thm:shift} converge only conditionally:
since $\gamma(e_G)=1$ for every character
$\gamma$, uniform convergence in every order of summation would give absolute convergence
at $e_G$, and hence $\widetilde F\in A(G)$.
Theorem~\ref{thm:wiener obstruction} is not a
consequence of the regularity obstruction described in
Subsection~\ref{ss:intro background}: it requires no expansion of $T$, and, for the
doubling map, the Wiener algebra contains functions which are not Walters functions. Its
conclusion is sharp, in that $F-Q(T,F)$ need not be a coboundary
(Proposition~\ref{p:wiener not walters}).

Theorem~\ref{thm:circle} produces a uniformly convergent Fourier series with rational coefficients, each of which can be determined exactly by a finite procedure. It is natural to ask whether the coefficient sequence admits a closed form. 
We show that it does not: the
sequence is not P-recursive, that is, it satisfies no nontrivial linear recurrence with polynomial
coefficients; see Section~\ref{s:P-recursive Fourier coefficients} for the definition, and for the sense in which this excludes a closed form.
This is a consequence of the following
result, whose first statement concerns continuous functions on $\T$ alone
and may be of independent interest. 
In its second statement, the self-map $T\:\T\to\T$ need not be a group endomorphism.

\begin{thmx}[Functions with P-recursive Fourier coefficients]
\label{thm:P-recursive coboundary}
Let $F\in C(\T,\R)$, and suppose that $\bigl\{\wh F(n)\bigr\}_{n\in\N_0}$ is P-recursive, i.e., satisfies a nontrivial linear recurrence with polynomial coefficients.
Then $F$ is H\"older continuous, and real analytic outside a finite subset of $\T$.
Moreover, if $T\in C^1(\T,\T)$ satisfies $\abs{T'}>1$ on $\T$, and some $F$-maximizing
measure for $T$ has full support, then $F-Q(T,F)$ is a coboundary.
\end{thmx}

The second statement in Theorem~\ref{thm:P-recursive coboundary}
follows from the first together with the Ma\~n\'e lemma, by the
argument recalled in Subsection~\ref{ss:intro background}. For $T(x)=d\,x\pmod 1$ with
$\abs d\ge2$, as in Theorem~\ref{thm:circle}, Lebesgue measure is $\widetilde F$-maximizing
and has full support. The coefficients $c_k$ of Theorem~\ref{thm:circle} are therefore not
P-recursive: otherwise $\widetilde F-Q\bigl(T,\widetilde F\bigr)$ would be a coboundary, and
every invariant measure would be $\widetilde F$-maximizing, contradicting
Theorem~\ref{thm:circle}. Equivalently, by Lemma~\ref{lem:P-recursive Fourier ODE} and
Remark~\ref{rem:ODE converse}, the function $\widetilde F$ satisfies no nontrivial linear
differential equation with trigonometric-polynomial coefficients.
More generally, if
$T\in C^1(\T,\T)$ satisfies $\abs{T'}>1$, then no continuous function whose unique
$T$-invariant maximizing measure has full support has P-recursive Fourier coefficients,
since such a $T$ has more than one invariant probability measure.

\subsection{Proof strategy}
\label{ss:intro strategy}

The set $\cM(G,T)\smallsetminus\{m_G\}$ is exhausted by the sets
$\{\nu\in\cM(G,T):D(\nu)\ge\rho\}$, $\rho\downarrow0$, where $D$ is a weak$^*$ continuous
gauge on $\cM(G,T)$, built from the Fourier coefficients of $\nu$, vanishing only at $m_G$
(Definition~\ref{def:separation function D} and Lemma~\ref{lem:D gauge}). Each of these sets
is weak$^*$ compact, so on each of them separation from $m_G$ can be witnessed by a single trigonometric polynomial; the difficulty is to add together infinitely many such
polynomials, at all scales $\rho\downarrow0$ simultaneously, without the later ones
destroying the gaps created by the earlier ones.

The construction alternates two procedures to overcome this. First, from finite
character-orbit averages $B_{r,N}$ we build single-frequency detectors $K_{r,N}$
(Definition~\ref{d:fourier detector}), and combine finitely many of these, following a
truncation of $D$, into a trigonometric polynomial $\cS_\rho$ which has zero Haar mean, is
everywhere at most $\rho/4$, and has integral below $-\rho/4$ against every invariant
measure $\nu$ with $D(\nu)\ge\rho$ (Proposition~\ref{p:separation}). Second, we flatten.
A sufficiently long Birkhoff average of a trigonometric polynomial with zero Haar mean 
is cohomologous to the polynomial itself
and is small in $L^2(m_G)$.
Subtracting a Bernstein polynomial approximation of the positive part of this average, and 
adding back the Haar mean of that approximation,
therefore lowers
the maximum ergodic average below a prescribed positive threshold,
while preserving the Haar mean and the rationality of the  Fourier coefficients
(Proposition~\ref{p:polynomial-flatten}).
The
parameters in Section~\ref{s:induction} 
are chosen so that this does not destroy the
negative gaps created at earlier steps.
  Alternating these two procedures at the scales
$\rho_n=2^{-n}$, with an explicit schedule of amplitudes, produces the sequence $\{F_n\}_{n\in \N_0}$
and proves Theorem~\ref{thm:main} (Section~\ref{s:induction}).

To obtain the Fourier series of Theorems~\ref{thm:circle}, \ref{thm:torus},
and~\ref{thm:shift}, 
Section~\ref{s:Fourier series} relocates the frequencies of the increments $F_n-F_{n-1}$: each increment is split into many relocated copies, so that no
truncation of the resulting series can remove more than a prescribed amount from a single
block of frequencies.

For Theorem~\ref{thm:wiener obstruction}, the hypothesis $F\in A(G)$ allows
$\int\!F\,\mathrm d\nu$, for $\nu\in\cM(G,T)$, to be written as a sum over the orbits of the
dual endomorphism $A$, each orbit contributing one Fourier coefficient of $\nu$ multiplied
by the sum of the Fourier coefficients of $F$ over that orbit. If the sum over some
nontrivial orbit is nonzero, we perturb $m_G$ within $\cM(G,T)$ so that, to first order,
only the Fourier coefficients on that orbit and its inverse change, and the $F$-integral
increases. The perturbation is obtained by averaging under $T$ a Riesz product built on a
dissociate subsequence of the orbit, and the first-order terms survive the averaging
because the subsequence has positive density in the orbit: a theorem of Lind~\cite{Li82}
allows the subsequence to be taken along an arithmetic progression of iterates of $A$.

For Theorem~\ref{thm:P-recursive coboundary}, P-recursiveness of the Fourier coefficients
translates into a distributional linear differential equation for $F$ with
trigonometric-polynomial coefficients, so $F$ is real analytic away from the finitely many
zeros of the leading coefficient. Near such a zero $x_0$, which may be an irregular singular
point, classical asymptotic theory \cite{Bi09,Fa85,Was65} describes $F$ by asymptotic
expansions built from terms $e^{Q(r)}r^\nu(\log r)^k$, where $r\=\abs{x-x_0}$ and $Q$ is a
polynomial in a fractional power of $1/r$. The main step is to show that, in the expansion
of $F(x)-F(x_0)$, continuity of $F$ excludes every term that does not tend to $0$. This is
not automatic, since such terms may oscillate at different rates and might cancel one
another; a mean-square argument shows that they cannot. The surviving terms are
$\cO\bigl(r^\beta\bigr)$ for some $\beta>0$, and together with a bound on $F'$ by a power of $1/r$,
this gives H\"older continuity by interpolation.

\subsection{Organisation}

Section~\ref{s:separation} recalls the relevant Fourier-analytic preliminaries, introduces
the separation function $D$ of Definition~\ref{def:separation function D}, and constructs
from finite character-orbit averages the detectors $K_{r,N}$ and $\cS_\rho$ of
Proposition~\ref{p:separation}. Section~\ref{sec:averaging} develops the flattening
procedure of Proposition~\ref{p:polynomial-flatten}. Section~\ref{s:induction} combines the
two into the recursive construction of Subsection~\ref{ss:inductive construction}, proving
Theorem~\ref{thm:main}. Section~\ref{s:Fourier series} relocates frequencies, carrying this
out in Subsection~\ref{ss:effective Fourier construction} explicitly for expanding circle
endomorphisms, the one-sided full shift, and ergodic toral endomorphisms, thereby proving
Theorems~\ref{thm:circle}--\ref{thm:shift}.
The final two sections are independent of the construction, and of each other, and
demonstrate the extent to which the irregularity of the functions it produces is
unavoidable. 
Section~\ref{s:Wiener} proves Theorem~\ref{thm:wiener obstruction}, and gives the example of Proposition~\ref{p:wiener not walters}.
Section~\ref{s:P-recursive Fourier coefficients} proves
Theorem~\ref{thm:P-recursive coboundary}.

\medskip
\noindent{\bf Acknowledgments.} 
Y.~H.\ and Z.~L.\ were partially supported by Beijing Natural Science Foundation (JQ25001) and National Natural Science Foundation of China (12471083). H.~Z.\ was partially supported by National Natural Science Foundation of China (123B2006).

\section{Detectors and quantitative separation}
\label{s:separation}

Fourier coefficients separate invariant measures from Haar measure $m_G$.
We first construct a detector for a single Fourier coefficient
in Lemma~\ref{lem:single detector} and then combine finitely many such
detectors to obtain a quantitative gap at a prescribed scale in
Proposition~\ref{p:separation}.

We begin by recalling some standard notation and facts from Fourier analysis on compact abelian groups; see \cite{Fo16}.

\begin{notation}
We identify $\T=\R/\Z$ with the circle group $\{z\in \C:\abs{z}=1\}$ in the usual way.
A \emph{character} of $G$ is a continuous group
homomorphism $\gamma\:G\to\T$. We write $\Gamma\=\wh G$ for the group of all characters of $G$, with pointwise multiplication, so that for all $\gamma$, $\eta\in \Gamma$ and $x\in G$,
\begin{equation*}
(\gamma\eta)(x)=\gamma(x)\eta(x)
\quad\text{ and }\quad
\gamma^{-1}(x)=\overline{\gamma(x)}.
\end{equation*}
Let $e_G$ denote the identity element of $G$, and let $1\in\Gamma$ denote the trivial character, given by $1(x)=1$ for all $x\in G$.
Then $\gamma(e_G)=1$ for all $\gamma\in\Gamma$. 
Recall that $A=\wh T$ is given by $A\gamma\=\gamma\circ T$.
Since $m_G$ is $T$-invariant and has full support,
$T$ is surjective, and hence $A$ is
injective. For a Borel probability measure $\nu$ on $G$ and
$\gamma\in\Gamma$, define the \emph{Fourier coefficient} of $\nu$ at $\gamma$ by $\wh\nu(\gamma)\=\int \! \overline\gamma\,\mathrm{d}\nu$.
If $\nu\in\cM(G,T)$, then
\begin{equation}\label{eq:fourier invariance}
 \wh\nu(A\gamma)=\wh\nu(\gamma)
 \quad\text{ for all }\gamma\in\Gamma.
\end{equation}
\end{notation}

The orthogonality of characters with respect to Haar measure is standard;  see \cite[Proposition~4.4]{Fo16}. In particular, for every nontrivial $\gamma\in\Gamma$,
\begin{equation}\label{eq:orthogonality}
\int\!\overline\gamma\,\mathrm{d}m_G=0.
\end{equation}

We require (one direction of) the familiar criterion that $m_G$ is $T$-ergodic if
and only if $A$ has no nontrivial finite orbit on $\Gamma$ (cf.~\cite{Hal43, Ro49, SR88}; we include a proof for completeness).

\begin{lemma}[Distinct characters along dual orbits]
\label{lem:distinct character}
The dual endomorphism $A=\wh T$ is injective. Moreover, for every
nontrivial $\gamma\in\Gamma$, the characters
$\gamma$, $A\gamma$, $A^2\gamma$, $\ldots$ are pairwise distinct.
\end{lemma}

\begin{proof}
The injectivity of $A$ follows from the surjectivity of $T$, as noted above.
Suppose that $A^j\gamma=A^k\gamma$ for some $0\le j<k$. 
Since $A$, and hence $A^j$, is injective, 
$A^j\bigl(A^{k-j}\gamma\bigr) = A^j\gamma$ forces 
$A^{k-j}\gamma = \gamma$, so $A^p\gamma = \gamma$ for some least $p \ge 1$.
Injectivity also gives
$A^\ell\gamma\neq1$ for every $\ell\ge0$, since $A^\ell 1=1$ and $\gamma\neq1$.
The function
$H\=\sum_{\ell=0}^{p-1}A^\ell\gamma$ 
satisfies $H\circ T=H$,
so by ergodicity $H$ is constant $m_G$-almost everywhere. Since $m_G$ has full support and $H$ is continuous, $H$ is a constant function.
By \eqref{eq:orthogonality}, $\int\!H\,\mathrm{d}m_G=0$, so $H\equiv0$,
contradicting the fact that $H(e_G)=p>0$.
\end{proof}

Since $G$ is compact metrizable, $L^2(G,m_G)$ is separable.
Moreover, $\wh G$ is an orthonormal basis of $L^2(G,m_G)$ by
\cite[Corollary~4.27]{Fo16}; hence $\wh G$ is countable. Since $G$ is nontrivial, $\Gamma\neq\{1\}$, and
Lemma~\ref{lem:distinct character} then shows that $\Gamma$ is infinite. Fix an
enumeration
\begin{equation}\label{eq:enumeration}
\Gamma\smallsetminus\{1\} =
\{\gamma_r:r\in\N\},
\end{equation}
in which each nontrivial character occurs exactly once, and set $\gamma_0\=1$.

\begin{definition}
\label{def:separation function D}
Define the \emph{separation function}
$D\:\cM(G,T)\to[0,1]$ by
$D(\nu)\=
\sum_{r=1}^{+\infty}
2^{-r}\abs{\wh\nu(\gamma_r)}^2$,
for $\nu\in\cM(G,T)$.
\end{definition}

We record the two properties of $D$ that will be used.

\begin{lemma}[Quadratic Fourier separation]
\label{lem:D gauge}
The separation function $D$ is weak$^*$ continuous.
Moreover, for $\nu\in\cM(G,T)$,
$D(\nu)=0$ if and only if $\nu=m_G$.
In particular,
$D(\nu)>0$ for every
$\nu\in\cM(G,T)\smallsetminus\{m_G\}$.
\end{lemma}

\begin{proof}
For each $r\in\N$, the map
$\nu\mapsto\wh\nu(\gamma_r)$ is weak$^*$ continuous. Since $2^{-r}\abs{\wh\nu(\gamma_r)}^2\le2^{-r}$
for every $r\in\N$ and every probability measure $\nu$, the series
defining $D$ converges uniformly. Hence $D$ is weak$^*$ continuous. If $D(\nu)=0$, then every nontrivial Fourier coefficient of $\nu$ vanishes, hence $\nu$ and $m_G$ have the same Fourier coefficients. By uniqueness
of Fourier coefficients for finite Borel measures on compact abelian
groups (cf.~\cite[Corollary~4.34]{Fo16}), $\nu=m_G$.
Conversely, $D(m_G)=0$, since $\wh{m_G}(\gamma_r)=0$ for every $r\in\N$
by \eqref{eq:orthogonality}.
\end{proof}

The function $D\:\cM(G,T)\to[0,1]$ 
measures how far an invariant measure is from $m_G$, but it is not itself 
of the form $\nu\mapsto\int\!f\, \mathrm{d}\nu$ with $f$ continuous. We now convert it into a function of this form, beginning with the detector attached to a single
Fourier coefficient.
We first fix some notation.

\begin{notation}
A \emph{trigonometric polynomial} on $G$ is a function of the form
$P=\sum_{\gamma\in E}a_\gamma\gamma$, where $E\subseteq\Gamma$ is finite and
$a_\gamma\in\C$. The set of these trigonometric polynomials is denoted by $\cP$. For $f\in C(G,\C)$ and $\gamma \in \Gamma$, we define the \emph{Fourier coefficient} of $f$ at $\gamma$ by $\wh f(\gamma)\=\int\!f\overline{\gamma}\,\mathrm{d}m_G$, so that
$\wh P(\gamma)=a_\gamma$ for $\gamma\in E$ and $\wh P(\gamma)=0$ otherwise, and
we set $\supp\wh P\=\bigl\{\gamma\in\Gamma:\wh P(\gamma)\neq0\bigr\}$. 
Recall that the
\emph{Wiener algebra} $A(G)$ consists of those $f\in C(G,\C)$ whose Fourier coefficients
$\bigl(\wh f(\gamma)\bigr)_{\gamma\in\Gamma}$ lie in $\ell^1(\Gamma)$, equipped with the norm
\begin{equation}\label{eq:wiener norm}
\wnorm{f}\=\sum_{\gamma\in\Gamma}\Absbig{\wh f(\gamma)};
\end{equation}
it is a commutative Banach algebra under pointwise multiplication; moreover
$\norm{f}_\infty\le\norm{f}_A$ and $\norm{f_1f_2}_A\le\norm{f_1}_A\norm{f_2}_A$ for
$f$, $f_1$, $f_2\in A(G)$. 
Every trigonometric polynomial lies in $A(G)$. 

Let $\cF_\Q=\cF_{\Q}(G)$ denote the set of inversion-invariant trigonometric
polynomials with rational Fourier coefficients; equivalently,
\begin{equation*}
\cF_{\Q}(G)=\bigl\{P\in\cP:
\wh P(\gamma)\in\Q
\text{ and }
\wh P\bigl(\gamma^{-1}\bigr)=\wh P(\gamma)
\text{ for each }\gamma\in \Gamma\bigr\},
\end{equation*}
that is, the real-valued trigonometric polynomials all of whose Fourier
coefficients are rational.
With respect to the fixed enumeration
$\Gamma\smallsetminus\{1\}=\{\gamma_r:r\in\N\}$ together with $\gamma_0\=1$,
we may equivalently write $P=\sum_{r=0}^{+\infty}a_r\gamma_r$, where only
finitely many $a_r$ are nonzero, and hence
$\wnorm{P}=\sum_{r=0}^{+\infty}\abs{a_r}$.
\end{notation}

The set $\cF_\Q$ is closed under the operations used in the construction:

\begin{lemma}
\label{lem:rational fourier closure}
Suppose $P$, $Q\in\cF_{\Q}(G)$, $j\in\N_0$, and $R\in\Q[t]$. Then $P+Q$, $PQ$, $P\circ T^j$, and $R\circ P$ all belong to $\cF_{\Q}(G)$.
Moreover, 
\begin{equation*}
\supp\wh{P\circ T^j}=A^j\bigl(\supp\wh P\bigr)
\quad\text{ and }\quad
\wNormbig{P\circ T^j}=\wnorm{P}.
\end{equation*}
\end{lemma}

\begin{proof}
It follows directly from the definition, since characters are closed under multiplication and inversion. The support identity follows from $\gamma\circ T^j=A^j\gamma$ for all $\gamma\in \Gamma$.
Since $A$ is injective by Lemma~\ref{lem:distinct character}, this identity shows moreover that $\gamma\mapsto\gamma\circ T^j$ is a bijection from $\supp\wh P$ onto $\supp\wh{P\circ T^j}$, so $\wNormbig{P\circ T^j}=\sum_{\gamma\in\supp\wh P}\Absbig{\wh P(\gamma)}=\wnorm P$.
\end{proof}

We now define the detector attached to the Fourier coefficient at $\gamma_r$, using averages of $\gamma_r$ along its orbit under $A$.

\begin{definition}
\label{d:fourier detector}
    Suppose $r$, $N\in\N$. For each $x\in G$, define
    $B_{r,N}(x)\=\frac1N \sum_{j=0}^{N-1}\overline{A^j\gamma_r(x)}$, and define the real-valued function
    \begin{equation*}
 K_{r,N}(x)\=\frac1N-\abs{B_{r,N}(x)}^2 = -\frac1{N^2}
\sum_{0\le\ell<j<N}
\bigl(\bigl(A^j\gamma_r \bigr)(x)\bigl(A^\ell\gamma_r \bigr)^{-1}(x) +\bigl(A^\ell\gamma_r \bigr)(x)\bigl(A^j\gamma_r \bigr)^{-1}(x)\bigr).
\end{equation*}
In particular, $K_{r,N}\in\cF_{\Q}(G)$.
\end{definition}

The detector $K_{r,N}$ has the following properties:

\begin{lemma}[A real-valued Fourier detector]\label{lem:single detector}
Suppose $r$, $N\in\N$. 
Then the following hold:
\begin{enumerate}[label=\rm{(\roman*)}]
\smallskip
\item $\norm{K_{r,N}}_{\infty}\le 1$ and $K_{r,N}\le \frac1N$.
\smallskip
\item $\int \!K_{r,N}\,\mathrm{d}m_G=0$ and $\int \!K_{r,N}\,\mathrm{d}\nu\le \frac1N-\abs{\wh\nu(\gamma_r)}^2$ for all $\nu\in\cM(G,T)$.
\smallskip
\item $\wnorm{K_{r,N}}<1$.
\end{enumerate}
\end{lemma}

\begin{proof}
(i) For each character $\gamma\in\Gamma$, $\norm{\gamma}_\infty=1$, thus $\norm{B_{r,N}}_\infty \le1$. Hence,
$\frac1N-1 \leq K_{r,N}(x) \leq \frac1N$ for all $x\in G$. This proves (i).

\smallskip
(ii) By Lemma~\ref{lem:distinct character}, the characters $A^j\gamma_r$ are distinct for $0\le j<N$. Thus their orthogonality gives
$\int \! \abs{B_{r,N}}^2\,\mathrm{d}m_G=\frac1N$,
so $K_{r,N}$ has zero Haar mean. 
By \eqref{eq:fourier invariance},
$\int\! B_{r,N}\,\mathrm{d}\nu=
\frac1N\sum_{j=0}^{N-1}\wh\nu\bigl(A^j\gamma_r\bigr)
=\wh\nu(\gamma_r)$.
Thus, by the Cauchy--Schwarz inequality,
$\int \! \abs{B_{r,N}}^2\,\mathrm{d}\nu
\ge\Absbig{\int \!B_{r,N}\,\mathrm{d}\nu}^2=\abs{\wh\nu(\gamma_r)}^2$,
which gives $\int\!K_{r,N}\,\mathrm{d}\nu\le \frac1N-\abs{\wh\nu(\gamma_r)}^2$.

\smallskip
(iii) By Definition~\ref{d:fourier detector}, $\wnorm{K_{r,N}}\le (N-1)/N<1$.
\end{proof}

At each scale $\rho$, the following function $\cS_\rho$ is designed to separate $m_G$
from all measures $\nu$ satisfying $D(\nu)\ge\rho$.

\begin{notation}
For $0<\rho\le1$, set
\begin{equation}\label{eq:defi of R N}
R(\rho)\=\lceil\log_2 (2/\rho)\rceil,\quad
N(\rho)\=\lceil 4/\rho\rceil,
\end{equation}
and define
\begin{equation}\label{eq:defi of S_rho}
 \cS_\rho\=
 \sum_{r=1}^{R(\rho)}
 2^{-r}K_{r,N(\rho)}.
\end{equation}
\end{notation}

It follows directly from Lemma~\ref{lem:rational fourier closure} that $\cS_\rho\in\cF_\Q$. 

\begin{prop}[Inversion-invariant trigonometric polynomial separation]
\label{p:separation}
Suppose $\rho\in(0,1]$. Then the following hold:
\begin{enumerate}[label=\rm{(\roman*)}]
    \smallskip
    \item $\int\!\cS_\rho\,\mathrm{d}m_G=0$, $\norm{\cS_\rho}_{\infty}<1$, and $\wnorm{\cS_\rho}<1$. 
    \smallskip
    \item For each $x\in G$, $\cS_\rho(x)<\rho/4$.
    \smallskip
    \item For each $\nu\in\cM(G,T)$ with $D(\nu)\ge\rho$, we have $\int\! \cS_\rho\,\mathrm{d}\nu<-\rho/4$.
\end{enumerate}
\end{prop}

\begin{proof}
(i) By Lemma~\ref{lem:single detector}, for all $r$, $N\in \N$, $\norm{K_{r,N}}_\infty\le1$, $\int\! K_{r,N}\,\mathrm{d}m_G=0$, and $\wnorm{K_{r,N}}<1$. Thus $\int\! \cS_\rho \,\mathrm{d}m_G=0$, $\norm{\cS_\rho}_\infty \le \sum_{r=1}^{R(\rho)}2^{-r}<1$, and $\wnorm{\cS_\rho}<\sum_{r=1}^{R(\rho)}2^{-r}<1$.

(ii) It follows from Lemma~\ref{lem:single detector}~(i) and \eqref{eq:defi of R N} that $K_{r,N(\rho)}\le N(\rho)^{-1}\le \rho/4$. Then by \eqref{eq:defi of S_rho}, we see that for each $x\in G$,
$\cS_\rho(x)\le\frac{\rho}{4}\sum_{r=1}^{R(\rho)}2^{-r}<\rho/4$.

(iii) Let $\nu\in\cM(G,T)$ satisfy $D(\nu)\ge\rho$. Then $\sum_{r=1}^{R(\rho)}2^{-r}\abs{\wh\nu(\gamma_r)}^2\ge \rho-2^{-R(\rho)}\ge\frac\rho2.$
Combining this with Lemma~\ref{lem:single detector}~(ii), we get
$\int\!\cS_\rho\,\mathrm{d}\nu
\le\sum_{r=1}^{R(\rho)}2^{-r}
\bigl(N(\rho)^{-1}-
\abs{\wh \nu (\gamma_r)}^2
\bigr)<\frac{\rho}{4}-\frac{\rho}{2}
=-\frac{\rho}{4}.$
\end{proof}

\section{Quantitative averaging and positive-part flattening}
\label{sec:averaging}

The detectors $\cS_\rho$, $0<\rho\le1$, constructed in Proposition~\ref{p:separation}
separate Haar measure from every invariant measure at scale $\rho$. Since no single
scale detects every invariant measure other than $m_G$, the construction must
combine detectors from infinitely many scales, and the difficulty in doing so is to
preserve the gaps already created by earlier detectors: at each stage, the newly added
detector must contribute a definite negative amount to the measures it detects, while
whatever positive amount it contributes to the remaining invariant measures must be
kept under control.

To control this positive contribution, we combine two estimates: an $L^2$ bound for
finite Birkhoff averages, together with the fact that a finite Birkhoff average is
cohomologous to the original function (Lemmas~\ref{lem:L2 estimate}
and~\ref{lem:finite coboundary}); and a Bernstein polynomial approximation of the
positive-part function, with an explicit error bound
(Lemma~\ref{lem:bernstein approximation}). Together these results underlie
Proposition~\ref{p:polynomial-flatten}, which supplies the flattening procedure used
in the recursion.

\begin{notation}
Throughout, $\norm{\cdot}_{L^2}$ denotes the $L^2(m_G)$-norm.
For $M\in\N$, $H\in C(G,\R)$, and $x\in G$, define the $M$-th Birkhoff sum by $S_MH(x)\=\sum_{j=0}^{M-1}H\bigl(T^j(x)\bigr)$.
\end{notation}

We begin with an $L^2$ bound for finite Birkhoff averages.

\begin{lemma}[A finite $L^2$ estimate]
\label{lem:L2 estimate}
Suppose $H\in \cF_\Q$ satisfies $\int\!H\,\mathrm{d}m_G=0$. Then for each $M\in \N$,
\begin{equation*}
\Normbig{M^{-1} S_MH}_{L^2}\le
M^{-1/2}\wnorm{H}.
\end{equation*}
\end{lemma}
\begin{proof}
Suppose $H=\sum_{r=1}^{+\infty} a_r\gamma_r$, where only finitely many $a_r$ are nonzero. Then 
\begin{equation*}
\Normbigg{\frac{S_MH}{M}}_{L^2}
=\Normbigg{\sum_{r=1}^{+\infty}a_r \biggl(\frac{\sum_{k=0}^{M-1}A^k\gamma_r}{M}\biggr)}_{L^2}
\le\sum_{r=1}^{+\infty}\abs{a_r} \Normbigg{\frac{\sum_{k=0}^{M-1}A^k\gamma_r}{M}}_{L^2}
\le \frac{\wnorm{H}}{M^{1/2}},
\end{equation*}
where the final inequality follows from Lemma~\ref{lem:distinct character},
which makes $\gamma_r$, $A\gamma_r$, $\ldots$, $A^{M-1}\gamma_r$ pairwise distinct, and
hence orthonormal in $L^2(m_G)$ by \eqref{eq:orthogonality}.
\end{proof}

The finite $L^2$ estimate above controls the finite Birkhoff average
$M^{-1}S_MH$. We shall also use the standard fact that a finite Birkhoff
average is cohomologous to the original function; see, for example,
\cite[Section~1]{Boc18}. For $H\in C(G,\R)$ and $M\in\N$, define
\begin{equation*}
U_{M,H}
\=\frac1M\sum_{j=0}^{M-2}(M-1-j)H\circ T^j,
\end{equation*}
with $U_{1,H}\=0$.

\begin{lemma}[A finite coboundary identity]
\label{lem:finite coboundary}
For $H\in C(G,\R)$ and $M\in\N$, the following hold:
\begin{enumerate}[label=\rm{(\roman*)}]
\smallskip
\item $H-M^{-1}S_MH=U_{M,H}-U_{M,H}\circ T$.
\smallskip
\item If $H\in\cF_\Q$, then $U_{M,H}\in\cF_\Q$ and $\norm{U_{M,H}}_\infty \le \wnorm{H}(M-1)/2$.
\end{enumerate}
\end{lemma}

\begin{proof}
(i)  This follows directly from the definition. 

\smallskip
(ii) If $H\in\cF_\Q$, then by Lemma~\ref{lem:rational fourier closure}, $U_{M,H}\in \cF_\Q$, and the inequality comes from 
\begin{equation*}\norm{U_{M,H}}_\infty\le
\frac1M \sum_{j=0}^{M-2}(M-1-j)\norm{H}_\infty
=\frac{M-1}{2}\norm{H}_\infty
\le \frac{M-1}{2}\wnorm{H}. \qedhere
\end{equation*}
\end{proof}

To preserve $\cF_{\Q}$ throughout the construction, we
approximate $t_+\=\max\{t,\,0\}$ by a Bernstein polynomial. Let $d\in \N$ and $U>0$. Set $y_j\=U\bigl(\frac{2j}{d}-1\bigr)$ for $0\le j\le d$.
For $t\in \R$, set $u\=(t+U)/(2U)$,
and define \begin{equation}\label{eq:bernstein approximation}
\cB_{d,U}(t)\=
\sum_{j=0}^{d}
{y_j}_+\binom dju^j(1-u)^{d-j}.
\end{equation}
Then on $[-U,U]$ (affinely identified with $[0,1]$), $\cB_{d,U}$ is the Bernstein polynomial of the positive-part function. Since $x\mapsto x_+$ is convex and 1-Lipschitz, the following standard estimate follows from \cite[Theorem~7.1.8]{Ph03} and \cite[Theorem~1]{Ma99}. We include the proof for the convenience of the reader.

\begin{lemma}[Bernstein approximation of the positive part]
\label{lem:bernstein approximation}
Suppose $U>0$ and $d\in \N$. Then for $t\in[-U,U]$, $0\le \cB_{d,U}(t)-t_+ \le U/\sqrt d$. Moreover, if $U\in\Q$, then $\cB_{d,U}\in\Q[t]$.
\end{lemma}

\begin{proof}
Consider a binomial random variable $X$ with distribution given by
$\mathbb P(X=j)=\binom dj u^j(1-u)^{d-j}$ for integers $0\le j\le d$. Set $Y\=U\bigl(\frac{2X}{d}-1\bigr)$. 
For $t\in [-U,U]$ we have
$\cB_{d,U}(t)=\mathbb E[Y_+]$, $\mathbb E [Y]=t$, and
$\Var(Y)=4U^2u(1-u)/d\le U^2/d$. 
Since $x\mapsto x_+$ is convex, Jensen's inequality gives $\cB_{d,U}(t)-t_+=\mathbb E[Y_+]-(\mathbb E [Y])_+\ge 0$; since $x\mapsto x_+$ is $1$-Lipschitz, $\mathbb E[Y_+]-(\mathbb E [Y])_+\le\mathbb E[\abs{Y-\mathbb E [Y]}]$, and the Cauchy--Schwarz inequality then gives $\mathbb E[\abs{Y-\mathbb E[Y]}]\le\sqrt{\Var(Y)}\le U/\sqrt d$.
Finally, if $U\in \Q$, then $\cB_{d,U}\in \Q[t]$ follows directly from the definition.
\end{proof}

Combining the preceding three lemmas yields the flattening estimate used
below.

\begin{prop}[Positive-part flattening]
\label{p:polynomial-flatten} 
Let $\theta$, $\eta>0$, and $B\in\Q_{>0}$,
and let $H\in\cF_\Q$
satisfy $\int\!H\,\mathrm{d}m_G=0$ and $\wnorm{H}\le B$.
Suppose $M\in\N$ satisfies
\begin{equation}\label{eq:recursive birkhoff average}
\sup_G M^{-1}S_MH<\theta
\quad\text{ and }\quad
M^{-1/2}B< \eta/2.
\end{equation}
Set $d\=1+\bigl\lceil(2B/\eta)^2\bigr\rceil$, and define
\begin{equation}\label{eq:polynomial flatten correction}
P_M\=M^{-1}S_MH,\quad
p_M\=\int\!\cB_{d,B}(P_M)\,\mathrm{d}m_G,\quad
Q_M\=p_M-\cB_{d,B}(P_M).
\end{equation}
Then the following hold:
\begin{enumerate}[label=\rm{(\roman*)}]
\smallskip
\item
$0\le p_M<\eta$.
\smallskip
\item
$\norm{Q_M}_\infty<\theta+\eta$ and, for each $K\in\N$,
\begin{equation*}\sup_G K^{-1}S_K(H+Q_M)<
\eta+ K^{-1}(M-1)B.\end{equation*}
In particular, $Q(T,H+Q_M)<\eta$.
\smallskip
\item
$H+Q_M\in\cF_\Q$,
$\int \! (H+Q_M)\,\mathrm{d}m_G=0$, and
$\wnorm{H+Q_M}<\bigl(1+2^d\bigr)B+\eta$.
\end{enumerate}
\end{prop}

\begin{proof}
(i) Since $\wnorm{P_M}\le \wnorm{H}\le B$, we have
$\norm{P_M}_\infty\le B$. Lemma~\ref{lem:bernstein approximation} gives
$0\le\cB_{d,B}(P_M)-(P_M)_+\le B/\sqrt{d}<\eta/2$. Thus $p_M\ge 0$ and
\begin{equation*}
p_M=\int\!\cB_{d,B}(P_M)\,\mathrm{d}m_G
\le \int\!\Bigl(\abs{P_M}+\frac{\eta}{2}\Bigr)\,\mathrm{d}m_G
\le \norm{P_M}_{L^2}+\frac{\eta}{2}
<\eta,
\end{equation*}
where the last inequality follows from Lemma~\ref{lem:L2 estimate} together with
\eqref{eq:recursive birkhoff average}.

\smallskip
(ii) By \eqref{eq:polynomial flatten correction} and \eqref{eq:recursive birkhoff average}, $P_M<\theta$.
Thus by Lemma~\ref{lem:bernstein approximation},
$\cB_{d,B}(P_M)\in [0,\theta+\eta/2)$. Since $p_M\in [0,\eta)$ by~(i), 
we have $-(\theta+\eta/2)<Q_M <\eta$ and thus
$\norm{Q_M}_\infty<\theta +\eta$. By
Lemmas~\ref{lem:finite coboundary} and~\ref{lem:bernstein approximation},
\begin{equation*}
H+Q_M = p_M+(P_M-\cB_{d,B}(P_M)) +U_{M,H}-U_{M,H}\circ T \le p_M+U_{M,H}-U_{M,H}\circ T.
\end{equation*}
Hence, for every $K\in\N$,
\begin{equation*}
\frac{S_K(H+Q_M)}{K}
\le p_M + \frac{U_{M,H}-U_{M,H}\circ T^K}{K} 
\le p_M+\frac{2\norm{U_{M,H}}_\infty}{K}
< \eta+\frac{(M-1)B}{K}.
\end{equation*}
Suppose $\nu\in\cM(G,T)$. By $T$-invariance,
$\int \! (H+Q_M)\,\mathrm{d}\nu
\le\int \! (p_M+U_{M,H}-U_{M,H}\circ T)\,\mathrm{d}\nu=p_M$.
Hence $Q(T,H+Q_M) \le p_M<\eta$.

\smallskip
(iii)
Since $B\in\Q$, Lemma~\ref{lem:bernstein approximation} shows that
$\cB_{d,B}\in \Q[t]$. By Lemma~\ref{lem:rational fourier closure},
$\cB_{d,B}(P_M)\in \cF_{\Q}$. Its Haar integral $p_M$ is therefore rational, being its Fourier
coefficient at the trivial character.
Consequently, $H+Q_M=H+p_M-\cB_{d,B}(P_M)\in\cF_\Q$. By \eqref{eq:polynomial flatten correction} and
$\int\!H\,\mathrm{d}m_G=0$, we have
$\int \! (H+Q_M)\,\mathrm{d}m_G=0$.

Set $V\=(P_M+B)/(2B)$.
Since $\wnorm{P_M}\le \wnorm{H}\le B$, we have
$\wnorm{V}\le1$ and $\wnorm{1-V}\le1$. Hence, by
\eqref{eq:bernstein approximation} and \eqref{eq:wiener norm},
$\wnorm{\cB_{d,B}(P_M)}\le
\sum_{j=0}^d {y_j}_+\binom dj
\wnorm{V}^j\wnorm{1-V}^{d-j}\le
B\sum_{j=0}^d\binom dj=
2^dB$.
Therefore, 
$\wnorm{H+Q_M}\le
\wnorm{H}+p_M+\wnorm{\cB_{d,B}(P_M)}<
B+\eta+2^dB=\bigl(1+2^d \bigr)B+\eta$.
\end{proof}

\section{The explicit construction}
\label{s:induction}

In this section, we construct the sequence $\{F_n\}_{n\in\N_0}$ using
Propositions~\ref{p:separation} and~\ref{p:polynomial-flatten}.
At each step, we first add a detector to enlarge the class of invariant measures on which a negative gap is obtained, and then apply the
flattening procedure to control the positive contribution introduced at the new step without destroying the gaps obtained previously.

\subsection{Scales and parameter choices}\label{ss:parameter pre}

Suppose $n\in\N$. We set $\rho_n\=2^{-n}$ as the detection scale and
consider $\cD_n\=\cS_{\rho_n}=\sum_{r=1}^{n+1}2^{-r}K_{r,2^{n+2}}$ (cf.~\eqref{eq:defi of S_rho}). 
Recall from Proposition~\ref{p:separation} that
\begin{equation}\label{ineq:properties of Dn recalled}
\int\!\cD_n\,\mathrm{d}m_G=0,
\quad
\norm{\cD_n}_\infty<1,
\quad
\wnorm{\cD_n}<1,
\quad
\sup_{x\in G}\cD_n(x)<\frac{\rho_n}{4},
\end{equation}
and that for each $\nu\in\cM(G,T)$ satisfying $D(\nu)\ge\rho_n$,
\begin{equation}\label{ineq:Dn as detector}
\int\!\cD_n\,\mathrm{d}\nu
<-\frac{\rho_n}{4}.
\end{equation}

As $n\to +\infty$, the detectors $\cD_n$
separate $m_G$ from the nested sets
$\{\nu\in\cM(G,T):D(\nu)\ge\rho_n\}$,
whose union is $\cM(G,T)\smallsetminus\{m_G\}$. At step $n$, we add the
scaled detector $a_n\cD_n$, while the choice of $a_n$ ensures that the
new perturbation does not destroy the gaps created at previous steps.

Starting with $a_1\=\frac14$ and $\delta_1\=a_1\rho_1$, we recursively define $a_{n+1}\=\delta_n/8$ and
$\delta_{n+1}\=a_{n+1}\rho_{n+1}$ for each $n\in\N$. Explicitly,
for each $n\in\N$, we have
\begin{equation}
\label{eq:defi of a}
a_n=2^{-\frac{n(n-1)}2-3n+1}
\quad\text{ and }\quad
\delta_n=
2^{-\frac{n(n-1)}2-4n+1}.
\end{equation}
For each $n\in\N_0$, we also set
\begin{equation}
\label{eq:b}
b_n\=\delta_{n+1}/64
=2^{-\frac{n(n+1)}2-4n-9}.
\end{equation}

\subsection{The inductive construction}\label{ss:inductive construction}

Let $F_0$ be the constant zero function on $G$, and set
$B_0\=0$, $\Theta_0\=0$.
Then \eqref{eq:recursive property} and \eqref{eq:inductive birkhoff bound}  below hold for
$n=1$: indeed $F_0\in\cF_\Q$, $\int\!F_0\,\mathrm{d}m_G=0$, $Q(T,F_0)=0<b_0=2^{-9}$,
$\norm{F_0}_{A}=0=B_0$, and $\sup_G K^{-1}S_KF_0=0<b_0$ for every $K\in\N$.

Suppose $(F_{n-1},\Theta_{n-1},B_{n-1})$ with $B_{n-1},\Theta_{n-1}\in \Q$ has been constructed for some
$n\in\N$ and satisfies
\begin{equation}
\label{eq:recursive property}
F_{n-1}\in\cF_\Q,
\quad \int\!F_{n-1}\,\mathrm{d}m_G=0,
\quad Q(T,F_{n-1})<b_{n-1},
\quad \wnorm{F_{n-1}}\le B_{n-1},
\end{equation}
and, for every $K\in\N$,
\begin{equation}
\label{eq:inductive birkhoff bound}
\sup_G K^{-1}S_KF_{n-1}
< b_{n-1}+K^{-1}\Theta_{n-1}.
\end{equation}
We now construct $(F_n,\Theta_n,B_n)$ with $B_{n}$, $\Theta_{n}\in \Q$.

We first add the detector and form the intermediate function
\begin{equation}\label{eq:from Fn-1 to Hn}
H_n\=F_{n-1}+a_n\cD_n.
\end{equation}
By \eqref{ineq:Dn as detector}, \eqref{eq:recursive property}, and the definition \eqref{eq:b} of $b_{n-1}$, we see that for each
$\nu\in\cM(G,T)$ satisfying $D(\nu)\ge\rho_n$,
\begin{equation}\label{eq:exposure property of Hn}
\int\!H_n\,\mathrm{d}\nu<
Q(T,F_{n-1})-\frac{\delta_n}{4}
<-\frac{15}{64}\delta_n.
\end{equation}

Set $\oB_n\=B_{n-1}+a_n\in\Q$. By \eqref{eq:recursive property},
\eqref{eq:from Fn-1 to Hn}, \eqref{ineq:properties of Dn recalled}, and \eqref{eq:inductive birkhoff bound}, we have
$\wnorm{H_n}<\oB_n$, and, for each $K\in\N$,
\begin{equation}
\label{eq:birkhoff bound for Hn}
\sup_G K^{-1}S_KH_n
<b_{n-1}+K^{-1}\Theta_{n-1}
+(1/4)\delta_n.
\end{equation}

To ensure that $Q(T,F_n)$ is sufficiently small relative to
$\delta_{n+1}$, we use the flattening procedure to construct $F_n$
from $H_n$. Note that
$\int\!H_n\,\mathrm{d}m_G=0$ by the induction hypothesis and
\eqref{ineq:properties of Dn recalled}.
We now apply
Proposition~\ref{p:polynomial-flatten} to $H_n$ with
$2b_{n-1}+\delta_n/4$ in place of $\theta$,  $b_n$ in place of $\eta$,
$\oB_n$ in place of $B$, and $M_n$ in place of $M$, where
\begin{equation}\label{eq:Mn}
M_n\=1+\bigl\lceil
\max\bigl\{\Theta_{n-1}/b_{n-1},\,
(2\oB_n/b_n)^2
\bigr\}\bigr\rceil.
\end{equation}  
It follows from \eqref{eq:birkhoff bound for Hn} and \eqref{eq:Mn}
that $\sup_G M_n^{-1}S_{M_n}H_n<
2b_{n-1}+ \delta_n/4$.
Moreover, by \eqref{eq:Mn}, $M_n^{-1/2}\oB_n< b_n/2$. We then set the Bernstein polynomial degree
\begin{equation}\label{eq:Bernstein degree}
d_n\=1+\bigl\lceil
(2\oB_n/b_n)^2
\bigr\rceil,
\end{equation}
and consider
\begin{equation*}C_n\=\int\!\cB_{d_n,\oB_n}
\biggl(\frac{S_{M_n}H_n}{M_n}\biggr)\,\mathrm{d}m_G\quad\text{ and }\quad
Q_n\=C_n-\cB_{d_n,\oB_n}
\biggl(\frac{S_{M_n}H_n}{M_n}\biggr).\end{equation*}
We define
\begin{equation}\label{eq:recursion construction}
F_n\=H_n+Q_n=
H_n+C_n-\cB_{d_n,\oB_n}
\bigl(M_n^{-1}S_{M_n}H_n\bigr).
\end{equation}
By Proposition~\ref{p:polynomial-flatten}, we see that
$F_n\in\cF_\Q$, $\int\!F_n\,\mathrm{d}m_G=0$, $C_n<b_n$, and
$Q(T,F_n)<b_n$. Moreover,
\begin{equation}\label{eq:Qn infty bound}
\norm{Q_n}_\infty<
2b_{n-1}+(1/4)\delta_n+b_n,
\end{equation}
and, for each $K\in\N$,
$
\sup_G K^{-1}S_KF_n <
b_n+K^{-1}(M_n-1)\oB_n.$ 
Moreover, $\wnorm{F_n}< \bigl(1+2^{d_n}\bigr)\oB_n+b_n$.
Thus, we define
\begin{equation}
\label{eq:B Theta n}
\begin{aligned}
B_n&\= \bigl(1+2^{d_n}\bigr)\oB_n+b_n
=\bigl(1+2^{d_n}\bigr)(B_{n-1}+a_n)+b_n
\in \Q\quad\text{ and}\\
\Theta_n&\=(M_n-1)\oB_n
=(M_n-1)(B_{n-1}+a_n)\in\Q.
\end{aligned}
\end{equation}
This completes the inductive step.

\subsection{Exposure}
\label{ss:exposure}
In this subsection, we prove Theorem~\ref{thm:main} and record the properties of $\{F_n\}_{n\in \N_0}$ for later reference.

\begin{prop}[Properties of the constructed sequence]
\label{p:recursive construction}
Let $\{F_n\}_{n\in\N_0}$, $\{\Theta_n\}_{n\in\N_0}$, and $\{B_n\}_{n\in\N_0}$ be the sequences recursively defined in Subsection~\ref{ss:inductive construction}.
Then for each $n\in\N$ and each $K\in \N$,
\begin{equation*}
F_n\in\cF_\Q,
\quad
\int\!F_n\,\mathrm{d}m_G=0,
\quad
Q(T,F_n)<b_n,
\quad
\wnorm{F_n}<B_n,
\quad\text{ and }\quad
\sup_G \frac{S_KF_n}{K}<
b_n+ \frac{\Theta_n}{K}.
\end{equation*}
\end{prop}

\begin{proof}
These are precisely the conclusions of the recursive construction in Subsection~\ref{ss:inductive construction}.
\end{proof}

We first show that $F_n$ has a negative integral, of definite size, against every invariant measure detected at scale $\rho_n$.

\begin{prop}[Exposure estimate]
\label{p:negative value estimate}
Let $F_n$ be as in Proposition~\ref{p:recursive construction}.
If $n\in\N$ and $\nu\in\cM(G,T)$ satisfy $D(\nu)\ge 2^{-n}$, then $\int\!F_n\,\mathrm{d}\nu
< -\frac{7}{32}\delta_n$.
\end{prop}

\begin{proof}
It follows from \eqref{eq:recursion construction},
Proposition~\ref{p:polynomial-flatten}~(i), and \eqref{eq:exposure property of Hn} that
$\int\!F_n\,\mathrm{d}\nu\le
\int\!H_n\,\mathrm{d}\nu+C_n
< -\frac{15}{64}\delta_n+b_n
< -\frac{7}{32}\delta_n$,
where the last inequality follows from the fact that
$b_n\le\delta_n/2^{10}$ (cf.~\eqref{eq:b} and
\eqref{eq:defi of a}).
\end{proof}

Proposition~\ref{p:negative value estimate} gives a negative gap at each
fixed step. We next estimate the change at each subsequent step.

\begin{prop}[Increment estimate]
\label{p:increment estimate}
Let $\{a_n\}_{n\in\N}$ be the sequence defined by
\eqref{eq:defi of a}. Then for each $n\in\N$,
\begin{equation*}
\norm{F_n-F_{n-1}}_\infty
<(8/7)a_n.
\end{equation*}
Consequently, $\{F_n\}_{n\in \N_0}$ converges uniformly to a function $F\in C(G,\R)$, and $\norm{F-F_N}_{\infty}<\frac{16}{105}\delta_N$
for each $N\in\N$.
\end{prop}

\begin{proof}
Recall that $\cD_n=\cS_{\rho_n}$ (cf.~Subsection~\ref{ss:parameter pre}).
It follows from \eqref{eq:from Fn-1 to Hn},
\eqref{eq:recursion construction}, and \eqref{eq:Qn infty bound}
that
$\norm{F_n-F_{n-1}-a_n\cD_n}_\infty
<2b_{n-1}+\frac{\delta_n}{4}+b_n$.
By \eqref{ineq:properties of Dn recalled}, we see that
$\norm{F_n-F_{n-1}}_\infty<
a_n+2b_{n-1}+ \delta_n/4+b_n$.
Moreover, by \eqref{eq:b} and \eqref{eq:defi of a},
$2b_{n-1}+\frac{\delta_n}{4}+b_n
\le\bigl(\frac{1}{32}+\frac{1}{4}+\frac{1}{2048}\bigr)\delta_n
<\frac27 \delta_n$.
Consequently, $\norm{F_n-F_{n-1}}_\infty
<a_n+\frac27 \delta_n
=a_n+\frac27\cdot 2^{-n}a_n
\le \frac87 a_n$.
Hence by \eqref{eq:defi of a}, $F_n$ converges uniformly to some $F\in C(G,\R)$ as $n\to +\infty$. By \eqref{eq:defi of a}, $a_{n+1}/a_{n}=2^{-(n+3)}\le \frac1{16}$. Hence
$\norm{F-F_N}_\infty
<\frac87 \sum_{n=N+1}^{+\infty}a_n
\le \frac87 \cdot\frac{a_{N+1}}{1-1/16}
\le \frac{16}{105}\delta_N$.
\end{proof}

Combining Propositions~\ref{p:negative value estimate} and~\ref{p:increment estimate} gives a quantitative form of exposure for the limit $F$.

\begin{prop}[Quantitative exposure bound]
\label{p:exposure bound}
Let $F$ be as in Proposition~\ref{p:increment estimate}. Then
\begin{equation*}
\int\!F\,\mathrm{d}m_G=0,
\end{equation*}
and for each $\nu\in\cM(G,T)\smallsetminus\{m_G\}$, 
\begin{equation}\label{eq:analytic-exposure-gap}
\int\!F\,\mathrm{d}m_G-
\int\!F\,\mathrm{d}\nu
>2^{-(n(\nu)+1)(n(\nu)+6)/2},
\end{equation}
where $n(\nu)\=\max\bigl\{
1,\,\bigl\lceil\log_2\frac1{D(\nu)}\bigr\rceil\bigr\}$.
\end{prop}
\begin{proof}
It follows from Proposition~\ref{p:recursive construction} and the
uniform convergence of $\{F_n\}_{n\in\N_0}$ that
$\int\!F\,\mathrm{d}m_G=0$.
Suppose $\nu\in\cM(G,T)\smallsetminus\{m_G\}$. By Lemma~\ref{lem:D gauge}, $D(\nu)>0$, so $n(\nu)$ is a well-defined
positive integer and
$2^{-n(\nu)}\le D(\nu)$. 
It then follows from
Propositions~\ref{p:negative value estimate} and
\ref{p:increment estimate} that
\begin{equation*}
\int\!F\,\mathrm{d}\nu\le
\int\!F_{n(\nu)}\,\mathrm{d}\nu
+\norm{F-F_{n(\nu)}}_\infty
<-\frac7{32}\delta_{n(\nu)}+\frac{16}{105}\delta_{n(\nu)}
<-\frac{\delta_{n(\nu)}}{16}.
\end{equation*}
Note that by \eqref{eq:defi of a}, for each $n\in\N$,
$\delta_n/16=2^{-(n+1)(n+6)/2}$.
Now \eqref{eq:analytic-exposure-gap} follows.
\end{proof}

We can now combine Propositions~\ref{p:recursive construction},
\ref{p:increment estimate}, and \ref{p:exposure bound} in order to prove Theorem~\ref{thm:main}:

\begin{proof}[Proof of Theorem~\ref{thm:main}]
By Proposition~\ref{p:recursive construction}, each $F_n\in\cF_\Q$; in particular $F_n$ is
a real-valued trigonometric polynomial on $G$ with rational Fourier coefficients. By
Proposition~\ref{p:increment estimate}, $\{F_n\}_{n\in \N_0}$ converges uniformly to a function
$F\in C(G,\R)$. By Proposition~\ref{p:exposure bound},
\begin{equation*}
\int\!F\,\mathrm{d}m_G=0
\quad\text{ and }\quad
\int\!F\,\mathrm{d}\nu<0
\quad\text{ for each }\nu\in\cM(G,T)\smallsetminus\{m_G\}.
\end{equation*}
Consequently $Q(T,F)=\sup_{\mu\in\cM(G,T)}\int\!F\,\mathrm{d}\mu=0$, and this supremum is
attained only at $\mu=m_G$; that is, $\cM_{\max}(F)=\{m_G\}$, so $m_G$ is the unique
$T$-invariant maximizing measure for $F$.
\end{proof}

\section{A uniformly convergent Fourier series}
\label{s:Fourier series}

Each finite approximant $F_n\in\cF_\Q(G)$ has a finite Fourier expansion, and hence all of its Fourier coefficients are determined explicitly by finitely many steps of the construction. It is therefore natural to ask whether an exposing function can itself be represented by a Fourier series converging in a prescribed frequency order.

To carry this out, we use the enumeration \eqref{eq:enumeration}, which so far has been
arbitrary, as the frequency order. For every $G$ and every enumeration, relocating the
frequencies of the increments $F_n-F_{n-1}$ gives a function $\widetilde F$ which exposes
$m_G$ and whose Fourier series converges uniformly in the order of the enumeration
(Propositions~\ref{p:abstract uniform fourier} and~\ref{p:tilde F exposure}). Under
additional effectivity assumptions on the enumeration, each Fourier coefficient of
$\widetilde F$ is determined exactly by a finite procedure
(Proposition~\ref{p:exact coefficients}). In
Subsection~\ref{ss:effective Fourier construction} we specialise to expanding circle
endomorphisms, ergodic toral endomorphisms, and the one-sided full shift, where the
classical frequency orders satisfy these assumptions; this gives
Theorems~\ref{thm:circle}, \ref{thm:torus}, and~\ref{thm:shift}.

\subsection{A Fourier series with exact Fourier coefficients}

For all $P\in C(G,\R)$ and $J\in\N_0$ we have $\Normbig{P\circ T^J}_\infty=\norm P_\infty$,
since $T^J$ is surjective, and, by the $T$-invariance of $\nu$,
\begin{equation}\label{eq:frequency-copy-invariance}
\int\!P\circ T^J\,\mathrm{d}\nu=\int\!P\,\mathrm{d}\nu
\quad\text{ for all }\nu\in\cM(G,T).
\end{equation}

For each $n\in \N$, set $\Delta_n\=F_n-F_{n-1}$
and $\alpha_n\=\norm{\Delta_n}_\infty.$
By Lemma~\ref{lem:rational fourier closure}, $\Delta_n\in\cF_\Q$. We estimate the size of the increment $\Delta_n$. By
Proposition~\ref{p:increment estimate}, $\alpha_n< \frac{8}{7}a_n$, so by \eqref{eq:defi of a}, we have
\begin{equation}\label{eq:alpha summable}
\sum_{n=1}^{+\infty}\alpha_n<\frac{8}{7}\sum_{n=1}^{+\infty}a_n<\frac{8}{7}\sum_{n=1}^{+\infty}2^{-3n+1} <+\infty.
\end{equation}
Since the increments are summable in the uniform norm, we may relocate them into separated
frequency blocks: 
$\Delta_n\circ T^{J}$ has the same uniform norm as $\Delta_n$ and,
by \eqref{eq:frequency-copy-invariance}, 
the same integral
with respect to each invariant measure,
while, by
Lemma~\ref{lem:rational fourier closure}, its frequency support is $A^{J}\bigl(\supp\wh{\Delta_n}\bigr)$.

Relocation alone, however, does not suffice. A partial sum whose cutoff falls inside the
frequency support of a single $\Delta_n$ may differ from $\Delta_n$ by as much as
$\wnorm{\Delta_n}$, and the Wiener norms $\wnorm{\Delta_n}$ are in general far larger than
the uniform norms $\alpha_n$. We therefore also \emph{split}: we replace $\Delta_n$ by
$ K_n$ relocated copies of $ K_n^{-1}\Delta_n$, placed in disjoint frequency blocks,
where $K_n$ is large enough that each block has Wiener norm at most $2^{-n}$. The uniform norm of the sum does not increase, but no cutoff can now remove more than $2^{-n}$ from a single block. Concretely, for $n\in\N$ set
\begin{equation}\label{eq:number-of-frequency-copies}
K_n\=1+\lceil2^n\wnorm{\Delta_n}\rceil,
\end{equation}
let $J_{n,s}\in\N_0$, $1\le s\le K_n$, be integers to be specified below, and define
\begin{equation}\label{eq:spread-increment}
\widetilde\Delta_{n,s}\=\frac1{ K_n}\Delta_n\circ T^{J_{n,s}}
\quad\text{ and }\quad
\widetilde\Delta_n\=\sum_{s=1}^{ K_n}\widetilde\Delta_{n,s}
=\frac1{ K_n}\sum_{s=1}^{ K_n}\Delta_n\circ T^{J_{n,s}}.
\end{equation}
The \emph{frequency support} of $\Delta_n$ is
\begin{equation}\label{eq:defi of E_n}
E_n\=\bigl\{\gamma\in\Gamma:\wh{\Delta_n}(\gamma)\neq0\bigr\},
\end{equation}
which is finite by Proposition~\ref{p:recursive construction} and Lemma~\ref{lem:rational fourier closure}, and does not contain the trivial character because
$\int\!\Delta_n\,\mathrm{d}m_G=0$.
For convenience, we include the following fact.
\begin{lemma}[Nonempty frequency support]\label{lem:nonempty support}
For each $n\in\N$, $\Delta_n(e_G)<0$, and consequently, $E_n\ne\emptyset$.
\end{lemma}

\begin{proof}
Since $T(e_G)=e_G$, the Dirac measure $\delta_{e_G}$ belongs to $\cM(G,T)$. 
Moreover, $\wh{\delta_{e_G}}(\gamma_r)=1$ for each $r\in\N$, so Definition~\ref{def:separation function D} gives $D(\delta_{e_G})=1\ge\rho_n$. 
Hence \eqref{ineq:Dn as detector} yields
$\cD_n(e_G)=\int\!\cD_n\,\mathrm d\delta_{e_G}<
-\frac{\rho_n}{4}$.
By Lemma~\ref{lem:bernstein approximation},
$\cB_{d_n,\oB_n}\ge0$ on
$[-\oB_n,\oB_n]$;
since
$\Normbig{M_n^{-1}S_{M_n}H_n}_\infty\le\norm{H_n}_\infty\le\wnorm{H_n}<\oB_n$
by Subsection~\ref{ss:inductive construction},
we have
$M_n^{-1}S_{M_n}H_n(e_G)\in[-\oB_n,\oB_n]$, so
\eqref{eq:recursion construction} gives
$Q_n(e_G)\le C_n$, while
Proposition~\ref{p:polynomial-flatten}~(i) gives $C_n<b_n$.
Therefore, using
\eqref{eq:from Fn-1 to Hn},
\eqref{eq:recursion construction}, and
$\delta_n=a_n\rho_n$, we obtain
\begin{equation*}
\Delta_n(e_G)=a_n\cD_n(e_G)+Q_n(e_G)<
-(1/4)a_n\rho_n+b_n=-(1/4)\delta_n+b_n<0,
\end{equation*}
where the last inequality follows from
\eqref{eq:defi of a} and \eqref{eq:b}.
By Proposition~\ref{p:recursive construction}, $\int\!\Delta_n\,\mathrm dm_G=\int\!F_n\,\mathrm dm_G-\int\!F_{n-1}\,\mathrm dm_G=0$.
If $E_n=\emptyset$ then every Fourier coefficient of $\Delta_n$ vanishes, so
$\Delta_n\equiv0$, contradicting $\Delta_n(e_G)<0$.
\end{proof}

Now we index the relocated blocks $A^{J_{n,s}}E_n$ by the
pairs $(n,s)$ in lexicographic order.

\begin{lemma}[Frequency relocation]
\label{lem:frequency relocation}
Let $E\subseteq\Gamma\smallsetminus\{1\}$ be finite and $M\in\N_0$. Then there exists an integer $J\in[0,M\abs{E}]$ such that
\begin{equation}\label{eq:frequency relocation}
A^JE\cap\{\gamma_1,\,\gamma_2,\,\ldots,\,\gamma_M\}=\emptyset.
\end{equation}
Equivalently, every character in $A^JE$ occurs after $\gamma_M$ in the prescribed enumeration.
\end{lemma}

\begin{proof}
For each $\gamma\in E$ and each $i\in\{1,\,2,\,\ldots,\,M\}$, by Lemma~\ref{lem:distinct character}, $A^J\gamma=\gamma_i$ holds for at most one $J\in \N_0$. 
Thus, there are at most $M\abs{E}$ integers $J$ for which $A^JE\cap\{\gamma_1,\,\gamma_2,\,\ldots,\,\gamma_M\} \neq \emptyset$.
Since the $M\abs{E}+1$ integers $0,1,\ldots,M\abs{E}$ therefore cannot all yield a nonempty intersection, at least
one of them satisfies \eqref{eq:frequency relocation}. Finally, $A^J\gamma\neq1$ for
$\gamma\neq1$ by injectivity of $A$, so \eqref{eq:frequency relocation} says exactly that
every character of $A^JE$ occurs after $\gamma_M$ in the enumeration.
\end{proof}

We shall use the index set $\cI\=\bigl\{(a,b)\in \N^2\:1\le b\le K_a\bigr\}$.
Now we define $J_{n,s}$ recursively.
Assume $J_{a,b}$ has been determined for all
$(a,b)\in \cI$ with $(a,b)<(n,s)$ in the lexicographical order.
We record the Fourier positions that have already been occupied by the previous blocks:
\begin{equation}\label{eq:union of disjoint block}
\cO_{n,s}\=\bigcup_{(a,b)\in \cI: (a,b)<(n,s)}A^{J_{a,b}}E_a.
\end{equation}
Since $\cO_{n,s}$ is finite, we define $M_{n,s}\=
\max\{r\in\N:\gamma_r\in\cO_{n,s}\}$,
with the convention that $M_{n,s}\=0$ if
$\cO_{n,s}=\emptyset$.
By Lemma~\ref{lem:frequency relocation}, at least one integer
$J\in\{0,\,1,\,\ldots,\,M_{n,s}\abs{E_n}\}$ satisfies
\begin{equation}\label{eq:fourier box recursion}
A^{J}E_n\cap
\{\gamma_1,\,\gamma_2,\,\ldots,\,\gamma_{M_{n,s}}\}
=\emptyset ;
\end{equation}
we let $J_{n,s}$ be the least such $J$.
Thus all characters in $A^{J_{n,s}}E_n$ are placed strictly after all previously occupied frequency positions.

By the recursive choice of $J_{n,s}$, the family of \emph{Fourier blocks}
$A^{J_{n,s}}E_n$, indexed by the pairs $(n,s)$ with $n\in\N$ and
$1\le s\le K_n$, has the following two properties:
\begin{enumerate}[label=\rm{(B\arabic*)}]
\smallskip
\item\label{B1} the blocks are pairwise disjoint;
\smallskip
\item\label{B2} if $(a,b)\in \cI$ with $(a,b)<(n,s)$ in the lexicographic order, then every character
in $A^{J_{n,s}}E_n$ occurs strictly later than every character in $A^{J_{a,b}}E_a$
with respect to the enumeration \eqref{eq:enumeration}.
\end{enumerate}
\smallskip
The proofs of Propositions~\ref{p:abstract uniform fourier}
and~\ref{p:tilde F exposure} below use the integers $J_{n,s}$ only through
\ref{B1} and \ref{B2}. Accordingly, in
Subsection~\ref{ss:effective Fourier construction} the recursive choice above is
replaced, in each of three concrete settings, by an explicit prescription, and all
that has to be checked is that \ref{B1} and \ref{B2} still hold.

For $N\in\N$, define $\widetilde F_N\=\sum_{n=1}^{N}\widetilde\Delta_n$.
Since $\Normbig{\widetilde\Delta_N}_\infty\le\norm{\Delta_N}_\infty=\alpha_N$,
it follows from \eqref{eq:alpha summable} that
$\bigl\{\widetilde F_N\bigr\}_{N\in\N}$ converges uniformly; we denote its limit
by $\widetilde F$.
Each $\widetilde F_N$ is a finite linear combination of
characters, hence continuous, so $\widetilde F\in C(G,\R)$.

\begin{prop}[Fourier series in a prescribed order]
\label{p:abstract uniform fourier}
There exists a sequence $\{c_r\}_{r\in\N}$ in $\Q$ such that $c_r=\wh{\widetilde F}(\gamma_r)$
for all $r\in\N$ and that
\begin{equation}\label{eq:abstract fourier series}
\widetilde F(x)=\sum_{r=1}^{+\infty}c_r\gamma_r(x),
\end{equation}
with the partial sums converging uniformly on $G$ in the enumeration order
\eqref{eq:enumeration}.
\end{prop}
\begin{proof}
Recall from \ref{B1} that for $n\in\N$ and $1\le s\le K_n$,
the Fourier blocks $A^{J_{n,s}}E_n$
are pairwise disjoint. If $\gamma_r=A^{J_{n,s}}\gamma$ for some $n\in \N$, $1\le s \le K_n$, and $\gamma\in E_n$, then set
\begin{equation}\label{eq:New Fourier coefficient}
c_r\=\frac1{K_n}\wh{\Delta_n}(\gamma);
\end{equation}
otherwise set $c_r\=0$. 
This is well-defined: by \ref{B1} at most one pair $(n,s)$ can occur, and $A^{J_{n,s}}$ is
injective by Lemma~\ref{lem:distinct character}, so $\gamma$ is then determined by
$\gamma_r$.

We now prove the uniform convergence of the Fourier partial sums.
For each subset $E\subseteq E_n$,
\begin{equation*}\Normbigg{\frac1{K_n}\sum_{\gamma\in E}
\wh{\Delta_n}(\gamma)
A^{J_{n,s}}\gamma}_\infty
\le
\frac{1}{K_n}\sum_{\gamma\in E}
\Absbig{\wh{\Delta_n}(\gamma)}
\le\frac{\wnorm{\Delta_n}}{K_n}
<2^{-n}.\end{equation*}
Hence, a cutoff inside a single Fourier block creates an error of at most $2^{-n}$.

For $R\in\N$, if the cutoff at $\gamma_R$ falls inside the block
corresponding to $\widetilde\Delta_{n,s_0}$, then the remaining part is
bounded as follows:
\begin{equation}\label{eq:abstract partial sum}
\Normbigg{
\widetilde F-\sum_{r=1}^{R}c_r\gamma_r}_\infty
\le2^{-n}+\sum_{s=s_0+1}^{K_n}
\Normbig{\widetilde\Delta_{n,s}}_\infty
+\sum_{j=n+1}^{+\infty}\Normbig{\widetilde\Delta_j}_\infty 
\le2^{-n}+\alpha_n+\sum_{j=n+1}^{+\infty}\alpha_j .
\end{equation}
If the cutoff lies between two consecutive blocks, the first term is absent.

Write $g(n)\=2^{-n}+\sum_{j=n}^{+\infty}\alpha_j$, so that \eqref{eq:abstract partial sum} gives
$\bigl\Vert\widetilde F-\sum_{r=1}^{R}c_r\gamma_r\bigr\Vert_\infty\le g(n)$ whenever the
cutoff at $\gamma_R$ lies in, or immediately after, a block with first index $n\in\N$. By
\eqref{eq:alpha summable}, $g$ is nonincreasing and $g(n)\to0$ as $n\to+\infty$. For each
$n$ the set $\bigcup_{(a,b)\le(n,K_n)}A^{J_{a,b}}E_a$ is finite, so for all
sufficiently large $R\in \N$ the cutoff lies beyond every block with first index at most $n$;
hence $\bigl\Vert\widetilde F-\sum_{r=1}^Rc_r\gamma_r\bigr\Vert_\infty\le g(n)$ for all
such $R$, and $\sum_{r=1}^{R}c_r\gamma_r\to\widetilde F$ uniformly on $G$ as $R\to +\infty$.

Since this convergence is uniform and the characters are orthonormal in
$L^2(m_G)$, term-by-term integration gives $c_r=\wh{\widetilde F}(\gamma_r)$ for all
$r\in\N$, and $\wh{\widetilde F}(1)=0$; thus \eqref{eq:abstract fourier series} is the
Fourier series of $\widetilde F$, listed in the prescribed order, as claimed.
\end{proof}

Since relocation does not change integrals against invariant measures, $\widetilde F$ inherits the exposure property of $F$:

\begin{prop}[A Fourier series exposure]
\label{p:tilde F exposure}
The normalised Haar measure $m_G$ is the unique $T$-invariant maximizing measure for $\widetilde F$ defined above.
\end{prop}

\begin{proof}
By \eqref{eq:spread-increment}, \eqref{eq:frequency-copy-invariance}, and
$\widetilde F_N=\sum_{n=1}^N\widetilde\Delta_n$, we have
$\int\!\widetilde F_N\,\mathrm{d}\nu=\int\!F_N\,\mathrm{d}\nu$ for all
$\nu\in\cM(G,T)$. Letting $N\to+\infty$ and using the uniform convergence of
$\{F_N\}_{N\in \N}$ and $\bigl\{\widetilde F_N\bigr\}_{N\in \N}$ together with $\nu(G)=1$, we obtain
$\int\!\widetilde F\,\mathrm{d}\nu=\int\!F\,\mathrm{d}\nu$. The conclusion now
follows from Theorem~\ref{thm:main}.
\end{proof}

Finally, the coefficients $c_r$ can be computed exactly, provided that the enumeration and the integers $J_{n,s}$ are effective:

\begin{prop}[Exact computability of the coefficients]
\label{p:exact coefficients}
Suppose the enumeration \eqref{eq:enumeration} is chosen so that, given $r\in\N$, the
character $\gamma_r$ and the integers $J_{n,s}$ can be produced by finite procedures, and
that equality in $\Gamma$ is decidable. Then there is an algorithm which, on input
$r\in\N$, returns the exact value of $c_r\in\Q$ after finitely many operations in the sense of Remark~\ref{rem: computable}, 
generating at most $r$ Fourier blocks.
\end{prop}
\begin{proof}
    Every quantity in the recursion of Subsection~\ref{ss:inductive construction} is a rational
number or a positive integer obtained from earlier data by finitely many 
operations of types \ref{op:arith}--\ref{op:haar}:
this is the content of \eqref{eq:defi of a}, \eqref{eq:b}, \eqref{eq:defi of R N},
\eqref{eq:Mn}, \eqref{eq:Bernstein degree}, and \eqref{eq:B Theta n}, together with
Lemmas~\ref{lem:rational fourier closure} and~\ref{lem:bernstein approximation}. In
particular, for each $n\in\N$ the complete finite Fourier expansion of $F_n$, hence of
$\Delta_n=F_n-F_{n-1}$, is produced exactly, and so are $\norm{\Delta_n}_{A}\in\Q$ and
$K_n=1+\lceil2^n\norm{\Delta_n}_{A}\rceil$. 
Enumerate $\cI$ in lexicographic order as $(n_1,s_1)$, $(n_2,s_2)$, $\ldots$, which is
possible since each $K_a$ is finite, and write $\cE_i\=A^{J_{n_i,s_i}}E_{n_i}$. By
Lemma~\ref{lem:nonempty support} each $E_n$ is nonempty, and $A$ is injective, so every
$\cE_i$ is nonempty; set $\mu(i)\=\min\{r\in\N:\gamma_r\in\cE_i\}$.
By \ref{B2}
the
integers $\mu(1)$, $\mu(2)$, $\ldots$ are strictly increasing, so $\mu(i)\ge i$ for all
$i\in\N$, and therefore
\begin{equation*}
\cE_i\cap\{\gamma_1,\,\gamma_2,\,\ldots,\,\gamma_r\}=\emptyset
\quad\text{ for all }
i>r.
\end{equation*}
On input $r$ it is thus enough to generate $\cE_1,\ldots,\cE_r$,
that is, at
most $r$ Fourier blocks,
for which
only the first $n_r\le r$ steps of the recursion of
Subsection~\ref{ss:inductive construction} are needed.
Testing whether $\gamma_r$ lies in one of these finitely many finite sets uses only
decidability of equality in $\Gamma$.
By \ref{B1} at most one $i$ occurs, and
\eqref{eq:New Fourier coefficient} gives
$c_r= K_{n_i}^{-1}\wh{\Delta_{n_i}}(\gamma)$
if $\gamma_r=A^{J_{n_i,s_i}}\gamma$
with $\gamma\in E_{n_i}$ and $i\le r$,
and $c_r=0$ otherwise.
\end{proof}

\subsection{Effective Fourier series construction}
\label{ss:effective Fourier construction}

The integer $J_{n,s}$ in \eqref{eq:fourier box recursion} is found by a finite search.
Hence Propositions~\ref{p:abstract uniform fourier}, \ref{p:tilde F exposure},
and~\ref{p:exact coefficients} apply to every enumeration \eqref{eq:enumeration} for which
$\gamma_r$ can be produced from $r$ by a finite procedure and equality in $\Gamma$ is
decidable, and this is all that is needed for Theorem~\ref{thm:torus}. For expanding circle
endomorphisms and the one-sided full shift, where $A$ multiplies each frequency index by a
fixed integer, the integers $J_{n,s}$ can moreover be prescribed by explicit recursive
formulas. In these two settings we follow the same three steps: we introduce a gauge on
$\cF_\Q$ (the \emph{degree}, or the \emph{depth}) which is controlled under the four closure
properties of Lemma~\ref{lem:rational fourier closure}; we use it to bound the frequency
support of $\Delta_n$ and hence of its relocated copies; and we then choose the integer
$J_{n,s}$ so that consecutive blocks are strictly separated, which gives \ref{B1}
and~\ref{B2} and allows Propositions~\ref{p:abstract uniform fourier}
and~\ref{p:tilde F exposure} to apply.

\subsubsection{Expanding endomorphisms on the circle: proof of Theorem~\ref{thm:circle}}

Here we prove Theorem~\ref{thm:circle}.
We choose the enumeration in \eqref{eq:enumeration} by setting 
\begin{equation}\label{eq:enumeration for circle}
\gamma_{2k-1}(x)\=\exp(2\pi\I kx),\quad\gamma_{2k}(x)\=\exp(-2\pi\I kx),
\end{equation}
for $k\in\N$. 
On $\T$, inversion-invariance of the Fourier coefficients is the classical condition for a function to be even, so every $P\in\cF_{\Q}$ can be written as
\begin{equation}\label{eq:cosine representation}
P(x)=a_0+\sum_{k=1}^{N}a_k\cos(2\pi kx),\end{equation} 
where $a_0,\,a_1,\,\ldots,\,a_N\in\Q$ and $N\in\N_0$. 

We define the \emph{nonnegative Fourier support} of $P$ to be
$\{k\in\N_0:a_k\neq0\}$, and set $\deg P\=\max(\{k\in\N:a_k\neq0\}\cup\{0\})$. We require the following simple lemma, whose proof is left to the reader.

\begin{lemma}
\label{lem:circle degree estimate}
For $P$, $Q\in \cF_{\Q}$, $R\in\Q[t]\smallsetminus\{0\}$, $j\in \N_0$, and the expanding map $T(x)=d\, x \pmod{1}$ on $\T$, 
with $d\in\Z$, $\abs{d}\ge 2$,
the following hold:
\begin{enumerate}[label=\rm{(\roman*)}]
\smallskip
\item $\deg(P+Q)\le \max\{\deg P,\,\deg Q\}$;
\smallskip
\item $\deg(P\cdot Q)\le \deg P+\deg Q$;
\smallskip
\item $\deg\bigl(P\circ T^j \bigr)=\abs{d}^j\deg P$;
\smallskip
\item $\deg(R\circ P)= \deg R\cdot \deg P$.
\end{enumerate}
\end{lemma}

\begin{proof}[Proof of Theorem~\ref{thm:circle}]
We prescribe a sequence $\{e_n\}_{n\in\N_0}$ recursively so that
$\deg F_n\le e_n$ for all $n\in\N_0$. Set $e_0\=1$. By Subsection~\ref{ss:inductive construction}, $\deg F_0=0\le e_0$. Now assume that $\deg F_{n-1}\le e_{n-1}$ for some $n\in \N$.
By Definition~\ref{d:fourier detector}, for each $k\in\N$ and each $r\in\{2k-1,\,2k\}$, the nonzero indices in the cosine expansion of $K_{r,N}$ in \eqref{eq:cosine representation} are of the form $\Absbig{d^j-d^\ell}k$ with integers $0\le\ell<j<N$. For such $k$ and $r$, we have $k\le r$, while $\Absbig{d^j-d^\ell}\le2\abs{d}^{N-1}\le\abs{d}^N$
since $\abs{d}\ge2$; hence
$\deg K_{r,N}\le r\abs{d}^{N}$
and so $\deg \cD_n \le(n+1)\abs{d}^{2^{n+2}}$ by Subsection~\ref{ss:parameter pre}, \eqref{eq:defi of S_rho}, \eqref{eq:defi of R N}, and Lemma~\ref{lem:circle degree estimate}.
It follows from \eqref{eq:from Fn-1 to Hn}, Lemma~\ref{lem:circle degree estimate}~(i), and the induction hypothesis that
\begin{equation}\label{eq:deg of H_n}
\deg H_n\le
\max\bigl\{e_{n-1},\,(n+1)\abs{d}^{2^{n+2}}\bigr\}.
\end{equation}
Using \eqref{eq:bernstein approximation},
\eqref{eq:Bernstein degree},
and Lemma~\ref{lem:circle degree estimate}, we obtain
\begin{equation}\label{eq:deg of Bernstein poly}
\deg \bigl(\cB_{d_n,\oB_n}(S_{M_n}H_n/M_n)\bigr)
\le e_n,
\end{equation}
where we define 
\begin{equation*}
e_n\= d_n \abs{d}^{M_n}
\max\bigl\{e_{n-1},\,(n+1)\abs{d}^{2^{n+2}}\bigr\}.
\end{equation*}
By \eqref{eq:recursion construction}, \eqref{eq:deg of H_n}, \eqref{eq:deg of Bernstein poly}, and Lemma~\ref{lem:circle degree estimate}~(i),
$\deg F_n\le e_n$, which completes the induction.

For each $n\in \N$, the definition of $\Delta_n$ and Lemma~\ref{lem:circle degree estimate}~(i) give $\deg \Delta_n \le \max\{\deg F_{n},\,\deg F_{n-1}\} \le\max\{e_n,\, e_{n-1}\}= e_n$. Moreover, $\Delta_n$ has zero Haar mean by
Proposition~\ref{p:recursive construction}. Thus the nonnegative Fourier support of $\Delta_n$ is contained in $\{1,\,2,\,\ldots,\,e_n\}$. Since composition with $T^J$ multiplies every positive Fourier frequency by $\abs{d}^J$, the nonnegative Fourier support of $\Delta_n\circ T^J$ is contained in $\bigl[\abs{d}^J,\abs{d}^Je_n\bigr]$.
We choose
\begin{equation*}
J_{1,1}\=0,\quad
J_{n,s+1}\=J_{n,s}+1+\lceil \log_{\abs{d}}(e_n)\rceil,\quad
\text{ and }J_{n,K_n+1}\= J_{n+1,1},
\end{equation*}
where $n$, $s\in \N$ and $1\le s\le K_n$. The integers $J_{n,s}$
are defined recursively by the formulas above, so Proposition~\ref{p:exact coefficients} applies.
With the above choice of the $J_{n,s}$,
we have
$\abs{d}^{J_{n,s+1}}>e_n \abs{d}^{J_{n,s}}$: the relocated increment $\widetilde \Delta_{n,s+1}$ defined in \eqref{eq:spread-increment} (respectively
$\widetilde\Delta_{n+1,1}$, when $s= K_n$) has nonnegative Fourier support contained in $\bigl[\abs{d}^{J_{n,s+1}},+\infty\bigr)$, strictly beyond the nonnegative Fourier support of every relocated increment already placed, which is contained in
$\bigl[0,\,e_n \abs{d}^{J_{n,s}}\bigr]$.
By \eqref{eq:spread-increment} and \eqref{eq:defi of E_n},
$A^{J_{n,s}}E_n$ is precisely the frequency support of
$\widetilde\Delta_{n,s}$.
By \eqref{eq:enumeration for circle}, the pairs of characters corresponding to $\{k,\,-k\}$ are ordered by increasing $k$; it follows that the
blocks $A^{J_{n,s}}E_n$, $n\in\N$, $1\le s\le K_n$, are pairwise disjoint, and that
consecutive blocks $A^{J_{n,s}}E_n$ in the lexicographic order of $(n,s)$ occur strictly successively with
respect to \eqref{eq:enumeration}. 
Thus \ref{B1} and \ref{B2} hold, so Proposition~\ref{p:abstract uniform fourier} applies:
$\widetilde F=\sum_{r=1}^{+\infty}\beta_r\gamma_r$ with $\beta_r\=\wh{\widetilde F}(\gamma_r)\in\Q$,
the partial sums converging uniformly in the enumeration order.

Each $\widetilde\Delta_n$ lies in $\cF_\Q$ by Lemma~\ref{lem:rational fourier closure},
so $\wh{\widetilde F}(-k)=\wh{\widetilde F}(k)\in\Q$ for all $k\in\N$. 
By our choice of enumeration, for all $k\in \N$, $\beta_{2k-1}=\wh{\widetilde F}(k)$ and $\beta_{2k}=\wh{\widetilde F}(-k)$. 
Set $c_k\=2\wh{\widetilde F}(k)\in\Q$. Then $c_k\cos(2\pi kx)=\beta_{2k-1}\gamma_{2k-1}(x)+\beta_{2k}\gamma_{2k}(x)$. Hence for all $K\in \N$,
the partial sum $\sum_{r=1}^{2K}\beta_r\gamma_r$ of
Proposition~\ref{p:abstract uniform fourier} is equal to $\sum_{k=1}^{K}c_k\cos(2\pi kx)$; being
a subsequence of a uniformly convergent sequence, these cosine partial sums converge
uniformly to $\widetilde F$. Together with
Proposition~\ref{p:tilde F exposure}, this proves Theorem~\ref{thm:circle}.
\end{proof}

\subsubsection{The one-sided full shift: proof of Theorem~\ref{thm:shift}}

Here we prove Theorem~\ref{thm:shift}.
We take the enumeration in \eqref{eq:enumeration} to be the Walsh
ordering of \eqref{eq:walsh ordering}, namely $\gamma_r\=w_r$ for $r\in\N$. Since each $w_r$ is
real-valued with $w_r^{-1}=w_r$, every $P\in\cF_\Q$ has the form
$P=\sum_{k=0}^{N}a_kw_k$ with $a_0$, $a_1$, $\ldots$, $a_N\in\Q$.
It will be convenient to define
\begin{equation*}
\operatorname{depth} P\=
\min\{p\in\N_0 :
P(x)=P(y)\text{ whenever }x,\,y\in X
\text{ and }x_j=y_j\text{ for }0\le j<p\}.
\end{equation*}
We require the following simple lemma, whose proof is left to the reader.

\begin{lemma}
\label{lem:shift depth estimate}
For $P$, $Q\in\cF_{\Q}$, $R\in\Q[t]\smallsetminus\{0\}$, $j\in\N_0$, and the one-sided
shift $T$ on $X\=(\Z/2\Z)^{\N_0}$, the
following hold:
\begin{enumerate}[label=\rm{(\roman*)}]
\smallskip
\item $\operatorname{depth}(P+Q)\le
\max\{\operatorname{depth} P,\,\operatorname{depth} Q\}$;
\smallskip
\item $\operatorname{depth}(P\cdot Q)\le
\max\{\operatorname{depth}P,\,\operatorname{depth}Q\}$;
\smallskip
\item $\operatorname{depth}\bigl(P\circ T^j \bigr)\le
j+\operatorname{depth}P$;
\smallskip
\item $\operatorname{depth}(R\circ P)\le
\operatorname{depth}P$.
\end{enumerate}
\end{lemma}

\begin{proof}[Proof of Theorem~\ref{thm:shift}]
We prescribe a sequence $\{e_n\}_{n\in\N_0}$ recursively so that $\operatorname{depth} F_n\le e_n$ for all $n\in\N_0$. Set $e_0\=1$. Then by Subsection~\ref{ss:inductive construction}, $\operatorname{depth} F_0=0\le e_0$. Now assume that $\operatorname{depth} F_{n-1}\le e_{n-1}$ for some $n\in \N$.
By the definition of the Walsh ordering, $\operatorname{depth}w_r=1+\lfloor\log_2r\rfloor$ for $r\in\N$. It follows from Definition~\ref{d:fourier detector} and Lemma~\ref{lem:shift depth estimate} that for all $r$, $N\in \N$,
$\operatorname{depth}K_{r,N}\le N+\lfloor\log_2r\rfloor$, and so $\operatorname{depth}\cD_n\le
2^{n+2}+\lfloor\log_2(n+1)\rfloor$ by Subsection~\ref{ss:parameter pre}, \eqref{eq:defi of S_rho}, \eqref{eq:defi of R N}, and Lemma~\ref{lem:shift depth estimate}. It follows from \eqref{eq:from Fn-1 to Hn}, Lemma~\ref{lem:shift depth estimate}~(i), and the induction hypothesis that 
\begin{equation}\label{eq:depth of H_n}
\operatorname{depth} H_n\le \max\bigl\{e_{n-1},\,
2^{n+2}+\lfloor\log_2(n+1)\rfloor\bigr\}.
\end{equation}
Using \eqref{eq:bernstein approximation},
\eqref{eq:Bernstein degree},
and Lemma~\ref{lem:shift depth estimate}, we obtain
\begin{equation}\label{eq:depth of Bernstein poly}
\operatorname{depth} \bigl(\cB_{d_n,\oB_n}(S_{M_n}H_n/M_n)\bigr)\le e_n,
\end{equation}
where we define $e_n\= \max\bigl\{e_{n-1},\,
2^{n+2}+\lfloor\log_2(n+1)\rfloor\bigr\}+ M_n -1$.
By \eqref{eq:recursion construction}, \eqref{eq:depth of H_n}, \eqref{eq:depth of Bernstein poly}, and Lemma~\ref{lem:shift depth estimate}~(i), $\operatorname{depth} F_n\le e_n$, which completes the induction.

By the definition of $\Delta_n$ and Lemma~\ref{lem:shift depth estimate}~(i),
$\operatorname{depth} \Delta_n\le
\max\{\operatorname{depth} F_n,\,\operatorname{depth} F_{n-1}\}\le e_n$ for each $n\in \N$.
Moreover, a Walsh polynomial of depth at most $e_n$ involves only
$w_k$ with $0\le k\le 2^{e_n}-1$. Since $\Delta_n$ has zero Haar mean
for each $n\in\N$ by Proposition~\ref{p:recursive construction}, its frequency support is contained in $\{w_k:\,k\in\N,\,1\le k\le 2^{e_n}-1\}$.
It follows directly from the definition that $w_k\circ T^J=w_{2^Jk}$ for all $k$, $J\in \N_0$. Thus for $n\in \N$, the frequency support of $\Delta_n\circ T^J$ is contained in
$\bigl\{w_k:2^J\le k\le(2^{e_n}-1)2^J \bigr\}$. We choose
\begin{equation*}
J_{1,1}\=0,\quad
J_{n,s+1}\=J_{n,s}+e_n,\quad
\text{ and }J_{n,K_n+1}\= J_{n+1,1},
\end{equation*}
where $n$, $s\in \N$ and $1\le s\le K_n$.
The integers $J_{n,s}$ are defined recursively by the formulas above, so Proposition~\ref{p:exact coefficients} applies.
With this choice, $2^{J_{n,s+1}}>(2^{e_n}-1)2^{J_{n,s}}$ for each $n\in \N$ and each integer $1\le s \le K_n$: the block $A^{J_{n,s+1}}E_n$
(respectively $A^{J_{n+1,1}}E_{n+1}$, when $s= K_n$) has Walsh-frequency support in
$\bigl\{w_k:k\in \N,\, k\ge 2^{J_{n,s+1}}\bigr\}$, strictly beyond the Walsh-frequency support in
$\bigl\{w_k:k\in \N,\,1\le k\le(2^{e_n}-1)2^{J_{n,s}}\bigr\}$ of every block already placed. Since 
$\gamma_r=w_r$ for all $r\in\N$, it follows that the blocks
$A^{J_{n,s}}E_n$, $n\in\N$, $1\le s\le K_n$, are pairwise disjoint, and that consecutive blocks $A^{J_{n,s}}E_n$ in the lexicographic order of $(n,s)$ occur strictly successively with respect to
\eqref{eq:enumeration}. Thus \ref{B1} and \ref{B2} hold, and Theorem~\ref{thm:shift} follows
from Propositions~\ref{p:abstract uniform fourier} and~\ref{p:tilde F exposure}.
\end{proof}

\subsubsection{Ergodic toral endomorphisms: proof of Theorem~\ref{thm:torus}}

Here we prove Theorem~\ref{thm:torus}.
The dual group of $\T^m$ is identified with $\Z^m$,
the frequency $k\in\Z^m$ corresponding to the character
$x\mapsto\exp(2\pi\I\langle k,x\rangle)$, and composition with $T_B$ maps the character with
frequency $k$ to the character with frequency $\transpose{B}k$.
We choose the enumeration in
\eqref{eq:enumeration} by ordering $\Z^m\smallsetminus\{0\}$ first by increasing
$\norm{k}_\infty$ and then lexicographically,
and write $k_r$ for the frequency of
$\gamma_r$. Unlike in the two settings above, $\NormBig{\transposeparen{B}^{J}k}_\infty$
need not increase with $J$ if $\transpose{B}$ has a contracting direction, so we keep the
integers $J_{n,s}$ chosen in \eqref{eq:fourier box recursion}.

\begin{proof}[Proof of Theorem~\ref{thm:torus}]
For $R\in\N_0$ set $L(R)\=(2R+1)^m-1$, the number of $k\in\Z^m\smallsetminus\{0\}$ with
$\norm{k}_\infty\le R$.
Since $L(r)\ge r$, the enumeration satisfies
$\norm{k_r}_\infty\le r$, so $k_r$ is found from $r$ by listing the finitely many
$k\in\Z^m$ with $\norm{k}_\infty\le r$ in the prescribed order.
Equality of characters is
equality in $\Z^m$, and $A$ acts on frequencies by the integer matrix $\transpose B$.
Consequently the integers $J_{n,s}$ of \eqref{eq:fourier box recursion}
are produced by
finite procedures, and Propositions~\ref{p:abstract uniform fourier},
\ref{p:tilde F exposure}, and~\ref{p:exact coefficients} apply.
Since the enumeration is
ordered by increasing $\ell^\infty$-norm, for each $R\in\N_0$ the square partial sum
$\Sigma_R\widetilde F$ of \eqref{eq:T^d partial sum} coincides with $\sum_{r=1}^{L(R)}c_r\gamma_r$, where the $c_r\in\Q$ are given by
Proposition~\ref{p:abstract uniform fourier}.
The square partial sums therefore form a
subsequence of the partial sums in the prescribed enumeration, and converge uniformly to
$\widetilde F$. Together with Propositions~\ref{p:tilde F exposure}
and~\ref{p:exact coefficients}, this proves Theorem~\ref{thm:torus}, the coefficient of
$\exp(2\pi\I\langle k_r,x\rangle)$ being $c_r$.
\end{proof}

\section{Riesz products and the Wiener algebra obstruction}\label{s:Wiener}

In this section we prove Theorem~\ref{thm:wiener obstruction}. For
$f\in C(G,\R)\cap A(G)$, the integral $\int\!f\,\mathrm d\nu$ depends on
$\nu\in\cM(G,T)$ only through the Fourier coefficients of $\nu$ at one representative
character from each $A$-orbit in $\Gamma$, weighted by the corresponding orbit sums of
$\wh f$, defined in \eqref{eq:orbit sum}; this is \eqref{eq:orbit decomposition} below. If
the orbit sums over all nontrivial $A$-orbits vanish, then $\int\!f\,\mathrm d\nu=\wh f(1)$
for all $\nu\in\cM(G,T)$, so that $f-\wh f(1)$ is a weak coboundary. It therefore suffices
to show that if the orbit sum over some nontrivial $A$-orbit is nonzero, then $m_G$ is not
$f$-maximizing; this is Proposition~\ref{p:frequency orbit estimate}. Its proof constructs
an invariant measure whose Fourier coefficient at that orbit has a prescribed phase and
dominates all its other Fourier coefficients. The tool for this is the Riesz product, which
originated in the work of Riesz~\cite{Ri18} on the circle and was extended to compact
abelian groups by Hewitt and Zuckerman~\cite{HZ66}: built on a dissociate sequence of
characters, here supplied by Lemma~\ref{lem:dissociate}, it is a measure whose Fourier
coefficients are given explicitly. We follow the formulation in \cite[Section~4]{FSV24}.

\begin{definition}[Dissociate sequences]\label{defi: dissociativity}
An infinite sequence of distinct characters $\Lambda=\{\gamma_n\}_{n \in \N_0}$ in $\Gamma$ is said to be \emph{dissociate} if for each $k \in \N$ and all $k$-tuples $(n_1, n_2, \ldots, n_k) \in \N_0^k$ of distinct nonnegative integers, the equality
\begin{equation*}
\gamma_{n_1}^{\varepsilon_1} \gamma_{n_2}^{\varepsilon_2} \cdots \gamma_{n_k}^{\varepsilon_k}=1,
\end{equation*}
with $\varepsilon_j \in\{-2,\,-1,\,0,\,1,\,2\}$ for all $1\le j\le k$, implies that $\gamma_{n_1}^{\varepsilon_1}=\gamma_{n_2}^{\varepsilon_2}=\cdots=\gamma_{n_k}^{\varepsilon_k}=1$.
Equivalently, $\Lambda$ is dissociate if every character in $\Gamma$ can be represented in at most one way as a finite product $\gamma_{n_1}^{\tau_1}\gamma_{n_2}^{\tau_2}\cdots\gamma_{n_k}^{\tau_k}$ of elements of $\Lambda$, where the integers $n_j$ are distinct and $\tau_j\in\{-1,\,0,\,1\}$, except that the exponents $1$ and $-1$ are identified whenever $\gamma_{n_j}^2=1$. Representations are understood up to reordering the factors and omission of identity factors. The empty product represents $1$.
\end{definition}

In the following, the real part of a complex number $z$ is denoted by $\Re (z)$. We recall the standard Riesz product construction; see \cite[Theorem~3.2]{HZ66}.

\begin{lemma}[Riesz product]\label{lem:Riesz product}
Let $\{\gamma_n\}_{n\in\N_0}$ be a dissociate sequence of nontrivial characters in $\Gamma$. Let $\{a_n\}_{n\in\N_0}$ be a sequence in $\C$ satisfying
$\abs{a_n}\le1$ for all $n\in\N_0$. For $N\in\N_0$, let $\mu^{(N)}$ be the measure on $G$ that is absolutely continuous with respect to Haar measure $m_G$ with density
\begin{equation*}
\frac{\mathrm d \mu^{(N)}}{\mathrm d m_G}(x)
=\prod_{n=0}^N(1+\Re (a_n \gamma_n(x))).
\end{equation*}
Then $\bigl\{\mu^{(N)}\bigr\}_{N\in \N_0}$ is a sequence of probability measures converging in the weak$^*$ topology to some probability measure $\mu$ as $N\to +\infty$.
\end{lemma}

\begin{rem}
In the setting of Lemma~\ref{lem:Riesz product},
the limiting measure $\mu$ is called the Riesz product, and, for $a=\{a_n\}_{n\in \N_0}$, we write formally
\begin{equation*}
\mu_a\=\prod_{n=0}^{+\infty}(1+\Re (a_n \gamma_n(x))).
\end{equation*}
\end{rem}

For the construction in Proposition~\ref{p:frequency orbit estimate}, we need, for each nontrivial $\gamma\in\Gamma$, a set $S\subseteq\N_0$ of positive lower density such that
$\{A^m\gamma\}_{m\in S}$ is dissociate. In fact, $S$ may be chosen to be an arithmetic progression $L\N_0$ for some $L\in \N$.

\begin{lemma}[Dissociate subsequence]
\label{lem:dissociate}
Let $\gamma\in\Gamma$ be a nontrivial character. Then there exists $L\in\N$ such that the sequence
$\bigl\{A^{Ln}\gamma\bigr\}_{n\in\N_0}$ is dissociate.
\end{lemma}

\begin{proof}
We deduce the lemma from a theorem of Lind~\cite{Li82}. We first recall the relevant terminology, in multiplicative notation and in the form needed here. Let $H$ be an abelian group with identity $1$.
A sequence $\{E_n\}_{n\in\N_0}$ of subsets of $H$ is \emph{independent} if the equality $e_1e_2\cdots e_k=1$, where $k\in\N$ and $e_j\in E_{n_j}$ for pairwise distinct $n_1$, $n_2$, $\ldots$, $n_k\in\N_0$, implies that $e_1=e_2=\cdots=e_k=1$. An automorphism $\theta$ of $H$ is \emph{aperiodic} if 
for $p\in \N$ and $h\in H$, the equality
$\theta^ph=h$ implies that $h=1$. An aperiodic automorphism $\theta$ of $H$ is \emph{lacunarily independent} if for each finite set $E\subseteq H$, there exists $L\in\N$ such that the sequence $\{\theta^{m_n}E\}_{n\in\N_0}$ is independent whenever $\{m_n\}_{n\in\N_0}$ is a sequence of integers with $m_{n+1}-m_n\ge L$ for all $n\in\N_0$. Lind proved that every aperiodic automorphism of a countable abelian group is lacunarily independent; see~\cite[Theorem~2]{Li82}.

By Lemma~\ref{lem:distinct character}, $A$ is injective and $1$ is its only periodic point.
Since $A$ need not be surjective, 
we first embed $\Gamma$ in a larger group to which $A$ extends as an automorphism.
For $(\chi,n)$, $(\chi',n')\in\Gamma\times\N_0$, write $(\chi,n)\sim(\chi',n')$ if $A^{n'}\chi=A^{n}\chi'$. This is an equivalence relation, transitivity following from the injectivity of $A$. Let $\widetilde\Gamma\=(\Gamma\times \N_0)/\sim$ denote the set of equivalence classes, and write $[\chi,n]$ for the class of $(\chi,n)\in \Gamma\times \N_0$, which should be thought of as $A^{-n}\chi$. One checks directly that
\begin{equation*}
[\chi,n]\cdot[\chi',n']\=\bigl[A^{n'}\chi\cdot A^{n}\chi',\,n+n'\bigr]
\end{equation*}
is a well-defined operation making $\widetilde\Gamma$ an abelian group, with identity $[1,0]$ and $[\chi,n]^{-1}=\bigl[\chi^{-1},n\bigr]$. Similarly, $\widetilde A[\chi,n]\=[A\chi,n]$ defines an automorphism $\widetilde A$ of $\widetilde\Gamma$, with inverse $[\chi,n]\mapsto[\chi,n+1]$.
Define $\iota\:\Gamma\to \widetilde\Gamma$ by $\iota(\chi)\=[\chi,0]$. Then $\iota$ is a homomorphism such that $\widetilde A\circ\iota=\iota\circ A$. Moreover, $\iota$ is injective, as $(\chi,0)\sim(\chi',0)$ only if $\chi=\chi'$ by definition of $\sim$. In other words, $\widetilde\Gamma$ is the direct limit of $\Gamma\xrightarrow{A}\Gamma\xrightarrow{A}\cdots$. Dually, passing from $A$ to $\widetilde A$ corresponds to passing from $T$ to its natural extension. Since $\Gamma$ is countable (see the discussion before \eqref{eq:enumeration}),
the quotient $\widetilde\Gamma=(\Gamma\times\N_0)/{\sim}$ is at most countable, and is in fact countable since the embedding $\iota:\Gamma\to\widetilde\Gamma$ exhibits the infinite subset $\iota(\Gamma)$ of $\widetilde\Gamma$.
Moreover, $\widetilde A$ is aperiodic: if $\widetilde A^p[\chi,n]=[\chi,n]$ for some $p\in\N$ and $[\chi,n]\in \widetilde\Gamma$, then $A^{n+p}\chi=A^{n}\chi$, so $\chi=1$ by Lemma~\ref{lem:distinct character}, and hence $[\chi,n]=[1,0]$.

Now set $E\=\{\iota(\gamma^{\varepsilon}):\varepsilon\in\{-2,\,-1,\,0,\,1,\,2\}\}$. By Lind's theorem~\cite[Theorem~2]{Li82}, $\widetilde A$ is lacunarily independent, so there exists $L\in\N$ such that $\bigl\{\widetilde A^{Ln}E\bigr\}_{n\in\N_0}$ is independent. Let $k\in\N$, let $n_1$, $n_2$, $\ldots$, $n_k\in\N_0$ be pairwise distinct, and let $\varepsilon_1$, $\varepsilon_2$, $\ldots$, $\varepsilon_k\in\{-2,\,-1,\,0,\,1,\,2\}$ satisfy $\prod_{j=1}^k\bigl(A^{Ln_j}\gamma\bigr)^{\varepsilon_j}=1$.
For each integer $1\le j\le k$, $e_j\=\iota\bigl(\bigl(A^{Ln_j}\gamma\bigr)^{\varepsilon_j}\bigr)$ is equal to $\widetilde A^{Ln_j}\iota(\gamma^{\varepsilon_j}) \in \widetilde A^{Ln_j}E$, because $A$ and $\iota$ are homomorphisms with $\iota\circ A=\widetilde A\circ\iota$. Moreover, $e_1e_2\cdots e_k=\iota(1)=[1,0]$.
By independence, each $e_j$ equals $[1,0]=\iota(1)$, and the injectivity of $\iota$ gives $\bigl(A^{Ln_j}\gamma\bigr)^{\varepsilon_j}=1$. As $\gamma\neq1$, the characters $A^{Ln}\gamma$, $n\in\N_0$, are nontrivial and pairwise distinct by
 Lemma~\ref{lem:distinct character}, and so $\bigl\{A^{Ln}\gamma\bigr\}_{n\in\N_0}$ is dissociate.
\end{proof}

We group the characters in $\Gamma$ according to the orbit under iterations of $A$. Namely, write $\gamma\sim\gamma'$ if
$A^m\gamma=A^n\gamma'$ for some $m$, $n\in\N_0$. By injectivity of $A$, this is an equivalence relation on $\Gamma$. 
Let $\sO\=\Gamma/{\sim}$ denote the set of
orbits, and fix a representative $\gamma_{\cO}\in\cO$ for each $\cO\in\sO$. We write $\cO_0\=\{1\}$ for the trivial orbit.
The equivalence relation is compatible with inversion: if $\gamma\sim \gamma'$, then $\gamma^{-1}\sim (\gamma')^{-1}$. Hence for $\cO\in\sO$, we write $\cO^{-1}\=\bigl\{\gamma^{-1}:\gamma\in\cO\bigr\}$, which is again an orbit.
For each pair 
$\cO$, $\cO^{-1}$, with
$\cO\neq\cO^{-1}$, we choose the representatives so that
$\gamma_{\cO^{-1}}=\gamma_{\cO}^{-1}$.

For $f\in A(G)$ and each $\cO\in\sO$, define the \emph{orbit sum}
\begin{equation}\label{eq:orbit sum}
b_{\cO}(f)\=\sum_{\gamma \in \cO}\wh f(\gamma).
\end{equation}
Since
$\abs{b_\cO(f)}\le \sum_{\gamma \in \cO}\Absbig{\wh f(\gamma)}$, we have
$\sum_{\cO\in\sO}\abs{b_\cO(f)}<+\infty$.
Hence for all $\nu\in\cM(G,T)$, absolute convergence and \eqref{eq:fourier invariance} give
\begin{equation}\label{eq:orbit decomposition}
\int \!f\,\mathrm{d}\nu
=\sum_{\gamma\in\Gamma} \wh f(\gamma)
\int \!\gamma\,\mathrm{d}\nu
=\wh f(1) +
\sum_{\cO\in\sO\smallsetminus\{\cO_0\}}b_{\cO}(f) \, \wh \nu\bigl(\gamma_\cO^{-1}\bigr).
\end{equation}
The following result implies that if $m_G$ is $f$-maximizing, then
$b_{\cO}(f)=0$ for all nontrivial orbits.

\begin{prop}[A nonvanishing orbit sum prevents Haar measure from being maximizing]\label{p:frequency orbit estimate}
Let $f\in C(G,\R)\cap A(G)$. Suppose $\cO\in\sO$ is a nontrivial orbit with
$b_{\cO}(f)\neq 0$. Then there exists $\nu\in\cM(G,T)$ such that
\begin{equation*}
\int \! f\,\mathrm dm_G < \int \!f\,\mathrm d\nu.
\end{equation*}
\end{prop}

\begin{proof}
We distinguish two cases according to whether $\cO=\cO^{-1}$ or not.

\smallskip
\emph{Case 1.} $\cO\neq \cO^{-1}$.
By Lemma~\ref{lem:dissociate}, choose $L\in \N$ such that
$\bigl\{A^{LN}\gamma_\cO\bigr\}_{N\in \N_0}$ is dissociate, and set
$\sigma\= b_{\cO}(f)/\abs{b_{\cO}(f)}$. For $t\in(0,1]$, let $\mu_{t,\cO}$ be the Riesz product given by Lemma~\ref{lem:Riesz product},
\begin{equation*}
\mu_{t,\cO}
\=\prod_{N=0}^{+\infty}\bigl(1+\Re\bigl(t\sigma A^{LN}\gamma_\cO\bigr)\bigr)
=\prod_{N=0}^{+\infty}
\Bigl(1+\frac{t\sigma}{2} A^{LN}\gamma_\cO
+\frac{t\overline\sigma}{2} A^{LN}\gamma_\cO^{-1}\Bigr).
\end{equation*}
Since $\cO\neq\cO^{-1}$, we have $\gamma_\cO\neq\gamma_\cO^{-1}$. By Lemma~\ref{lem:distinct character}, $A$ is injective and hence $A^{LN}\gamma_\cO\ne A^{LN}\bigl(\gamma_\cO^{-1}\bigr)=\bigl(A^{LN}\gamma_\cO\bigr)^{-1}$
for each $N\in \N_0$; that is, no character in $\bigl\{A^{LN}\gamma_\cO\bigr\}_{N\in\N_0}$ is equal to its own inverse, and
we may apply \cite[(4.2), (4.3)]{FSV24} with coefficients $a_N\=t\sigma$ for
all $N\in \N_0$ (see also \cite[Theorem~3.2]{HZ66}). This gives, for all $N\in \N_0$,
\begin{equation}\label{eq:Fourier coefficient for LN}
\wh \mu_{t,\cO}(1)=1,\quad
\wh\mu_{t,\cO}\bigl(A^{LN}\gamma_\cO\bigr)=t\sigma/2,\quad\text{ and }\quad
\wh\mu_{t,\cO}\bigl(A^{LN}\gamma_\cO^{-1}\bigr)=t\overline{\sigma}/2.
\end{equation}
Any character 
$\gamma\notin\{1\}\cup\bigcup_{N\in \N_0}\bigl\{A^{LN}\gamma_\cO,\,A^{LN}\gamma_\cO^{-1}\bigr\}$ is either not of the form appearing in \cite[(4.2)]{FSV24}, in which case $\wh\mu_{t,\cO}(\gamma)=0$ by \cite[(4.3)]{FSV24}, or is of that form with at least two nonzero exponents. In either case,
\begin{equation}
\label{eq:not in O}
\abs{\wh \mu_{t,\cO}(\gamma)}\le t^2\bigl/4.
\end{equation}

For $M\in \N$, set
$\nu_M\= \frac1M\sum_{j=0}^{M-1}\mu_{t,\cO}\circ T^{-j}$.
It follows from the usual Ces\`aro averaging argument that there exists a strictly increasing sequence $\{M_k\}_{k\in\N}$ such that $\nu_{M_k}$ converges in the weak$^*$ topology to a $T$-invariant measure $\nu\in\cM(G,T)$ as $k\to +\infty$ (cf.~\cite[Theorem~6.9]{Walt82}).

For $m\in \N_0$ with $L\nmid m$, the character $A^m\gamma_\cO$ cannot be
$A^{LN}\gamma_\cO$ for some $N\in \N_0$ by Lemma~\ref{lem:distinct character}, and cannot be $A^{LN}\gamma_{\cO^{-1}}$ since $\cO\neq\cO^{-1}$. Hence, by \eqref{eq:Fourier coefficient for LN} and \eqref{eq:not in O},
\begin{equation}\label{eq:gamma_cO}
\wh\nu_{M}(\gamma_\cO)
= \frac{1}{M}\biggl(
\sum_{\substack{L\mid m \\ 0 \le m \le M-1}} \frac{t\sigma}{2}\biggr)
+\frac{1}{M}\biggl(\sum_{\substack{L\, \nmid\, m \\ 0 \le m \le M-1}}
\wh\mu_{t,\cO}(A^{m}\gamma_\cO)\biggr)
=\frac{1}{M}\cdot \biggl\lceil\frac{M}{L}\biggr\rceil
\cdot \frac{t\sigma}{2} + z_M,
\end{equation}
where $z_M\in \C$ and $\abs{z_M}\le \frac{t^2}{4}$ by \eqref{eq:not in O}. In the same way,
$\wh\nu_{M}(\gamma_{\cO^{-1}})
=(1/M)\cdot \lceil M/L\rceil
\cdot t\overline\sigma/2 + z_M',$
where $z_M'\in \C$ and $\abs{z_M'}\le \frac{t^2}{4}$.
For $\gamma\notin\{1\}\cup\cO\cup\cO^{-1}$, each $A^m\gamma$ with $m\in \N_0$ does not belong to any of the equivalence classes $\{1\}$, $\cO$, or $\cO^{-1}$.
Hence by \eqref{eq:not in O},
$\abs{\wh\nu_{M}(\gamma)}
=\Absbig{\frac{1}{M}\sum_{m=0}^{M-1}
\wh \mu_{t,\cO}(A^m \gamma)}
\le \frac{t^2}{4}$.

By weak$^*$ convergence, $\wh\nu_{M_k}(\gamma)\to\wh\nu(\gamma)$ for each $\gamma\in\Gamma$ as $k\to +\infty$.
Hence we have
$\wh\nu(\gamma_\cO)=t\sigma/(2L)+z_0$ and
$\wh\nu(\gamma_{\cO^{-1}})=t\overline\sigma/(2L)+z_0'$ for some $z_0$, $z_0'\in \C$ with $\abs{z_0}$, $\abs{z_0'}\le \frac{t^2}{4}$. Moreover,
$\abs{\wh\nu(\gamma)}\le\frac{t^2}4$ for each
$\gamma\notin\{1\}\cup\cO\cup\cO^{-1}$.

Since $f$ is real-valued, $\wh f(\gamma)=\overline{\wh f(\gamma^{-1})}$ and thus
$b_{\cO^{-1}}(f)=\overline{b_{\cO}(f)}$. Hence,
\begin{align}
\int\! f\,\mathrm d\nu - \int \!f\,\mathrm dm_G
&=
\sum_{\cO'\in\sO\smallsetminus\{\cO_0\}}b_{\cO'}(f)\wh \nu\bigl(\gamma_{\cO'}^{-1}\bigr) \notag\\
&=
b_{\cO}(f)\Bigl(\frac{t\overline\sigma}{2L}+z_0'\Bigr)
+b_{\cO^{-1}}(f)\Bigl(\frac{t\sigma}{2L}+z_0\Bigr)
+\sum_{\substack{\cO'\in \sO\smallsetminus\{\cO,\,\cO^{-1},\,\cO_0\}}} b_{\cO'}(f)\wh \nu\bigl(\gamma_{\cO'}^{-1}\bigr) \notag\\
&=
\frac{t}{L}\abs{b_\cO(f)}
+b_{\cO}(f)z_0'
+b_{\cO^{-1}}(f)z_0
+\sum_{\substack{\cO'\in \sO\smallsetminus\{\cO,\,\cO^{-1},\,\cO_0\}}} b_{\cO'}(f)\wh \nu\bigl(\gamma_{\cO'}^{-1}\bigr) \\
&\ge
\frac{t}{L}\abs{b_\cO(f)}
-\frac{t^2}{4}
\sum_{\cO'\in\sO\smallsetminus\{\cO_0\}}\abs{b_{\cO'}(f)}.\notag
\end{align}
Since $b_\cO(f)\neq0$, the right-hand side is positive for all sufficiently small $t>0$.

\smallskip
\emph{Case 2.} $\cO=\cO^{-1}$. We first note that $\gamma_{\cO}^2=1$. Indeed, there exist $m$, $n\in\N_0$ such that
$A^m\gamma_\cO=A^n \gamma_\cO^{-1}$. If $m>n$, the injectivity of $A^n$ gives $A^{m-n}\gamma_\cO=\gamma_\cO^{-1}$, and hence $A^{2(m-n)}\gamma_\cO=\gamma_\cO$, contradicting Lemma~\ref{lem:distinct character}. The case $n>m$ is analogous. Thus $m=n$, and the injectivity of
$A^m$ gives $\gamma_\cO^2=1$.

Since $f$ is real-valued and $\cO=\cO^{-1}$, we have $b_\cO(f)\in\R$.
By Lemma~\ref{lem:dissociate}, choose $L\in\N$ such that $\bigl\{A^{LN}\gamma_\cO\bigr\}_{N\in\N_0}$ is dissociate, and set $\sigma\=b_\cO(f)/\abs{b_\cO(f)}\in\{-1,\,1\}$. For $0<t\le 1$, we now take
\begin{equation*}
\mu_{t,\cO}
\=\prod_{N=0}^{+\infty}
\bigl(1+ \Re \bigl(t\sigma A^{LN}\gamma_\cO\bigr)\bigr)
=\prod_{N=0}^{+\infty}
\bigl(1+t\sigma A^{LN}\gamma_\cO\bigr).
\end{equation*}
Here $\gamma_\cO^2=1$, so also $\bigl(A^{LN}\gamma_\cO\bigr)^2=1$ for each $N\in \N_0$; moreover,
$A^{LN}\gamma_\cO\neq1$ by Lemma~\ref{lem:distinct character}. Since $0<t\le1$ and $\sigma\in\{-1,\,1\}$, we have $t\sigma\in[-1,1]$. We may therefore apply \cite[Theorem~3.2]{HZ66}, with the coefficient function $\beta$ in the notation of that theorem given by $\beta\bigl(A^{LN}\gamma_\cO\bigr)=t\sigma$ for all $N\in\N_0$.
As in Case~1,
this gives
$\wh\mu_{t,\cO}(1)=1$ and $\wh\mu_{t,\cO}\bigl(A^{LN}\gamma_\cO\bigr)=t\sigma$
for each $N\in \N_0$, while $\abs{\wh\mu_{t,\cO}(\gamma)}\le t^2$ for each
$\gamma\notin\{1\}\cup\bigl\{A^{LN}\gamma_\cO:N\in\N_0\bigr\}$.

Repeating the Ces\`aro averaging argument following \eqref{eq:not in O}, with $t\sigma$ in place of $t\sigma/2$ and
$t^2$ in place of $t^2/4$, gives a measure
$\nu\in\cM(G,T)$ such that
$\wh\nu(\gamma_\cO)=(t/L)\sigma+z_0$, with
$\abs{z_0}\le t^2$, and
$\abs{\wh\nu(\gamma)}\le t^2$ for
$\gamma\notin\{1\}\cup\cO$. Hence,
\begin{equation*}
\int\! f\,\mathrm d\nu-\int\! f\,\mathrm dm_G
\ge
\frac{t}{L}\abs{b_\cO(f)}
-t^2\sum_{\cO'\in\sO\smallsetminus\{\cO_0\}}\abs{b_{\cO'}(f)}.
\end{equation*}
Since $b_\cO(f)\neq0$,
the right-hand side is positive for all sufficiently small $t>0$.
\end{proof}

We can now prove Theorem~\ref{thm:wiener obstruction}.

\begin{proof}[Proof of Theorem~\ref{thm:wiener obstruction}]
Suppose $m_G$ is $F$-maximizing. By
Proposition~\ref{p:frequency orbit estimate},
$b_{\cO}(F)=0$ for all nontrivial orbits $\cO$. Hence
\eqref{eq:orbit decomposition} gives
$\int\!F\,\mathrm d\nu=\wh F(1)=\int\!F\,\mathrm dm_G=Q(T,F)$ for each $\nu\in\cM(G,T)$. Therefore
$\int\!(F-Q(T,F))\,\mathrm d\nu=0$ for all
$\nu\in\cM(G,T)$, so $F-Q(T,F)$ is a weak coboundary.
\end{proof}

\begin{rem}\label{rem:walters comparison}
When $T$ is expanding, as for the doubling map, the Ma\~n\'e lemma and the full support
of $m_G$ show that the conclusion of Theorem~\ref{thm:wiener obstruction} holds for Walters
functions $F$, in the stronger form that $F-Q(T,F)$ is a coboundary; see \cite{Bou01}, where
the Walters condition is also defined. 
On the one hand the Wiener algebra obstruction is not a consequence of the
Walters-class obstruction, and on the other hand the conclusion
of Theorem~\ref{thm:wiener obstruction}
is weaker:
Proposition~\ref{p:wiener not walters} below exhibits, for the doubling
map, an $F\in A(\T)$ that is a weak coboundary but not a coboundary; since $m_\T$ is
$F$-maximizing with full support,
the Ma\~n\'e lemma
then shows that $F$ is
not Walters.
\end{rem}

\begin{prop}[Sharpness of Theorem~\ref{thm:wiener obstruction}]
\label{p:wiener not walters}
Let $T\:\T\to\T$ be given by $T(x)=2x\pmod1$, set $\lambda_0\=-1$ and
$\lambda_m\=m^{-1/2}-(m+1)^{-1/2}$ for $m\in\N$, and define
\begin{equation}
\label{eq:wiener example}
F(x)\=\sum_{n=0}^{+\infty}\frac1{2^n}
\biggl(-\cos(2\pi(2n+1)x)
+\sum_{m=1}^{+\infty}\lambda_m\cos(2\pi(2n+1)2^mx)\biggr).
\end{equation}
Then the following hold:
\begin{enumerate}[label=\textup{(\roman*)}]
\smallskip
\item $F\in C(\T,\R)\cap A(\T)$, $\sum_{k\in\Z}\Absbig{\wh F(k)}=4$, and
$\wh F(k)\ne0$ for all $k\in\Z\smallsetminus\{0\}$.
\smallskip
\item $\int\!F\,\mathrm d\nu=0$ for every $\nu\in\cM(\T,T)$; in particular, $Q(T,F)=0$,
every invariant measure is $F$-maximizing, and $F$ is a weak coboundary.
\smallskip
\item There is no $u\in L^2(\T)$ such that $F=u-u\circ T$; in particular, $F$ is not a coboundary.
\smallskip
\item $F$ is not a Walters function.
\end{enumerate}
\end{prop}
\begin{proof}
(i) Each $k\in\N$ can be written uniquely in the form $k=(2n+1)2^m$ with $n$, $m\in\N_0$, so
\eqref{eq:wiener example} gives
\begin{equation}
\label{eq:wiener example coefficients}
\wh F(k)=\wh F(-k)=(1/2)\cdot 2^{-n}\lambda_m
\quad\text{ for }k=(2n+1)2^m,
\quad \wh F(0)=0 .
\end{equation}
Note that
\begin{equation}
\label{eq:lambda partial sum}
\sum_{m=1}^{M}\lambda_m=1-\frac1{\sqrt{M+1}}
\quad\text{ for all }M\in\N,
\quad\text{ so }
\sum_{m=1}^{+\infty}\lambda_m=1 .
\end{equation}

Since $\lambda_0=-1$ and $\lambda_m>0$ for all $m\in \N$, no $\wh F(k)$ with $k\ne0$
vanishes, and $\sum_{m=0}^{+\infty}\abs{\lambda_m}=\abs{\lambda_0}+\sum_{m=1}^{+\infty}\lambda_m=2$ by \eqref{eq:lambda partial sum}.
Thus $\sum_{k\in\Z}\Absbig{\wh F(k)}=\sum_{n=0}^{+\infty}2^{-n}\sum_{m=0}^{+\infty}\abs{\lambda_m}=4$. The series in \eqref{eq:wiener example}
therefore converges absolutely and uniformly, so $F$ is continuous and real-valued. This proves (i).

\smallskip
(ii) The orbits under iterations of $A$ on $\Z\smallsetminus\{0\}$ are the sets $\{2^mq:m\in\N_0\}$ with $q$
odd, so by \eqref{eq:wiener example coefficients} the orbit sum through
$q=\pm(2n+1)$ with $n\in\N_0$ is
\begin{equation*}
b_{\{2^mq:m\in\N_0\}}(F)=\frac12\cdot2^{-n}\sum_{m=0}^{+\infty}\lambda_m
=\frac12\cdot2^{-n}(\lambda_0+1)=0
\end{equation*}
by \eqref{eq:lambda partial sum}. Thus $b_\cO(F)=0$ for every nontrivial orbit $\cO$,
and \eqref{eq:orbit decomposition} gives $\int\!F\,\mathrm d\nu=\wh F(0)=0$ for all
$\nu\in\cM(\T,T)$. This proves (ii),
where the final part of (ii) follows from
the
characterisation of weak coboundaries recalled in Subsection~\ref{ss:intro main}.

\smallskip
(iii) Suppose that $u\in L^2(\T)$ satisfies $F=u-u\circ T$. Comparing Fourier
coefficients, and using $\wh{u\circ T}(k)=\wh u(k/2)$ for $k$ even and
$\wh{u\circ T}(k)=0$ for $k$ odd, we obtain from \eqref{eq:wiener example coefficients} that $\wh u(1)=\wh F(1)=-\tfrac12$ and
$\wh u\bigl(2^M\bigr)-\wh u\bigl(2^{M-1}\bigr)=\wh F\bigl(2^M\bigr)=\frac12\lambda_M$
for $M\in\N$. Summing and applying \eqref{eq:lambda partial sum},
\begin{equation*}
\wh u\bigl(2^M\bigr)=-\frac12+\frac12\sum_{m=1}^{M}\lambda_m
=-\frac1{2\sqrt{M+1}}
\quad\text{ for all }M\in\N_0 ,
\end{equation*}
so $\sum_{M=0}^{+\infty}\Absbig{\wh u\bigl(2^M\bigr)}^2
=\frac14\sum_{M=0}^{+\infty}(M+1)^{-1}=+\infty$. On the other hand, Parseval's identity gives $\sum_{k\in\Z}\Absbig{\wh u(k)}^2=\norm{u}_{L^2}^2<+\infty$, a contradiction. This
proves (iii).

\smallskip
(iv) Suppose that $F$ were Walters. Since $T$ is expanding, the Ma\~n\'e lemma 
(see \cite[Theorem~1]{Bou01}) gives
$\varphi\in C(\T,\R)$ with $F-Q(T,F)\le\varphi\circ T-\varphi$ on
$\T$, with equality on the support of every $F$-maximizing measure.
By (ii), $m_\T$ is $F$-maximizing with $Q(T,F)=0$. Together with $\supp m_\T=\T$, the equality holds
everywhere, i.e.,~$F=\varphi\circ T-\varphi$, making $F$ a coboundary, contrary to (iii).
This proves (iv).
\end{proof}

\begin{rem}
In the notation of Proposition~\ref{p:wiener not walters}, define, for $x\in\T$,
\begin{equation*}
f_1(x)\=\sum_{m=0}^{+\infty}\lambda_m\,e^{2\pi \I\cdot2^mx} .
\end{equation*}
Since $\lambda_m=r_m-r_{m+1}$, where $r_m\=m^{-1/2}$ for each $m\in \N$ with the convention $r_0\=0$, $f_1$ is exactly the
function used in \cite[Section~3]{BJ02} to exhibit a weak coboundary that is not a coboundary.
Since the orbits of
$A$ on $\Z\smallsetminus\{0\}$ are the sets $\{2^mq:m\in\N_0\}$ with $q$ odd, and
$f_1(qx)=\sum_{m=0}^{+\infty}\lambda_me^{2\pi \I q2^mx}$ relocates the frequency support of $f_1$ to the orbit through $q$,
\eqref{eq:wiener example} can be written as
\begin{equation*}
F(x)=\sum_{n=0}^{+\infty}\frac1{2^n}\,\Re(f_1 ((2n+1)x)). 
\end{equation*}
That is, the function $F$ is obtained by relocating the single-orbit construction of
\cite[Section~3]{BJ02} to every
nontrivial orbit of $A$ and summing.
The vanishing $\sum_{m=0}^{+\infty}\lambda_m=0$ correspondingly
becomes the vanishing of every orbit sum $b_\cO(F)$ rather than a single one, and the
non-coboundary argument in the proof of Proposition~\ref{p:wiener not walters}~(iii) is
the same Parseval obstruction as in \cite[Section~3]{BJ02}, applied on every orbit.
\end{rem}

\section{No closed form for the Fourier coefficients}
\label{s:P-recursive Fourier coefficients}

There is no canonical notion of closed form (see \cite{BC13} for a discussion of this issue).
For sequences, we work with the class of P-recursive sequences, defined below (see Definition~\ref{def:P-recursive}).
A P-recursive
sequence is determined by finitely many initial terms together with a single linear
recurrence with polynomial coefficients;
this class was studied systematically by
Stanley~\cite{St80}.
It is closed under termwise addition and multiplication and under
convolution, and it is the setting for the algorithmic proof of identities between
combinatorial sums and special functions; see \cite{Ka23, PWZ96}.
The class contains those
given by several more restrictive notions of closed form: the finite linear combinations of
hypergeometric terms, that is, of sequences whose ratio of successive terms is a rational
function of the index, which are the closed forms in the sense of \cite[Chapter~8]{PWZ96}; the
sequences satisfying a linear recurrence with constant coefficients; and the sequences of
Taylor coefficients of algebraic functions. A proof that a sequence is not P-recursive
therefore excludes a closed form in each of these senses, and it is in this sense that the
title of this section should be understood.

\begin{definition}\label{def:P-recursive}
A sequence
$\{c_n\}_{n\in\N_0}$ in $\C$ is called \emph{P-recursive},
or \emph{holonomic} (cf.~\cite{Ze90}),
if there exist
$r\in\N$, polynomials $p_0$, $p_1$, $\ldots$, $p_r\in\C[X]$ with
$p_r\not\equiv0$, and $n_0\in\N_0$ such that for all integers\footnote{The apparently stronger requirement that \eqref{eq:P-recursive recurrence} hold
for all $n\in\N_0$, which is the form in which the definition is usually
stated, is in fact equivalent. Indeed, if \eqref{eq:P-recursive recurrence}
holds for all $n\ge n_0$, then
replacing each $p_j$ by
$\widetilde p_j(X)=p_j(X)\prod_{i=0}^{n_0-1}(X-i)$ yields polynomials, still
with $\widetilde p_r\not\equiv0$ since $\C[X]$ is an integral domain, for which
the recurrence holds for all
$n\in\N_0$.} $n\ge n_0$,
\begin{equation}
\label{eq:P-recursive recurrence}
\sum_{j=0}^{r}p_j(n)c_{n+j}=0.
\end{equation}
The ordinary generating series
$C(z)=\sum_{n=0}^{+\infty}c_nz^n$,
in the ring of formal power series over $\C$,
is called \emph{D-finite} if there exist $s\in\N$ and polynomials $q_0$, $q_1$, $\ldots$, $q_s\in\C[X]$ with
$q_s\not\equiv0$, such that in the sense of formal power series,
\begin{equation}
\label{eq:D-finite}
\sum_{j=0}^{s}q_j(z)C^{(j)}(z)=0.
\end{equation}
\end{definition}

By \cite[Theorem~1.5]{St80},
the sequence $\{c_n\}_{n\in \N_0}$ is P-recursive if and only if its ordinary generating series
$C(z)$ is D-finite.

We write $\cD(\T)\=C^\infty(\T,\C)$ and denote its topological dual by
$\cD'(\T)$. We identify $F\in C(\T,\C)$ with the distribution defined by
$\langle F,\phi\rangle\=\int_{\T}\!F(x)\phi(x)\,\mathrm dm_\T(x)$ for $\phi\in C^\infty(\T,\C)$, and all differential equations below are understood in $\cD'(\T)$ (in particular higher derivatives are defined weakly via integration by parts against test functions). Recall that for $F\in C(\T,\C)$ and $n\in\Z$,
$\wh F(n)=\int_{\T}\!F(x)e^{-2\pi\I nx}\,\mathrm dm_\T(x)$.

For Fourier coefficients, P-recursiveness has an analytic formulation analogous to its
equivalence with D-finiteness: for $F\in C(\T,\R)$, the sequence
$\bigl\{\wh F(n)\bigr\}_{n\in\N_0}$ is P-recursive if and only if $F$ satisfies, in
$\cD'(\T)$, a nontrivial linear differential equation with trigonometric-polynomial
coefficients (Lemma~\ref{lem:P-recursive Fourier ODE} and Remark~\ref{rem:ODE converse}).
For example, if $\alpha>0$ then $u(x)\=\abs{\sin\pi x}^\alpha$ satisfies
$(1-\cos(2\pi x))\,u'=\alpha\pi\sin(2\pi x)\,u$ in $\cD'(\T)$, since $u$ is
absolutely continuous and the equation holds on $\T\smallsetminus\{0\}$; so the Fourier
coefficients of $u$ are P-recursive.
In these terms, Theorems~\ref{thm:circle} and~\ref{thm:P-recursive coboundary}
show that the function $\widetilde F$ of Theorem~\ref{thm:circle}, each of whose Fourier
coefficients is rational and is determined exactly by a finite procedure, satisfies no
nontrivial linear differential equation with trigonometric-polynomial coefficients.

We wish to prove Proposition~\ref{prop:P-recursive Fourier coefficients},
that a real-valued continuous function with P-recursive Fourier coefficients is H\"older continuous.
For this we use the fact that P-recursiveness of $\bigl\{\wh F(n)\bigr\}_{n\in\N_0}$ is equivalent to
D-finiteness of the associated generating series, and the substitution
$z=e^{2\pi\I x}$ converts the differential equation in $z$ into a
differential equation for $F$ itself, with trigonometric-polynomial
coefficients (Lemma~\ref{lem:P-recursive Fourier ODE}). Such an equation is
nonsingular away from the finitely many zeros of its leading coefficient, and
there $F$ is real analytic. At a zero the equation is singular, and the local
behaviour of its solutions is described by classical asymptotic theory: the
formal solutions were determined by Fabry~\cite{Fa85}, in the generality we
need by Cope~\cite{Co34}, and they are realised by actual solutions through
the asymptotic existence theorem (due in its original form to
Birkhoff~\cite{Bi09}, see
Wasow~\cite{Was65} for  the generality needed here). On each side of the
singular point $F$ is therefore a finite combination of solutions admitting
asymptotic expansions built from terms $e^{Q(r)}r^{\nu}(\log r)^{k}$, where $Q$
is a polynomial in a fractional power of $1/r$.
Continuity of $F$ then forces each $\Re Q$ to be bounded above; since $\Re Q$
is a real polynomial in a fractional power of $1/r$ without constant term, only
two possibilities remain, that $\Re Q\to-\infty$, in which case the terms
concerned are $\cO\bigl(r^N\bigr)$ for all $N>0$, or that $\Re Q\equiv0$. Among the
terms of the second kind, continuity forces every surviving exponent $\nu$ to
have $\Re\nu>0$.
This
yields a bound $\abs{F(x_0\pm r)-F(x_0)}=\cO\bigl(r^\beta\bigr)$ for some $\beta>0$,
together with a polynomial bound on the derivative; interpolating between the
two gives H\"older continuity. 
In both of these steps, the passage from continuity of $F$ to the conclusion uses
a non-cancellation estimate for finite
sums
$\sum_{j=1}^J c_je^{\I H_j(s)}s^{\I\lambda_j}$
(Lemma~\ref{lem:mean square}), proved using a mean-square argument.

\begin{lemma}[A differential equation for functions with P-recursive Fourier coefficients]
\label{lem:P-recursive Fourier ODE}
Let $F\in C(\T,\R)$.  If
$\bigl\{\wh F(n)\bigr\}_{n\in\N_0}$ is P-recursive, then there exist
$\ell\in\N$ and real-valued trigonometric polynomials
$a_0$, $a_1$, $\ldots$, $a_\ell$, with $a_\ell\not\equiv0$, such that
\begin{equation}
\label{eq:Fourier differential equation}
\sum_{j=0}^{\ell}a_j(x)\frac{\mathrm d^jF}{\mathrm dx^j}=0
\quad\text{ in }\cD'(\T).
\end{equation}
\end{lemma}

\begin{proof}
Since $\Absbig{\wh F(n)}\le\norm{F}_{L^1}$ for each $n\in\N_0$, the sequence $\bigl\{\wh F(n)\bigr\}_{n\in \N_0}$ is bounded.
Hence, the series $F_+\=\sum_{n=0}^{+\infty}\wh F(n)\exp(2\pi\I nx)$ converges in $\cD'(\T)$ and defines a distribution, by the rapid decay of the Fourier coefficients of smooth test functions (cf.~\cite[Theorem~3.3.9]{Gr14}).

By P-recursiveness of $\bigl\{\wh F(n)\bigr\}_{n\in\N_0}$, its generating series
$C(z)\=\sum_{n=0}^{+\infty}\wh F(n)z^n$ is D-finite, so that \eqref{eq:D-finite} holds for some $q_0$, $q_1$, $\ldots$, $q_s$ with $q_s\not\equiv0$.
Write $\theta\=z\,\frac{\mathrm d}{\mathrm dz}$. For all $j$, $n\in \N_0$,
\begin{equation*}
\theta(\theta-1)\cdots(\theta-j+1)z^n
=n(n-1)\cdots (n-j+1)z^n
=z^j\frac{\mathrm d^j}{\mathrm dz^j}z^n.
\end{equation*}
Hence, by linearity, $z^j C^{(j)}(z)=\theta(\theta-1)\cdots(\theta-j+1)C(z)$. Multiplying \eqref{eq:D-finite} by $z^{s}$ therefore gives
\begin{equation}
\label{eq:Euler form}
\sum_{j=0}^{s}q_j(z)\,z^{\,s-j}\,\theta(\theta-1)\cdots(\theta-j+1)C(z)=0 ,
\end{equation}
an identity of formal power series, in which every operator coefficient is a
polynomial in $z$ because $j\le s$.

We next transfer the formal power series identity to an identity of distributions.
Consider the coefficientwise Fourier correspondence $z^n\=e^{2\pi\I nx}$ for each $n\in \N_0$.
For each integer $j$ with $0\le j\le s$,
$q_j(z)z^{s-j}$ becomes a trigonometric polynomial and $\theta$ becomes
$(2\pi\I)^{-1}\frac{\mathrm d}{\mathrm dx}$. Hence, the left-hand side of
\eqref{eq:Euler form} becomes
$\sum_{j=0}^{s}b_j(x)\frac{\mathrm d^jF_+}{\mathrm dx^j}$ for some trigonometric
polynomials $b_0$, $b_1$, $\ldots$, $b_s$. Only the term $j=s$ of \eqref{eq:Euler form}
contributes to order $s$, the others having order at most $s-1$, so
$b_s(x)=(2\pi\I)^{-s}q_s\bigl(e^{2\pi\I x}\bigr)$; and $b_s\not\equiv0$, since a nonzero
polynomial has only finitely many roots on the unit circle. 

By \eqref{eq:Euler form} and the definition of $F_+$, both $F_+$ and the coefficients $b_j$ have only nonnegative Fourier frequencies, so for each $n\in \N_0$, the coefficient of
$e^{2\pi\I nx}$ in $\sum_{j=0}^{s}b_j\frac{\mathrm d^jF_+}{\mathrm dx^j}$ is a finite sum,
namely the coefficient of $z^n$ in the left-hand side of
\eqref{eq:Euler form}, which vanishes; and
for integer $n<0$ the coefficient vanishes for the same reason of frequency support.
Since a distribution on $\T$ is determined by its Fourier coefficients
(cf.~\cite[Chapter~VII, Section~1, Theorem~I]{Sc66}), a distribution whose Fourier
coefficients all vanish is identically zero. We therefore conclude that
\begin{equation}\label{eq:bF+}
\sum_{j=0}^{s}b_j(x)\frac{\mathrm d^jF_+}{\mathrm dx^j}=0
\quad\text{ in }\cD'(\T).
\end{equation}

A distributional differential equation satisfied by $F_+$ need not be satisfied by $F$.
To obtain \eqref{eq:Fourier differential equation}, we use the following fact.

\smallskip
\emph{Claim.}
If $u_1$, $u_2\in\cD'(\T)$ each satisfy a nontrivial linear
differential equation with trigonometric-polynomial coefficients, then so does $u_1+u_2$.

\smallskip
To prove the claim,  we recast it in the language of differential operators.
Writing $z\=e^{2\pi\I x}$, the trigonometric polynomials form the integral
domain $\C\bigl[z,z^{-1}\bigr]$, and $\frac{\mathrm d}{\mathrm dx}$ maps its field of fractions $K\=\C(z)$ into
itself. 
Since $\frac{\mathrm dz}{\mathrm dx}=2\pi\I z$, we have
$\frac{\mathrm d}{\mathrm dx}=2\pi\I z\frac{\mathrm d}{\mathrm dz}$ on $K$. 
Let $R$ denote the set of operators $\sum_{k=0}^{+\infty}f_k \frac{\mathrm d^k}{\mathrm dx^k}$ with
$f_k\in K$ and finitely many $f_k$ nonzero. Then $R$ is closed under composition: by the relation
$\frac{\mathrm d}{\mathrm dx}\circ g=g\frac{\mathrm d}{\mathrm dx}+g'$, valid for all
$g\in K$, any product of such operators may be rewritten with all
coefficients preceding all derivatives, and the coefficients so produced
again lie in $K$.
Moreover $R$ has no zero divisors: if $S$, $S'\in R$ are nonzero, of orders $a$
and $b$ and with leading coefficients $f_a$ and $g_b$, then the same relation
shows that $SS'$ has order $a+b$ and leading coefficient $f_ag_b\ne0$.

Suppose $u_1$, $u_2\in\cD'(\T)$ each satisfy a nontrivial linear differential equation with trigonometric-polynomial coefficients. Write these equations as
$P_1u_1=0$ and $P_2u_2=0$ in $\cD'(\T)$, where $P_1$, $P_2\in R\smallsetminus\{0\}$.

We first show that $P_1$, $P_2$ have a nonzero common left multiple in $R$.
For each $n\in \N_0$, let $R_n$ be the left $K$-vector subspace of $R$, which consists of those operators of order at most $n$. Hence, $R_n$ has dimension $n+1$, with basis $1$, $\frac{\mathrm d}{\mathrm dx}$, $\ldots$, $\frac{\mathrm d^n}{\mathrm dx^n}$.
Let $d_1$, $d_2$ denote the orders of $P_1$, $P_2$, respectively. Set $d\=\max\{d_1,\,d_2\}$ and take an integer $n\ge d$. 
For each $i\in\{1,\,2\}$, right multiplication by $P_i$ defines a left
$K$-linear map $p_i$. Since $R$ has no zero divisors, $p_i$ is injective; it therefore embeds $R_n$ into $R_{n+d}$,
giving two subspaces of dimension $n+1$ inside a space of dimension $n+d+1$.
As $2(n+1)>n+d+1$, the two subspaces $R_nP_1$ and $R_nP_2$ meet in some nonzero element, which may be written as $S_1P_1=S_2P_2\ne0$ with $S_1,S_2\in R_n$ nonzero
(cf.~\cite[Proposition~3.1]{Ze90}). Thus, $P_1$ and $P_2$ have a nonzero common left multiple in $R$.

Elements of $R$ have coefficients in the fraction field $K$, which may have poles on the unit
circle, so they need not act on $\cD'(\T)$; we therefore clear denominators.
Choose $g\in\C\bigl[z,z^{-1}\bigr]\smallsetminus\{0\}$ such that $gS_1$ and $gS_2$ both
have coefficients in $\C\bigl[z,z^{-1}\bigr]$ and
set $P\=(gS_1)P_1=(gS_2)P_2$. The
coefficients of $P$ are trigonometric polynomials, so $P$ does act on
$\cD'(\T)$, and $P\ne0$ since $g\neq 0$ and $S_1P_1\neq 0$. Since
$P_1u_1=P_2u_2=0$ in $\cD'(\T)$, we have
\begin{equation*}
P(u_1+u_2)=(gS_1)P_1u_1+(gS_2)P_2u_2=0
\quad\text{ in }\cD'(\T).
\end{equation*}
Since $P$ is nonzero and has trigonometric-polynomial coefficients, this is a
nontrivial linear differential equation of the required form, which proves the
claim.

\smallskip
Since $F$ is real-valued, $\wh F(-n)=\overline{\wh F(n)}$, and hence
$F=F_++\overline{F_+}-\wh F(0)$ in $\cD'(\T)$. Conjugating \eqref{eq:bF+} shows that $\overline{F_+}$ satisfies the analogous equation with
coefficients $\overline{b_j}$, which are trigonometric polynomials, while the
constant $-\wh F(0)$ is annihilated by $\frac{\mathrm d}{\mathrm dx}$. Thus, applying the claim twice to the above decomposition of $F$, we obtain a
nontrivial linear differential equation
$\sum_{j=0}^{L}g_jF^{(j)}=0$,
where $L\in\N_0$, $g_0$, $g_1$, $\ldots$, $g_L$ are trigonometric polynomials, and
$g_L\not\equiv0$.
Set $g_j=\alpha_j+\I\beta_j$ for each integer $0\le j\le L$, where $\alpha_j$, $\beta_j$ are real-valued
trigonometric polynomials. Since $F$ is real-valued, each $F^{(j)}$ is
a real-valued distribution. Separating real and imaginary parts gives
\begin{equation*}
\sum_{j=0}^{L}\alpha_jF^{(j)}=0
\quad\text{and}\quad
\sum_{j=0}^{L}\beta_jF^{(j)}=0 .
\end{equation*}
As $g_L\not\equiv0$, at least one of $\alpha_L$, $\beta_L$ is not identically
zero; interchanging the two equations if necessary, we may assume that
$\alpha_L\not\equiv0$.

If $L\ge1$, 
then the first equation, with
$\ell\=L$ and $a_j\=\alpha_j$ for each integer
$0\le j\le L$,
is of the required form. If
instead $L=0$, then $\alpha_0F=0$ in $\cD'(\T)$ with $\alpha_0\not\equiv0$; the product
$\alpha_0F$ is continuous, so it vanishes identically, and since a nonzero
trigonometric polynomial has only finitely many zeros, $F$ vanishes on a dense
subset of $\T$ and hence $F\equiv0$ by continuity. In that case $F$ satisfies
$F'=0$, which is of the required form with $\ell=1$. Either way the lemma
follows.
\end{proof}

\begin{rem}\label{rem:ODE converse}
The converse of Lemma~\ref{lem:P-recursive Fourier ODE} holds, and does not
require 
continuity or real-valuedness.
Suppose $u\in\cD'(\T)$ satisfies
$\sum_{j=0}^{\ell}a_ju^{(j)}=0$ in $\cD'(\T)$ for trigonometric polynomials $a_0$, $a_1$,
$\ldots$, $a_\ell$, not all identically zero, and write $\wh u(n)\=\langle u,e_{-n}\rangle$ for each $n\in \Z$,
where $e_{-n}(x)\=e^{-2\pi\I nx}$. Let
$k_-$ and $k_+$ be the smallest and largest
$k\in\Z$
such that $\wh{a_j}(k)\neq0$ for some integer $j$ with $0\le j\le \ell$, and set
$P_k(X)\=\sum_{j=0}^{\ell}\wh{a_j}(k)(2\pi\I(X-k))^j$ for $k\in\Z$. Comparing
Fourier coefficients gives
\begin{equation*}
\sum_{k=k_-}^{k_+}P_k(n)\,\wh u(n-k)=0
\quad\text{ for all }n\in\Z,
\end{equation*}
and $P_{k_-}\not\equiv0$, since $\wh{a_j}(k_-)\neq0$ for some $j$ and the polynomials
$(X-k_-)^j$, $0\le j\le\ell$, are linearly independent.
If $k_+>k_-$, the substitution
$n\=m+k_+$ turns this equation into a recurrence of the form \eqref{eq:P-recursive recurrence} for
$\{\wh u(m)\}_{m\in\N_0}$, of order $k_+-k_-$, with leading coefficient $P_{k_-}(X+k_+)\not\equiv0$, and valid for all $m\in\N_0$. If $k_+=k_-$, it gives $\wh u(n)=0$ for all
sufficiently large $n$. In either case $\{\wh u(n)\}_{n\in\N_0}$ is P-recursive.
\end{rem}

The equation obtained in Lemma~\ref{lem:P-recursive Fourier ODE} is singular
at the zeros of its leading coefficient, of which there are only finitely many,
since a nonzero trigonometric polynomial has only finitely many zeros in $\T$.
On every open interval containing no such zero, the following elementary lemma
applies, showing that any distributional solution there is a classical one,
that is, a function differentiable as many times as the order of the equation
and satisfying it pointwise. On an interval, the classical solutions of an
equation of order $m\in\N$ with nowhere-vanishing leading coefficient form a vector
space of dimension $m$, and a basis of it is called a \emph{fundamental set}.

\begin{lemma}[Regularity at nonsingular points]
\label{lem:regularity}
Let $J\subseteq\R$ be an open interval, let $m\in \N$, and let $b_0$, $b_1$, $\ldots$, $b_m\in C^\infty(J,\C)$ with $b_m$ nowhere zero. If $v\in\cD'(J)$ satisfies
$\sum_{j=0}^{m}b_jv^{(j)}=0$ in $\cD'(J)$, then $v\in C^\infty(J,\C)$ and $v$ is
a classical solution (i.e.,~satisfies the equation pointwise); 
in particular $v$ lies in the span of any fundamental set on $J$.
\end{lemma}

\begin{proof}
Dividing by $b_m$ we may assume $b_m\equiv1$. 
Set $V\=\transpose{\bigl(v,v',\ldots,v^{(m-1)}\bigr)}$, a vector of distributions on
$J$, and write
\begin{equation*}
M\=\begin{pmatrix}
0&1&&\\
&\ddots&\ddots&\\
&&0&1\\
-b_0&-b_1&\cdots&-b_{m-1}
\end{pmatrix},
\end{equation*}
a matrix with smooth entries, and it is directly verified that $V'=MV$.
Let $\Phi\in C^\infty(J,\mathrm{GL}_m(\C))$ be a classical fundamental matrix
for the system, that is, a matrix-valued solution of $\Phi'=M\Phi$ whose
columns are linearly independent at every point, so that $\det\Phi$ is
nowhere zero; such a $\Phi$ exists on the
whole of $J$: solutions of a linear system with
continuous coefficients extend across the whole interval on which the system
is defined, and taking the solutions with initial values
$e_1$, $e_2$, $\ldots$, $e_m\in\C^m$ at a point of $J$ gives a $\Phi$ whose determinant is
nowhere zero, by Abel's formula (see e.g.~\cite{Te12}).
Since $\Phi$ and $\Phi^{-1}$ are smooth, the product
$\Phi^{-1}V$ is a well-defined vector of distributions, and by the product
rule together with $\bigl(\Phi^{-1}\bigr)'=-\Phi^{-1}\Phi'\Phi^{-1}$,
\begin{equation*}
\bigl(\Phi^{-1}V\bigr)'
=-\Phi^{-1}\Phi'\Phi^{-1}V+\Phi^{-1}V'
=-\Phi^{-1}M\Phi\,\Phi^{-1}V+\Phi^{-1}MV=0 .
\end{equation*}
A distribution whose derivative vanishes on an interval is constant there,
so
$\Phi^{-1}V=c$ for some $c\in\C^m$ and $V=\Phi c$. In particular
$v=(\Phi c)_1$ is smooth and is a classical solution.
\end{proof}

The next four lemmas are independent of differential equations. The first (Lemma~\ref{lem:comparison}) is
an elementary comparison of exponential factors, the second (Lemma~\ref{lem:mean square}) 
gives a non-cancellation estimate,
the third (Lemma~\ref{lem:leading term}) combines them into the form in which they are applied at a singular
point in the proof of Lemma~\ref{lem:singular ODE regularity},
and the fourth (Lemma~\ref{lem:interpolation}) is a calculus lemma used at the end of that proof.

We use the convention that the \emph{sign} of a real number $t$ is
$-1$, $0$, or $1$ according as $t<0$, $t=0$, or $t>0$.

\begin{lemma}[Comparison of exponential factors]\label{lem:comparison}
Let $q$, $J\in\N$ and let $Q_1$, $Q_2$, $\ldots$, $Q_J$ be polynomials in $r^{-1/q}$ without
constant term. Then there is $r_0\in(0,1)$ such that, for each pair of integers $i$, $i'\in [1,J]$, the sign of $\Re Q_i(r)-\Re Q_{i'}(r)$ is constant on $(0,r_0)$. Moreover, this sign is zero on $(0,r_0)$ if and only if $\Re Q_i\equiv\Re Q_{i'}$.
Consequently, writing
$\Psi(r)\=\max\{\Re Q_i(r):i\in \N,\,i\le J\}$, the set
\begin{equation}\label{eq:dominant index set}
I^\ast\=\{i\in \N:i\le J,\,\Re Q_i(r)=\Psi(r)\text{ for all sufficiently small }r>0\}
\end{equation}
is nonempty, and the indices split into two classes as follows. For $i\in I^\ast$ we have
$\Re Q_i=\Psi$ throughout $(0,r_0)$. For $i\notin I^\ast$ we have
$\Re Q_i(r)-\Psi(r)\to-\infty$ as $r\downarrow0$, and moreover
$e^{\Re Q_i(r)-\Psi(r)}=\cO(r^\tau)$ as $r\downarrow0$, for every $\tau>0$.
\end{lemma}
\begin{proof}
For each pair of integers $i$, $i'\in[1,J]$, $\Re(Q_i-Q_{i'})$ is a real polynomial in $s\=r^{-1/q}$ without constant term.
So as $r\downarrow0$, $\Re(Q_i-Q_{i'})$ tends to $+\infty$, tends to $-\infty$, or
vanishes identically, according to the sign of the leading coefficient or its
being the zero polynomial; in particular its sign is constant below some
threshold $r_{i,i'}>0$, and is zero throughout exactly when
$\Re Q_i\equiv\Re Q_{i'}$. 
Set $r_0\=\min(\{1/2\}\cup\{r_{i,i'}:i,\,i'\in \N,\,1\le i<i'\le J\})\in(0,1)$.
On $(0,r_0)$, the relative order of the functions $\Re Q_1$, $\Re Q_2$, $\ldots$, $\Re Q_J$ is fixed,
so the maximum $\Psi$ is attained throughout $(0,r_0)$ by one and
the same nonempty set of indices, which is therefore exactly the set $I^\ast$
of \eqref{eq:dominant index set}. Fixing $i^\ast\in I^\ast$, for
$i\notin I^\ast$ the difference $\Re Q_i-\Psi=\Re(Q_i-Q_{i^\ast})$ is a nonzero
real polynomial in $s$ without constant term which is negative on $(0,r_0)$,
and hence tends to $-\infty$ as $r\downarrow 0$. More precisely, the exponential of a polynomial in $r^{-1/q}$ with negative leading coefficient decays faster than every positive power of $r$ as $r\downarrow0$.
\end{proof}

Next we establish a
non-cancellation estimate, proved by a mean-square argument.

\begin{lemma}[Mean-square non-cancellation]\label{lem:mean square}
Let $J\in\N$, let $H_1$, $H_2$, $\ldots$, $H_J$ be real polynomials with zero constant term, and let $\lambda_1$, $\lambda_2$, $\ldots$, $\lambda_J\in\R$, and suppose that the pairs
$(H_j,\lambda_j)$ are pairwise distinct. Then, for all
$c_1$, $c_2$, $\ldots$, $c_J\in\C$, the function
$S(s)\=\sum_{j=1}^{J}c_je^{\I H_j(s)}s^{\I\lambda_j}$ satisfies
\begin{equation*}
\lim_{R\to+\infty}\frac{1}{\log R}\int_1^R\!\abs{S(s)}^2\,\frac{\mathrm ds}{s}
=\sum_{j=1}^{J}\abs{c_j}^2 .
\end{equation*}
In particular, if $S(s)\to0$ as $s\to+\infty$ then $c_1=c_2=\cdots=c_J=0$.
\end{lemma}

\begin{proof}
It suffices to show that for a real polynomial $H$ with zero constant term and
$\mu\in\R$ with $(H,\mu)\ne(0,0)$,
\begin{equation}\label{eq:off-diagonal estimate}
\lim_{R\to+\infty}\frac{1}{\log R}
\int_1^R\!e^{\I H(s)}s^{\I\mu}\,\frac{\mathrm ds}{s}=0.
\end{equation}
The lemma then follows by expanding $\abs{S}^2$.
The diagonal terms contribute
$\sum_{j=1}^{J}\abs{c_j}^2$, while the off-diagonal terms are handled
by \eqref{eq:off-diagonal estimate} with $H\=H_j-H_k$ and $\mu\=\lambda_j-\lambda_k$ for integers $1\le j$, $k\le J$ with $j\ne k$. For such $j$ and $k$, $H_j-H_k$ is a real polynomial with zero constant term, and $(H_j-H_k,\lambda_j-\lambda_k)\ne(0,0)$ because the pairs $(H_j,\lambda_j)$ are distinct.

If $H$ is constant, then $H\equiv0$ and $\mu\ne0$. Hence
$\int_1^R\!s^{\I\mu-1}\,\mathrm ds=\bigl(R^{\I\mu}-1\bigr)\big/(\I\mu)$ is bounded. Otherwise, denote
$d\=\deg H\ge1$. Since $H'$ is a nonzero polynomial, there exists $R_0>1$ such that $H'(s)\neq 0$ for all $s\ge R_0$. There is a constant $C>0$ such that, for all $s\ge R_0$,
\begin{equation}\label{eq:integration by parts estimate}
\Absbigg{\frac{s^{\I\mu-1}}{H'(s)}}\le Cs^{-d}
\quad\text{and}\quad
\Absbigg{\frac{\mathrm d}{\mathrm ds}\biggl(\frac{s^{\I\mu-1}}{H'(s)}\biggr)}
\le Cs^{-d-1}.
\end{equation}
Since $e^{\I H(s)}s^{\I\mu}\frac{\mathrm ds}{s}
=\frac{s^{\I\mu-1}}{\I H'(s)}\,\mathrm d\bigl(e^{\I H(s)}\bigr)$,
integration by parts gives
\begin{equation}\label{eq:parts}
\int_{R_0}^R\! e^{\I H(s)}s^{\I\mu}\frac{\mathrm ds}{s}
=\frac{e^{\I H(s)}s^{\I\mu-1}}{\I H'(s)}\bigg|_{R_0}^{R}
-\int_{R_0}^R\! e^{\I H(s)}
\frac{\mathrm d}{\mathrm ds}
\biggl(\frac{s^{\I\mu-1}}{\I H'(s)}\biggr)\,\mathrm ds .
\end{equation}
By \eqref{eq:integration by parts estimate}, the boundary term in \eqref{eq:parts} is bounded; also the remaining integrand is
$\cO\bigl(s^{-d-1}\bigr)$ with $d+1\ge2$, so the integral term is also bounded.
Hence the right-hand side of \eqref{eq:parts} is
$\cO(1)$. Adding the contribution from $[1,R_0]$, a constant independent of
$R$, and dividing by $\log R$, gives \eqref{eq:off-diagonal estimate}.
\end{proof}

The following consequence of Lemmas~\ref{lem:comparison} and~\ref{lem:mean square} 
isolates the dominant
term of a finite sum of exponential-power-logarithmic terms, of the kind
occurring in the asymptotic expansions below. The difficulty it addresses is
that several $Q_i$ may share the largest real part without being equal, since
$\Re Q_i\equiv\Re Q_{i'}$ forces only that $Q_i-Q_{i'}$ be purely imaginary.
After factoring out a dominant exponential factor as in Lemma~\ref{lem:comparison}, and then the dominant power and logarithmic factors, the surviving terms have constant nonzero
modulus and differ only in phase;  Lemma~\ref{lem:mean square} prevents these oscillatory terms from cancelling.

In the following, the imaginary part of a complex number $z$ is denoted by $\Im (z)$.

\begin{lemma}[The leading term of a finite exponential sum]\label{lem:leading term}
Let $q$, $J\in\N$, and let $(Q_1,\nu_1,k_1)$, $(Q_2,\nu_2,k_2)$, $\ldots$, $(Q_J,\nu_J,k_J)$ be pairwise distinct triples, with each $Q_i$ a polynomial in $r^{-1/q}$ without constant
term (possibly $Q_i\equiv0$), $\nu_i\in\C$, and $k_i\in\N_0$, and suppose
$c_1$, $c_2$, $\ldots$, $c_J\in\C\smallsetminus\{0\}$. Set
\begin{equation}\label{eq:finite exponential sum}
w(r)\=\sum_{i=1}^{J}c_i\,e^{Q_i(r)}\,r^{\nu_i}(\log r)^{k_i}
\quad\text{ for }0<r<1.
\end{equation}
Let $\Psi$ and $I^\ast$ be as in Lemma~\ref{lem:comparison}, formed from $Q_1$, $Q_2$, $\ldots$, $Q_J$, so that $I^\ast\ne\emptyset$. Define
\begin{equation}\label{eq:rho K star}
\rho^\ast\=\min\{\Re\nu_i:i\in I^\ast\}\in \R\quad\text{ and }\quad
K^\ast\=\max\{k_i:i\in I^\ast,\,\Re\nu_i=\rho^\ast\}\in \N_0.
\end{equation}
Then there exists a constant $c>0$ and a sequence $\{r_l\}_{l\in\N}$ of points in $(0,1)$ with $r_l\to0$ as $l\to +\infty$ such that
\begin{equation}\label{eq: w asymptotic}
\abs{w(r_l)}\ge c\,e^{\Psi(r_l)}\,r_l^{\rho^\ast}
\abs{\log r_l}^{K^\ast}\quad\text{ for all }l\in\N.
\end{equation}
\end{lemma}

\begin{proof}
Let $r_0$ be as in Lemma~\ref{lem:comparison}, fix $i^\ast\in I^\ast$, and set
$Q^\ast\=Q_{i^\ast}$, so that $\Re Q^\ast=\Psi$ on $(0,r_0)$, and
$\Re Q_i\equiv\Re Q^\ast$ holds precisely for $i\in I^\ast$.

Divide $w(r)$ by $e^{Q^\ast(r)}r^{\rho^\ast}\abs{\log r}^{K^\ast}$. A term in $w(r)$
with $\Re Q_i\not\equiv\Re Q^\ast$ acquires the factor $e^{(Q_i-Q^\ast)(r)}$,
of modulus $e^{\Re(Q_i-Q^\ast)(r)}$, which by Lemma~\ref{lem:comparison} tends
to $0$ faster than every power of $r$. Since the remaining factor
$r^{\nu_i-\rho^\ast}(\log r)^{k_i}\abs{\log r}^{-K^\ast}$ is bounded in modulus
by a power of $r^{-1}$, the whole term tends to $0$ as $r\downarrow 0$. By the same elementary comparison, a term in $w(r)$
with $\Re Q_i\equiv\Re Q^\ast$ but
$\Re\nu_i>\rho^\ast$, or with $\Re Q_i\equiv\Re Q^\ast$, $\Re\nu_i=\rho^\ast$ but $k_i<K^\ast$, acquires a factor tending to $0$ by a power of $r$ or of $\log r$, respectively. What
remains is
\begin{equation}\label{eq:leading sum}
\frac{w(r)}{e^{Q^\ast(r)}r^{\rho^\ast}\abs{\log r}^{K^\ast}}
=\sum_{i\in I_0} c_i'\,e^{\I \, \Pi_i(s)}\,s^{\I\mu_i}+o(1),
\end{equation}
where $s\=r^{-1/q}$ and $I_0\=\{i\in I^\ast:\Re\nu_i=\rho^\ast,\,k_i=K^\ast\}$; each $c_i'$ is
a nonzero constant multiple of $c_i$, namely $c_i'=(-1)^{K^\ast}c_i$;
$\mu_i\=-q\Im\nu_i$; and $\Pi_i\=\Im(Q_i-Q^\ast)$, a real polynomial in $s$
without constant term, the relevance of $\Re(Q_i-Q^\ast)\equiv0$ for $i\in I_0$
being that it makes $e^{(Q_i-Q^\ast)(r)}=e^{\I\Pi_i(s)}$. For all $i\in I_0$, the pairs $(\Pi_i,\mu_i)$ are pairwise distinct.
Indeed, suppose that $(\Pi_i,\mu _i)=(\Pi_{i'},\mu_{i'})$ for some distinct $i$, $i'\in I_0$. Then $\Pi_i=\Pi_{i'}$ implies $\Im Q_i=\Im Q_{i'}$, which together with $\Re Q_i=\Re Q_{i'}$ gives
$Q_i=Q_{i'}$. Equality of $\mu_i$ and $\mu_{i'}$ then gives $\Im\nu_i=\Im\nu_{i'}$,
which together with $\Re\nu_i=\Re\nu_{i'}=\rho^\ast$ gives $\nu_i=\nu_{i'}$.
Since $k_i=k_{i'}=K^\ast$, the triples $(Q_i,\nu_i,k_i)$ and $(Q_{i'},\nu_{i'},k_{i'})$ would coincide, a contradiction.

Hence, Lemma~\ref{lem:mean square} applies to
$S(s)\=\sum_{i\in I_0} c_i'e^{\I\Pi_i(s)}s^{\I\mu_i}$, the sum on the
right-hand side of \eqref{eq:leading sum}.
Writing $\Sigma\=\sum_{i\in I_0}\abs{c_i'}^2$, a positive number since $I_0$ is
nonempty by the definitions of $\rho^\ast$ and $K^\ast$ and $c_i'\ne0$ for each $i\in I_0$,
we obtain that
\begin{equation}\label{eq:Sigma limit}
\frac{1}{\log R}\int_1^R \!\abs{S(s)}^2\,\frac{\mathrm ds}s 
\longrightarrow \Sigma\quad\text{ as }R\to+\infty.
\end{equation}
Consequently, $\abs{S(s)}$ cannot remain smaller than $\frac12\sqrt\Sigma$ for all
sufficiently large $s$. Otherwise, there would exist $s_0>1$ such that $\abs{S(s)}\le \frac12\sqrt\Sigma$ for all
$s\ge s_0$. Then for $R>s_0$, the left-hand side of \eqref{eq:Sigma limit} would be at most
$\frac{1}{\log R}\int_1^{s_0}\!\abs S^2\frac{\mathrm d s}{s}+\frac14\Sigma$,
which tends to $\frac14\Sigma<\Sigma$ as $R\to+\infty$, contradicting \eqref{eq:Sigma limit}.

Hence, there exists a sequence $\{s_l\}_{l\in \N}$ in $[1,+\infty)$ such that $s_l\to+\infty$ as $l\to+\infty$ and $\abs{S(s_l)}\ge\frac12\sqrt\Sigma$ for all $l\in \N$.
Writing $r_l\=s_l^{-q}$, we have $r_l\to0$ as $l\to+\infty$.
Since the error term in \eqref{eq:leading sum} is $o(1)$ as $r\downarrow0$, its absolute value at $r=r_l$ is smaller than $\frac14\sqrt\Sigma$ once $l$ exceeds some $l_0\in\N$. Therefore,
\begin{equation*}
\frac{\abs{w(r_l)}}{e^{\Re Q^\ast(r_l)}r_l^{\rho^\ast}\abs{\log r_l}^{K^\ast}}
 \ge \frac14\sqrt\Sigma\quad\text{ for }l>l_0.
\end{equation*}
Discarding the terms $r_l$ with $l\le l_0$ and those with $r_l\ge r_0$, on which
$e^{\Re Q^\ast}$ and $e^{\Psi}$ might differ, and relabelling the remainder as
$r_1$, $r_2$, $\ldots$, we see that \eqref{eq: w asymptotic} holds with $c\=\tfrac14\sqrt\Sigma$.
\end{proof}

The following is an elementary interpolation, 
deducing H\"older continuity from a
decay estimate at a point and a growth estimate on the derivative.

\begin{lemma}[H\"older continuity by interpolation]\label{lem:interpolation}
Let $x_0\in\R$ and $\eta_0\in(0,1/2]$, and let
$g\:(x_0-\eta_0,x_0+\eta_0)\to\C$ be continuous, and continuously
differentiable on $(x_0-\eta_0,x_0)\cup(x_0,x_0+\eta_0)$. For
$\sigma\in\{-1,\,1\}$, define $g_\sigma\:[0,\eta_0)\to\C$ by
$g_\sigma(r)\=g(x_0+\sigma r)-g(x_0)$, so that $g_\sigma(0)=0$. Suppose there
are $C>0$, $\beta\in(0,1)$, and $M\in\N$ such that
\begin{equation}\label{eq:interp hyp}
\abs{g_\sigma(r)}\le Cr^\beta
\quad\text{ and }\quad
\abs{g_\sigma'(r)}\le Cr^{-M}
\end{equation}
for each sign $\sigma\in \{-1,\,1\}$ and each $r\in(0,\eta_0)$. Then $\alpha\=\beta/(\beta+M) \in (0,\beta)$
satisfies 
\begin{equation*}
\abs{g(x)-g(y)}\le 3C\abs{x-y}^{\alpha}
\quad\text{ for all }x,\,y\in(x_0-\eta_0,x_0+\eta_0).
\end{equation*}
\end{lemma}
\begin{proof}
Since $\beta+M>1$, we have $\alpha=\beta/(\beta+M)<\beta$ and $\alpha>0$. Any two points of $(x_0-\eta_0,x_0+\eta_0)$ are at distance at most
$2\eta_0\le1$; consequently $h^\beta\le h^\alpha$ for every distance $h$ arising below.

Fix $\sigma\in\{-1,\,1\}$, let $0\le s<t<\eta_0$, and set $h\=t-s$. If $s=0$, then
$\abs{g_\sigma(t)-g_\sigma(0)}=\abs{g_\sigma(t)}\le Ct^\beta\le Ch^\alpha$.
If $t\ge2s>0$, then $s\le h$ and $t\le2h$, so
$\abs{g_\sigma(t)-g_\sigma(s)}\le C\bigl(t^\beta+s^\beta\bigr)\le 3Ch^\beta\le 3Ch^\alpha$.

It remains to consider the case $0<t<2s$. Then
$\abs{g_\sigma(t)-g_\sigma(s)}\le\min\bigl\{3Cs^\beta,\,Chs^{-M}\bigr\}$, the first bound
because $t^\beta\le(2s)^\beta\le2s^\beta$, the second by integrating
\eqref{eq:interp hyp} over $[s,t]$. If $h\ge s^{\beta+M}$, then
$s^\beta\le h^{\beta/(\beta+M)}=h^\alpha$, and the first bound gives
$\abs{g_\sigma(t)-g_\sigma(s)}\le 3Ch^\alpha$. If $h<s^{\beta+M}$, then
$h^{1-\alpha}<s^{(\beta+M)(1-\alpha)}=s^M$, so $hs^{-M}<h^\alpha$, and the
second bound gives $\abs{g_\sigma(t)-g_\sigma(s)}\le Ch^\alpha$.

Since $g_\sigma(t)-g_\sigma(s)=g(x_0+\sigma t)-g(x_0+\sigma s)$, this covers
every pair of points lying strictly on one side of $x_0$. The case $s=0$ covers pairs
one of which is $x_0$. Finally, for all $r$, $t>0$ with $x_0+r$ and $x_0-t$ in
$(x_0-\eta_0,x_0+\eta_0)$, the first estimate in \eqref{eq:interp hyp} gives
\begin{equation*}
\abs{g(x_0+r)-g(x_0-t)}\le C\bigl(r^\beta+t^\beta\bigr)\le 2C(r+t)^\alpha
=2C\abs{(x_0+r)-(x_0-t)}^{\alpha}. \qedhere
\end{equation*}
\end{proof}

The following lemma uses the classical local
theory of linear differential equations at a singular point (specifically, Fabry's theorem on
formal solutions, and the asymptotic existence theorem realising them as
genuine solutions). 

\begin{lemma}[Local structure of solutions at a singular point]
\label{lem:local expansion}
Let $I\subseteq\R$ be an open interval, $x_0\in I$, $m\in\N$, and
$a_0$, $a_1$, $\ldots$, $a_m$ be real analytic functions on $I$ with $a_m$ vanishing at
$x_0$ and nowhere else in $I$. Put
$L\=a_m(x)\frac{\mathrm d^m}{\mathrm dx^m}+\cdots+a_1(x)\frac{\mathrm d}{\mathrm dx}+a_0(x)$, and fix $\sigma\in\{-1,\,1\}$.
Then there are $q\in\N$, $\eta>0$, formal expressions
$\widetilde z_1$, $\widetilde z_2$, $\ldots$, $\widetilde z_m$, and classical solutions
$z_1$, $z_2$, $\ldots$, $z_m$ of $Lz=0$ on $I_\sigma^\eta\=\{x\in \R:0<\sigma(x-x_0)<\eta\}$, a
subinterval of $I$, with the following properties, in which
$\zeta_j(r)\=z_j(x_0+\sigma r)$.
\begin{enumerate}[label=\textup{(\roman*)}]
\smallskip
\item Each $\widetilde z_j$ has the form
\begin{equation}
\label{eq:formal solution}
\widetilde z_j\= e^{Q_j(r)}r^{\kappa_j}
\sum_{k=0}^{N_j}\sum_{i=0}^{+\infty} b^{(j)}_{k,i}\, r^{i/q}(\log r)^k ,
\end{equation}
with $N_j\in\N_0$, $\kappa_j\in\C$,  $b^{(j)}_{k,i}\in\C$, and $Q_j$ a polynomial in $r^{-1/q}$
without constant term.
\smallskip
\item The $\widetilde z_1$, $\widetilde z_2$, $\ldots$, $\widetilde z_m$ are linearly independent over
$\C$; and expressions built on distinct exponential factors are independent, in
the sense that if $Q$ runs over finitely many distinct polynomials in
$r^{-1/q}$ without constant term, and for each of them $S_Q$ is a formal sum of
terms $r^{\nu}(\log r)^{k}$ with $\nu\in\C$ and $k\in\N_0$, then
$\sum_{Q}e^{Q(r)}S_Q(r)=0$ only if every coefficient of every $S_Q$ vanishes.
In particular, every $\C$-linear combination of
$\widetilde z_1$, $\widetilde z_2$, $\ldots$, $\widetilde z_m$ and the constant function $1$ has, for
each triple $(Q,\nu,k)$, a uniquely determined coefficient of
$e^{Q(r)}r^\nu(\log r)^k$, obtained by summing all contributions corresponding
to that same triple. Terms carrying different exponential factors are never combined.
\smallskip
\item $z_1$, $z_2$, $\ldots$, $z_m$ form a fundamental set for $L$ on $I_\sigma^\eta$.
\smallskip
\item Writing $\widetilde z_{j,[n]}$ for \eqref{eq:formal solution} with the
inner sum restricted to $0\le i\le n$, for each $\tau>0$ and each
$\ell\in\N_0$, there is $n_0\in \N$ such that, for all integers $n\ge n_0$ and $1\le j\le m$,
\begin{equation}
\label{eq:remainder}
e^{-Q_j(r)}\,\partial_r^{\,\ell}
\bigl(\zeta_j(r)-\widetilde z_{j,[n]}(r)\bigr)=\cO(r^\tau)
\quad \text{ as }r\downarrow0.
\end{equation}
\end{enumerate}
\end{lemma}

\begin{proof}
It is convenient to move the singular point to the origin. The substitution
$x=x_0+\sigma r$ transforms $L$ into
\begin{equation*}
L_\sigma\=\sum_{k=0}^{m}\sigma^{k}a_k(x_0+\sigma r)
\frac{\mathrm d^k}{\mathrm dr^k},
\end{equation*}
whose coefficients are real analytic in $r$ near $0$, and whose leading
coefficient $\sigma^ma_m(x_0+\sigma r)$ vanishes at $r=0$ to the same finite
order as $a_m$ does at $x_0$, and nowhere else nearby; and $z$ solves $Lz=0$ on
$I_\sigma^\eta$ if and only if $r\mapsto z(x_0+\sigma r)$ solves
$L_\sigma\zeta=0$ on $(0,\eta)$. Being real analytic, the coefficients of
$L_\sigma$ extend holomorphically to a complex neighbourhood of $r=0$, so that
$L_\sigma$ is a linear differential operator with an isolated singular point at
the origin, and the classical local theory at such a point applies to it.

Parts (i) and (ii) are Fabry's theorem on formal solutions, in the form given
by Cope~\cite[Theorem~I]{Co34}: at such a point there are as many linearly
independent formal solutions as the order of the equation, each of the form
\eqref{eq:formal solution}. The notion of linear independence is
\cite[Definition~6]{Co34}, that no relation
$\sum_{j=1}^{m}\lambda_j\widetilde z_j=0$ holds with constants $\lambda_j$ not all
zero; equivalently, the formal Wronskian does not vanish
\cite[Theorem~IV]{Co34}. Cope's theorem is proved at $x=\infty$, and the
reduction of a finite singular point to that case, together with the resulting
form \eqref{eq:formal solution}, is recorded in \cite[p.~411]{Co34} and
\cite[p.~132, (1.1)]{Sv81}; the uniqueness assertion in (ii) is
\cite[Definition~1]{Co34}, \cite[pp.~132--133]{Sv81}.

Parts (iii) and (iv) are the main asymptotic existence theorem, applied on a
sector of sufficiently small opening about the positive real $r$-axis. The
method originates with Birkhoff~\cite{Bi09}, who assumed the characteristic
roots distinct, in which case neither fractional powers nor logarithms occur;
in the generality needed here, where the formal normal form is that of Hukuhara
and Turrittin, we follow Wasow~\cite[Theorem~19.1]{Was65}, and
\cite[pp.~140--141]{Sv81} for the statement in the one-sided real form used
here. The estimate \eqref{eq:remainder} for $\ell\ge1$ requires no separate
argument: the solutions are produced within a class defined by requiring the
asymptotic relation for every derivative \cite[p.~135, (2.1)--(2.2)]{Sv81}. Only finitely many $\zeta_j$ occur, so a
single $q$ may be fixed to serve all of them.\footnote{Wasow's theorem is
stated at infinity and for first-order systems: the substitution $\xi=1/r$ and
passage to the companion system of $L_\sigma$ bring our setting into
that form, the companion matrix being meromorphic at $r=0$ because $L_\sigma$
has real-analytic coefficients and a leading coefficient vanishing to finite
order. Its conclusion is a fundamental matrix $\wh Y(\xi)\,\xi^{G}e^{\Lambda(\xi)}$
in which $\Lambda$ is diagonal with entries polynomial in a fractional power of
$\xi$, $G$ is a constant matrix, and $\wh Y$ admits an asymptotic expansion in
fractional powers of $\xi^{-1}$. Its columns are the solutions of the companion
system, whose first entries are solutions of $L_\sigma\zeta=0$, and each of
these has the form \eqref{eq:formal solution}: the powers of $\log r$ are those
produced by $\xi^{G}=e^{G\log\xi}$ when $G$ is not diagonalisable.
In this complex-variable setting, the validity of
\eqref{eq:remainder} for $\ell\ge1$ amounts to the term-by-term
differentiability of an asymptotic expansion on a sector,
cf.~\cite[Theorem~8.8]{Was65}, which gives the same conclusion independently of \cite{Sv81}.}
\end{proof}

The main step in the proof of Proposition~\ref{prop:P-recursive Fourier coefficients} is
Lemma~\ref{lem:singular ODE regularity} below. For the distributional differential equation in Lemma~\ref{lem:singular ODE regularity},
away from the zeros of the leading coefficient, standard regularity gives analyticity, so the main issue is the behaviour near a singular point.
Lemma~\ref{lem:local expansion} provides the required local asymptotic
expansions. After truncation to arbitrary order, these are finite sums of terms
$e^{Q(r)}r^\nu(\log r)^k$,
where $Q$ is a polynomial in $r^{-1/q}$ without constant term, $\nu\in\C$, and $k\in\N_0$, together with an error of order
$\cO\bigl(e^{\Re Q(r)}r^\tau\bigr)$ for any prescribed $\tau>0$.

We then separate the exponential factors according as $\Re Q(r)\to+\infty$, $\Re Q(r)\to-\infty$, or $\Re Q\equiv0$ as
$r\downarrow0$.
Lemma~\ref{lem:leading term} ensures that the dominant terms in such a finite sum cannot cancel completely near $0$.
Thus, the continuity of the solution $u$ forces
terms of the first type, whose absolute values blow up near $0$, to vanish.
Terms of the second type decay faster than any power of $r$ and are therefore negligible for our purposes.
For the remaining terms of the third type,
Lemma~\ref{lem:leading term} and continuity imply that every surviving exponent $\nu$ in $e^{Q(r)}r^\nu(\log r)^k$ satisfies $\Re\nu>0$. This gives a positive power bound for $u(x)-u(x_0)$, while differentiating the asymptotic expansion gives at most polynomial growth for its derivative. Lemma~\ref{lem:interpolation} then yields the required H\"older continuity.

\begin{lemma}[Continuous solutions of singular analytic ODEs]\label{lem:singular ODE regularity}
Let $U\subseteq\R$ be an open interval, let $m\in \N$, and let $a_0$, $a_1$, $\ldots$, $a_m$ be real-analytic functions on $U$ with $a_m\not\equiv0$. Define
\begin{equation*}
L\=a_m(x)\frac{\mathrm d^m}{\mathrm dx^m}+\cdots+a_1(x)\frac{\mathrm d}{\mathrm dx}+a_0(x).
\end{equation*}
Set $Z(a_m)\=\{x\in U\: a_m(x)=0\}$.
If $u\in C(U,\C)$ and $Lu=0$ in $\cD'(U)$, then $u$ is locally H\"older continuous. Moreover, $u$ is real analytic on $U\smallsetminus Z(a_m)$.
\end{lemma}

\begin{proof}
Since $a_m$ is analytic and not identically zero, its zeros are isolated and
have finite order. On every component of $U\smallsetminus Z(a_m)$,
Lemma~\ref{lem:regularity} shows that $u$ is a $C^\infty$ classical solution;
analyticity of the coefficients then implies that $u$ is real analytic there,
since the companion system extends holomorphically to a complex neighbourhood
of each such point, where e.g.~\cite[Theorem~4.1]{Te12} gives an analytic solution
that agrees with $u$ on the real axis by uniqueness.

Fix $x_0\in Z(a_m)$. After shrinking the interval if necessary, assume that $x_0$ is the only zero of
$a_m$ in $I\=(x_0-\varepsilon,x_0+\varepsilon)$.
Since the same argument applies on both sides of $x_0$, it suffices to work on one side of $x_0$. Fix $\sigma\in\{-1,\,1\}$ and write $I_\sigma\=\{x\in I:\sigma(x-x_0)>0\}$.

Apply Lemma~\ref{lem:local expansion} on $I$ with this fixed $\sigma$, and let
$q\in\N$, $\eta>0$, the formal expressions
$\widetilde z_1$, $\widetilde z_2$, $\ldots$, $\widetilde z_m$, and the classical solutions
$z_1$, $z_2$, $\ldots$, $z_m$ be as provided there; properties (i)--(iv) of that lemma will be used below.
As in Lemma~\ref{lem:local expansion}, we write $\zeta_j(r)\=z_j(x_0+\sigma r)$, and
$Q_j$, $\kappa_j$, $N_j$, $b^{(j)}_{k,i}$ for the objects attached to
$\widetilde z_j$.
Set $v(r)\=u(x_0+\sigma r)-u(x_0)$ for $0<r<\eta$, so that $v(r)\to0$ as
$r\downarrow0$ by continuity. By Lemma~\ref{lem:regularity} applied on
$I_\sigma$, the function $u$ is a classical solution of $Lz=0$ on $I_\sigma$, so by
property~(iii) of Lemma~\ref{lem:local expansion},
\begin{equation}
\label{eq:v expansion}
v(r)=\sum_{j=1}^{m}c_j\,\zeta_j(r)-u(x_0)
\end{equation}
for suitable constants $c_1$, $c_2$, $\ldots$, $c_m\in\C$.

For any prescribed constant $\tau>0$, truncating \eqref{eq:formal solution} at $i=n$ for some sufficiently large $n\in \N$ expresses
$\zeta_j$, for each integer $1\le j\le m$, as a finite $\C$-linear combination of terms
$e^{Q_j(r)}r^{\nu}(\log r)^{k}$, with $\nu\=\kappa_j+(i/q)$ for one of the
finitely many retained pairs $(k,i)$, together with an error which by
\eqref{eq:remainder} is $\cO\bigl(e^{\Re Q_j(r)}r^{\tau}\bigr)$. Note the factor $e^{\Re Q_j}$ in this error, which
arises because \eqref{eq:remainder} estimates
$e^{-Q_j}\bigl(\zeta_j(r)-\widetilde z_{j,[n]}(r)\bigr)$ rather than
$\zeta_j(r)-\widetilde z_{j,[n]}(r)$ itself: the error is comparable to the
terms retained, not small relative to them, so a lower bound on the truncation
does not by itself transfer to $v$. The exponents $\Re Q_j$ must therefore be brought
under control first; once the possibility $\Re Q_j(r)\to+\infty$ has been
excluded, the factor $e^{\Re Q_j}$ is bounded and disappears from the error term.

Before carrying out this analysis, we first fix some notation for the coefficients of the expansion of $v$, and record two facts that will be used throughout. By \eqref{eq:v expansion}, $v$ is
asymptotic to the formal expression
$\widetilde v\=\sum_{j=1}^{m}c_j\widetilde z_j-u(x_0)$, the constant $-u(x_0)$
being regarded, if nonzero, as the trivial term $-u(x_0)e^0r^0(\log r)^0$, corresponding to $(Q,\nu,k)=(0,0,0)$.
By property~(ii) of Lemma~\ref{lem:local expansion}, terms carrying different
exponential factors are never combined, whereas terms carrying the same one
must be.
Therefore, for $Q$ a polynomial in $r^{-1/q}$
without constant term, $\nu\in\C$, and $k\in\N_0$, let $C_{Q,\nu,k}\in\C$ denote
the total coefficient of $e^{Q(r)}r^{\nu}(\log r)^{k}$ in the formal asymptotic expression $\widetilde v$,
obtained by summing the contributions to this term from the indices $j\in\{1,\,2,\,\ldots,\,m\}$ for which $Q_j=Q$, together with
$-u(x_0)$ when $(Q,\nu,k)=(0,0,0)$. For each fixed $Q$, the exponents $\nu$
occurring lie in the finite union of the sets $\kappa_j+\frac1q\N_0$ over those
$j$ with $Q_j=Q$, so the set of real parts $\Re\nu$ occurring is bounded below.
Finally, if $Q\ne0$ and $c_j\ne0$ for at least one integer $1\le j\le m$ with
$Q_j=Q$, then $\sum_{1\le j\le m\: Q_j=Q}c_j\widetilde z_j\ne0$ by property~(ii) of
Lemma~\ref{lem:local expansion}, and
hence $C_{Q,\nu,k}\ne0$ for at least one pair $(\nu,k)$.

We now consider the possible behaviours of the exponential factors.
For each integer $1\le j\le m$, $\Re Q_j$ is a real polynomial in $r^{-1/q}$ without constant
term. So as
$r\downarrow0$, $\Re Q_j$ either vanishes identically, or tends to $+\infty$, or tends to $-\infty$; in the latter two cases, the behaviour is determined by the sign of its leading nonzero coefficient. We consider these three cases separately.

\smallskip
\emph{Claim 1.} For all integers $1\le j\le m$ with $c_j\ne0$, the case $\Re Q_j(r)\to+\infty$ as $r\downarrow0$
cannot occur.

\smallskip
Suppose that $\Re Q_{j'}(r)\to+\infty$ as
$r\downarrow0$ for some integer $1\le j'\le m$ with $c_{j'}\ne0$. Among the polynomials $Q_j$ with
$c_j\ne0$, together with $Q=0$ if $u(x_0)\ne0$, Lemma~\ref{lem:comparison}
provides one, say $Q^\ast$, with $\Re Q_j\le\Re Q^\ast$ near $0$ for all $Q_j$ with $c_j\neq 0$.
Then $\Re Q^\ast(r)\to+\infty$ as $r\downarrow 0$ by the assumption on $Q_{j'}$, and in particular $Q^\ast\ne0$; since the
constant $-u(x_0)$ contributes only to $C_{0,0,0}$, it contributes nothing to
any $C_{Q,\nu,k}$ with $Q=Q^\ast$.
Write
\begin{equation*}
\cQ^\ast\=\{Q_j:j\in \N,\,j\le m,\,c_j\ne0,\text{ and }\Re Q_j\equiv\Re Q^\ast\}.
\end{equation*}
Since $Q^\ast\in\cQ^\ast$ and $Q^\ast\ne0$, there is a pair $(\nu,k)$ with
$C_{Q^\ast,\nu,k}\ne0$. Consequently
\begin{equation*}
\rho^\ast\=\min\{\Re\nu: \nu\in \C,\, C_{Q,\nu,k}\ne0
 \text{ for some }Q\in\cQ^\ast\text{ and some }k\in\N_0\}
\end{equation*}
is well defined, since the set on the right is nonempty and is contained in a finite union of arithmetic progressions of the form $\Re\kappa_j+\frac1q\N_0$ with integer $1\le j\le m$.

Now fix $\tau>\max\{\rho^\ast,\,0\}$. For each integer $1\le j\le m$, every exponent occurring in $\widetilde z_j$ has the form $\nu=\kappa_j+i/q$ with $i\in \N_0$.
Taking the truncations sufficiently long, we may assume that, for each $1\le j\le m$, the remainder
$\zeta_j-\widetilde z_{j,[n]}$
is
$\cO\bigl(e^{\Re Q_j(r)}r^\tau\bigr)$, by \eqref{eq:remainder} with $\ell=0$, and 
that all terms of $\widetilde v$ with
$\Re\nu\le\rho^\ast$ are retained.
Let $w$ be the finite sum obtained from the retained terms by collecting
all contributions corresponding to the same triple $(Q,\nu,k)$, and
discarding those triples whose resulting coefficient vanishes.
Since all
contributions to a triple with $\Re\nu\le\rho^\ast$ are retained, the coefficient in $w$ of
each such triple is exactly the total coefficient $C_{Q,\nu,k}$.
For triples with
$\Re\nu>\rho^\ast$ the truncation may omit some contributions, so their coefficients in $w$
need not be the total coefficients; the argument below does not use them.
Since $\Re Q_j(r)\le\Re Q^\ast(r)$ for each integer $j$ with $c_j\ne0$ and
all sufficiently small $r>0$, the errors above combine to give
\begin{equation}\label{eq:truncation error}
\abs{v(r)-w(r)}\le C_1\,e^{\Re Q^\ast(r)}\,r^{\tau},
\end{equation}
for some constant $C_1>0$ and all sufficiently small $r>0$. The triples $(Q,\nu,k)$ occurring in $w$ are pairwise distinct and carry nonzero coefficients, and no $Q$ occurring in $w$ has $\Re Q$ exceeding $\Re Q^\ast$ near $0$.
Moreover at least one does attain it: the minimum defining $\rho^\ast$ is attained, so there are $Q\in\cQ^\ast$, $\nu$ with
$\Re\nu=\rho^\ast$, and $k$ such that $C_{Q,\nu,k}\ne0$, and by the choice of the truncations the corresponding coefficient in $w$ is the same nonzero coefficient $C_{Q,\nu,k}$.
Hence, the function $\Psi$ of Lemma~\ref{lem:comparison}, formed from the
polynomials occurring in $w$, agrees with $\Re Q^\ast$ near $0$. The
terms indexed by $I^\ast$ are precisely those with
$\Re Q\equiv\Re Q^\ast$, and
the quantity $\rho^\ast$ that \eqref{eq:rho K star} then produces is the one
defined above. Lemma~\ref{lem:leading term}, applied to $w$, therefore yields
$K^\ast\in\N_0$, $c>0$, and a sequence $r_l\to0$ as $l\to +\infty$ such that
$\abs{w(r_l)}\ge c\,e^{\Re Q^\ast(r_l)}r_l^{\rho^\ast}\abs{\log r_l}^{K^\ast}$.
Combining this with \eqref{eq:truncation error} and using the triangle inequality gives
\begin{equation}\label{eq:ReQ to +infty}
\abs{v(r_l)}\ge \bigl(c\abs{\log r_l}^{K^\ast}-C_1r_l^{\tau-\rho^\ast}\bigr) e^{\Re Q^\ast(r_l)}\,r_l^{\rho^\ast}
\ge(1/2)c\,e^{\Re Q^\ast(r_l)}\,r_l^{\rho^\ast}
\end{equation}
for all sufficiently large $l\in\N$, where the last inequality holds because $\tau>\rho^\ast$.
Since $\Re Q^\ast(r)$ is a real polynomial in $r^{-1/q}$ tending to $+\infty$ as $r\downarrow 0$, $e^{-\Re Q^\ast(r)}=\cO\bigl(r^N\bigr)$ for every $N\in\N$.
Hence the right-hand side of \eqref{eq:ReQ to +infty} tends to $+\infty$, contradicting the boundedness of $v$ near $0$ by continuity. The claim follows.

\smallskip
By Claim~1 and the preceding classification for $\Re Q_j$, each $Q_j$ with $c_j\ne0$ satisfies either $\Re Q_j(r)\to-\infty$ as $r\downarrow0$ or $\Re Q_j\equiv0$, and
hence $e^{\Re Q_j(r)}\le1$ for all sufficiently small $r>0$. Thus, for every constant $\tau>0$, truncating \eqref{eq:formal solution} far enough yields a finite sum $\wh w$, obtained from the retained terms by collecting
all contributions corresponding to the same triple $(Q,\nu,k)$, and discarding the vanishing ones, with $v(r)-\wh w(r)=\cO(r^{\tau})$ and no exponential factor remaining in the error.

We first consider the terms for which $\Re Q(r)\to-\infty$ as $r\downarrow 0$. For such a $Q$, the polynomial $\Re Q$ in $r^{-1/q}$ has negative leading coefficient, and hence
$e^{\Re Q(r)}=\cO\bigl(r^N\bigr)$ for all $N\in \N$; since $r^\nu(\log r)^k$ is bounded by a
power of $r^{-1}$, each term of $\wh w$ carrying such a factor is $\cO\bigl(r^N\bigr)$ for
all $N\in\N$. 
Now let $\tau>0$, let $\wh w$ be as above,
and let
$w_0$ be the sub-sum of $\wh w$ consisting of its terms with $\Re Q\equiv0$.
Since $\wh w$ is a finite sum,  $\wh w-w_0=\cO(r^\tau)$,
and together with
$v-\wh w=\cO(r^{\tau})$, this gives $v-w_0=\cO(r^{\tau})$.

It remains to consider the terms with $\Re Q\equiv0$. Set
\begin{equation*}
\Xi\=\{(Q,\nu,k):\Re Q\equiv0\text{ and }C_{Q,\nu,k}\ne0\}.
\end{equation*}
Suppose first that $\Xi=\emptyset$, and let $\tau>0$. Take the truncation of \eqref{eq:formal solution} long 
enough that $v-\wh w=\cO(r^{\tau})$ and all terms with $\Re\nu\le\tau$ are retained. Then the coefficient 
in $w_0$
of every triple
$(Q,\nu,k)$ with $\Re Q\equiv0$ and $\Re\nu\le\tau$ is the total coefficient $C_{Q,\nu,k}$,
which vanishes since $\Xi=\emptyset$.
Thus, $w_0$ is a finite linear combination of terms $e^Qr^\nu(\log r)^k$, with $\Re Q\equiv 0$ and $\Re \nu >\tau$, so $w_0=\cO(r^\tau)$,
and together with $v(r)-w_0(r)=\cO(r^{\tau})$ this gives $v=\cO(r^\tau)$. Since $\tau>0$ was arbitrary, \eqref{eq:gamma bound} below holds with every $\gamma>0$ and $J_0=0$.

Suppose therefore that $\Xi\ne\emptyset$, and set
$\rho^{\dagger}\=\min\{\Re\nu:(Q,\nu,k)\in\Xi\}$, again a minimum over a set
of reals that is bounded below and is contained in a finite union of arithmetic progressions of the form $\Re\kappa_j+\frac1q\N_0$.

\smallskip
\emph{Claim 2.} $\rho^{\dagger}>0$.

\smallskip
Fix $\tau>\max\{\rho^{\dagger},\,0\}$, truncate long enough so that every term with $\Re\nu\le\rho^{\dagger}$ is retained, and let $w_0$ be the sub-sum of $\wh w$ consisting of its terms with $\Re Q\equiv0$. For this truncation, we still have $v-w_0=\cO(r^\tau)$ by the argument above.
The sum $w_0$ is not empty: the minimum defining
$\rho^{\dagger}$ is attained, and for the corresponding triple $(Q,\nu,k)$ its coefficient in $w_0$ is the total coefficient $C_{Q,\nu,k}$, which is nonzero by the definition of $\Xi$.
Since all terms occurring in $w_0$ have $\Re Q\equiv0$, the function
$\Psi$ of Lemma~\ref{lem:comparison}, formed from the polynomials occurring in $w_0$, vanishes identically,
the associated set $I^\ast$ consists of every index occurring in $w_0$, and
\eqref{eq:rho K star} returns the number $\rho^{\dagger}$ just defined; since
$e^{\Psi}\equiv1$, the exponential factor drops out of
\eqref{eq: w asymptotic}. Lemma~\ref{lem:leading term} then gives
$K^{\dagger}\in\N_0$, $c'>0$, and $r'_l\to0$ as $l\to +\infty$ with
$\abs{w_0(r'_l)}\ge c'\,(r'_l)^{\rho^{\dagger}}\abs{\log r'_l}^{K^{\dagger}}$.
If $\rho^{\dagger}\le0$, the right-hand side does not tend to $0$ as
$r'_l\downarrow0$, so neither does $\abs{w_0(r'_l)}$; but $v(r)\to0$ as $r\downarrow 0$ and
$v-w_0=\cO(r^{\tau})$ with $\tau>0$, so $w_0(r)\to0$ as $r\downarrow 0$, a contradiction. The claim follows.

\smallskip
Keep the same $\tau$ and $w_0$ as in the proof of Claim~2. 
By the choice of the truncation,
the coefficient in
$w_0$ of each triple $(Q,\nu,k)$ with $\Re Q\equiv0$ and
$\Re\nu<\rho^{\dagger}$ is the total coefficient $C_{Q,\nu,k}$,
which vanishes by the
definition of $\rho^{\dagger}$.
Hence $w_0$ is a finite linear combination of terms $e^{Q(r)}r^{\nu}(\log r)^{k}$ with
$\Absbig{e^{Q(r)}}=1$, $\Re\nu\ge\rho^{\dagger}>0$, and $k\in\N_0$, and
$v-w_0=\cO(r^{\tau})$ with $\tau>0$. 
Setting $\gamma\=\min\{\rho^{\dagger},\,\tau\}>0$,
there are constants $J_0\in\N_0$ and $C>0$ with
\begin{equation}\label{eq:gamma bound}
\abs{v(r)}\le Cr^\gamma(1+\abs{\log r})^{J_0}
\end{equation}
for all sufficiently small $r>0$.
Choose $0<\beta<\min\{1,\,\gamma\}$. The logarithmic factors can be absorbed into
$r^{\gamma-\beta}$, so there is $C'>0$ such that $\abs{v(r)}\le C'r^\beta$
for all sufficiently small $r>0$.

The derivative of $v$ is handled similarly. By the preceding discussion, each
$Q_j$ with $c_j\ne0$ satisfies either $\Re Q_j\equiv0$ or
$\Re Q_j(r)\to-\infty$ as $r\downarrow0$. Hence
$\Absbig{e^{Q_j(r)}}\le1$
for all sufficiently small $r>0$. Fix $\tau>0$. By \eqref{eq:v expansion} and
Lemma~\ref{lem:local expansion}~(iv) with $\ell=1$, $v'$ differs from the derivative of a sufficiently long truncation of $\widetilde v$ by $\cO(r^\tau)$. Since 
\begin{equation*}
\frac{\mathrm d}{\mathrm dr}
\bigl(e^{Q_j(r)}r^\nu(\log r)^k\bigr)
=e^{Q_j(r)}
\bigl(Q_j'(r)r^\nu(\log r)^k
+\nu r^{\nu-1}(\log r)^k
+k r^{\nu-1}(\log r)^{k-1}\bigr),
\end{equation*}
the derivative of the truncation is a finite sum of terms
$e^{Q_j(r)}$ times a function bounded by a power of $r^{-1}$.
Since $\Absbig{e^{Q_j(r)}}\le1$ for all sufficiently small $r$, it follows that $v'(r)$ has at most polynomial growth in $r^{-1}$.
Repeating the argument for the other sign gives analogous estimates.
After decreasing $\beta$, increasing $M$, shrinking the neighbourhood, and enlarging
$C$ if necessary, we may
choose common constants $\beta>0$, $C>0$, and $M\in \N$ such that,
writing
$v_\sigma(r)\=u(x_0+\sigma r)-u(x_0)$,
\begin{equation}
\label{eq:ODE bounds}
\abs{v_\sigma(r)}\le Cr^\beta\quad\text{ and }\quad
\abs{v_\sigma'(r)}\le Cr^{-M}
\end{equation}
for both signs $\sigma$ and all sufficiently small $r>0$.

Shrinking $\eta$ if necessary so that \eqref{eq:ODE bounds} holds throughout
$(0,\eta)$ and $\eta\le\tfrac12$, Lemma~\ref{lem:interpolation} applied to
$g\=u$ shows that $u$ is H\"older continuous on $(x_0-\eta,x_0+\eta)$, with
exponent $\alpha\=\beta/(\beta+M)$ and constant $3C$. 
Since the zeros of $a_m$ are isolated and $u$ is analytic away
from them, $u$ is locally H\"older continuous on $U$.
\end{proof}

For $\alpha\in(0,1]$ we write $\Holder{\alpha}(\T,\R)$ for the space of real-valued $\alpha$-H\"older functions on $\T$, with 
respect to
$d_\T$, the quotient metric on $\T=\R/\Z$ induced by the Euclidean
metric on $\R$.

\begin{prop}[Regularity of functions with P-recursive Fourier coefficients]
\label{prop:P-recursive Fourier coefficients}
Let $F\in C(\T,\R)$. If $\bigl\{\wh F(n)\bigr\}_{n\in\N_0}$ is P-recursive, then $F\in\Holder{\alpha}(\T,\R)$ for some $\alpha\in(0,1]$, and $F$ is real analytic outside a finite subset of $\T$.
\end{prop}

\begin{proof}
By Lemma~\ref{lem:P-recursive Fourier ODE}, $F$ satisfies
\eqref{eq:Fourier differential equation} for some real-valued
trigonometric polynomials $a_0$, $a_1$, $\ldots$, $a_\ell$ with $a_\ell\not\equiv0$. In particular, the coefficients are real analytic. Consider $F$ and all the coefficients $a_j$ as $1$-periodic functions on $\R$. For
$\phi\in C^\infty_c(\R,\C)$, the periodisation $\Phi\=\sum_{k\in\Z}\phi(\,\cdot\,+k)$
lies in $C^\infty(\T,\C)$, and, for each integer $j\in[0,\ell]$,
$\sum_{k\in \Z}(a_j\phi)^{(j)}(\,\cdot\,+k)=(a_j\Phi)^{(j)}$; hence
\begin{equation*}
\sum_{j=0}^{\ell}(-1)^j\int_\R\!F(x)\,(a_j\phi)^{(j)}(x)\,\mathrm dx
=\sum_{j=0}^{\ell}(-1)^j\int_\T\!F(x)\,(a_j\Phi)^{(j)}(x)\,\mathrm dm_\T(x)=0,
\end{equation*}
where the last equality is \eqref{eq:Fourier differential equation} applied to the test
function $\Phi$.
Thus \eqref{eq:Fourier differential equation} also holds in $\cD'(\R)$.
Lemma~\ref{lem:singular ODE regularity}, applied with $U\=\R$ and
$m\=\ell$, shows that $F$ is locally H\"older continuous on $\R$, hence on $\T$,
and that
$F$ is real analytic outside the zero set of $a_\ell$ in $\T$, which is finite since $a_\ell$
is a nonzero trigonometric polynomial.
Since $\T$ is compact,
choose a finite open cover $U_1$, $U_2$, $\ldots$, $U_N$ of $\T$ such that, for each $i\in\{1,\,2,\,\ldots,\,N\}$, there exist  constants
$C_i>0$ and $\alpha_i\in(0,1]$ such that
\begin{equation*}
\abs{F(x)-F(y)}
\le C_i d_{\T}(x,y)^{\alpha_i},
\quad\text{ for all } x,\,y\in U_i.
\end{equation*}
Set $\alpha\=\min\{\alpha_i\: i\in \N,\,i\le N \}$ and $C_0\=\max\{C_i\: i\in \N,\, i\le N\}$, and let $\delta\in(0,1)$ be a Lebesgue number for this cover.
Fix $x$, $y\in \T$.
If $d_{\T}(x,y)<\delta$, then $x$, $y$ belong to some common $U_i$. Since $d_{\T}(x,y)<1$
and $\alpha\le\alpha_i$, we have
$\abs{F(x)-F(y)}\le C_0d_{\T}(x,y)^\alpha$.
If $d_{\T}(x,y)\ge\delta$, then $\abs{F(x)-F(y)}\le 2\norm{F}_\infty\le2\norm{F}_\infty\delta^{-\alpha}d_{\T}(x,y)^\alpha$.
Therefore $F\in\Holder{\alpha}(\T,\R)$, with H\"older constant at most
$\max\{C_0,\,2\norm{F}_\infty\delta^{-\alpha}\}$.
\end{proof}

\begin{rem}
The exponent in Proposition~\ref{prop:P-recursive Fourier coefficients} cannot be bounded
below independently of $F$, and $F$ need not be real analytic on the whole of $\T$. For $0<\alpha\le1$, the function $u(x)\=\abs{\sin\pi x}^\alpha$
satisfies $(1-\cos(2\pi x))\,u'=\alpha\pi\sin(2\pi x)\,u$ in $\cD'(\T)$, so its
Fourier coefficients are P-recursive by Remark~\ref{rem:ODE converse}.
It is H\"older
continuous of exponent $\alpha$, but not of any larger exponent,
and it is real analytic on
$\T\smallsetminus\{0\}$ but not at $0$.
\end{rem}

We can now prove Theorem~\ref{thm:P-recursive coboundary}.

\begin{proof}[Proof of Theorem~\ref{thm:P-recursive coboundary}]
The first statement is
Proposition~\ref{prop:P-recursive Fourier coefficients}.
For the
second, let $T\in C^1(\T,\T)$ satisfy $\abs{T'}>1$. By compactness there is $\lambda>1$ with
$\abs{T'}\ge\lambda$ on $\T$, so $T$ is a covering map of $\T$, each of whose inverse
branches contracts distances by a factor of at most $\lambda^{-1}$.
In particular, for all $x,\, y\in\T$ there is $x'\in T^{-1}(y)$ with
$d_\T(x,x')\le d_\T(T(x),y)$, which is the weak expansion property
of \cite[Definition~3]{Bou01}. Moreover, as an expanding circle map, $T$ is
topologically transitive.
Since $F$ is H\"older
continuous and the inverse branches of $T$ are contractions, $F$ satisfies the Walters
condition with respect to $T$, and the Ma\~n\'e lemma \cite[Theorem~1]{Bou01} (see also
e.g.~\cite{CLT01}) gives $\varphi\in C(\T,\R)$ with
\begin{equation*}
F-Q(T,F)\le\varphi\circ T-\varphi
\quad\text{ on }\T.
\end{equation*}
If $\mu$ is an $F$-maximizing measure, then the continuous function
$\varphi\circ T-\varphi-F+Q(T,F)$ is nonnegative and has zero integral with respect to $\mu$,
so it vanishes on $\supp\mu$. Taking $\mu$ of full support, it vanishes identically, so
$F-Q(T,F)=\varphi\circ T-\varphi$ is a coboundary.
\end{proof}

\begin{rem}
\label{rem:uniformly expanding}
In the proof of the second statement of Theorem~\ref{thm:P-recursive coboundary}, the
hypothesis on $T$ is used only to ensure that $T$ is transitive 
and has the weak expansion
property of \cite[Definition~3]{Bou01},
and that H\"older continuous
functions satisfy the Walters condition with respect to $T$. All three properties hold for all
continuous maps $T\colon\T\to\T$ that are distance-expanding, in the sense that there exist
$\varepsilon>0$ and $\lambda>1$ such that $d_\T(T(x),T(y))\ge\lambda\,d_\T(x,y)$
whenever $d_\T(x,y)\le\varepsilon$, so the second statement of Theorem~\ref{thm:P-recursive coboundary} in fact
holds for all such maps.
\end{rem}


\begin{thebibliography}{PER89}

        \bibitem[Al71]{Al71}
        \textsc{Alfsen,~E.M.}, 
        \textit{Compact Convex Sets and Boundary Integrals}, Springer, Berlin, 1971.

        \bibitem[Bi09]{Bi09}
        \textsc{Birkhoff,~G.D.},
        Singular points of ordinary linear differential equations.
        \textit{Trans.\ Amer.\ Math.\ Soc.} {\bf 10} (1909), 436--470.

        \bibitem[Boc18]{Boc18}
		\textsc{Bochi,~J.},
		Ergodic optimization of Birkhoff averages and Lyapunov exponents.
		In \textit{Proc.\ Internat.\ Congr.\ Math.\ (Rio de Janeiro 2018)}, Volume~III,
		World Sci.\ Publ., Singapore, 2018, pp.~1825--1846.

        \bibitem[BC13]{BC13}
        \textsc{Borwein,~J.M.} and \textsc{Crandall,~R.E.},
        Closed forms: what they are and why we care.
        \textit{Notices Amer.\ Math.\ Soc.} {\bf 60} (2013), 50--65.

        \bibitem[Bou00]{Bou00}
		\textsc{Bousch,~T.}, 
		Le poisson n'a pas d'ar\^{e}tes.
		\textit{Ann.\ Inst.\ Henri Poincar\'{e} Probab.\ Stat.} {\bf 36} (2000), 489--508.

        \bibitem[Bou01]{Bou01}
        \textsc{Bousch,~T.}, 
        La condition de Walters.
        \textit{Ann.\ Sci.\ \'Ec.\ Norm.\ Sup\'er.\ (4)} {\bf 34} (2001), 287--311.

        \bibitem[Bou11]{Bou11}
        \textsc{Bousch,~T.}, 
        Le lemme de Ma\~n\'e-Conze-Guivarc’h pour les syst\`emes amphidynamiques rectifiables.
        \textit{Ann.\ Fac.\ Sci.\ Toulouse Math.\ (6)} {\bf 20} (2011), 1--14.

         \bibitem[BJ02]{BJ02}
		\textsc{Bousch,~T.} and \textsc{Jenkinson,~O.},
		Cohomology classes of dynamically non-negative $C^k$ functions.
		\textit{Invent.\ Math.} {\bf 148} (2002), 207--217.

        \bibitem[Bow71]{Bow71}
        \textsc{Bowen,~R.}, Entropy for group endomorphisms and homogeneous spaces.
        \textit{Trans. Amer. Math. Soc.} {\bf 153} (1971), 401--414.

        \bibitem[CH10]{CH10}
        \textsc{Chazottes,~J.-R.} and  \textsc{Hochman,~M.}, 
        On the zero-temperature limit of Gibbs states. \textit{Comm. Math. Phys.} {\bf 297} (2010), 265--281.

        \bibitem[Co16]{Co16}
		\textsc{Contreras,~G.},
		Ground states are generically a periodic orbit.  
		\textit{Invent.\ Math.} {\bf 205} (2016), 383--412.

        \bibitem[CLT01]{CLT01}
		\textsc{Contreras,~G.}, \textsc{Lopes,~A.O.}, and \textsc{Thieullen,~Ph.},
		Lyapunov minimizing measures for expanding maps of the circle.
		\textit{Ergodic Theory Dynam.\ Systems} {\bf 21} (2001), 1379--1409.

        \bibitem[CG95]{CG95}
        \textsc{Conze,~J.-P.} and \textsc{Guivarc'h,~Y.},
        Croissance des sommes ergodiques et principe variationnel.
        \textit{Unpublished manuscript}, circa 1995.

        \bibitem[Co34]{Co34}
        \textsc{Cope,~F.T.},
        Formal solutions of irregular linear differential equations.~I.
        \textit{Amer.\ J.\ Math.} {\bf 56} (1934), 411--437.

        \bibitem[Fa85]{Fa85}
        \textsc{Fabry,~E.},  
        \textit{Sur les int\'egrales des \'equations diff\'erentielles lin\'eaires \`a coefficients rationnels}. 
        Ph.D.\ Thesis, Facult\'e des sciences de Paris, 1885.
        \url{https://patrimoine.sorbonne-universite.fr/fonds/item/3199-sur-les-integrales-des-equations-differentielles-lineaires-a-coefficients-rationnels}

        \bibitem[FSV24]{FSV24}
        \textsc{Fan},~A., \textsc{Schmidt,~K.}, and \textsc{Verbitskiy,~E.},
        Bohr chaoticity of principal algebraic actions and Riesz product measures.
        \textit{Ergodic Theory Dynam.\ Systems} {\bf 44} (2024), 2933--2959.

        \bibitem[Fi49]{Fi49}
        \textsc{Fine,~N.J.},
        On the Walsh functions.
        \textit{Trans.\ Amer.\ Math.\ Soc.} {\bf 65} (1949), 372--414.

        \bibitem[Fo16]{Fo16}
        \textsc{Folland,~G.B.},
        \emph{A Course in Abstract Harmonic Analysis},
        2nd ed.,
        CRC Press, Boca Raton, FL, 2016.

        \bibitem[GS24]{GS24}
        \textsc{Gao},~R. and \textsc{Shen},~W.,
        Low complexity of optimizing measures over an expanding circle map.
        \textit{Comm.\ Math.\ Phys.}
        {\bf 405} (2024), article no.~246.

        \bibitem[GSZ25]{GSZ25}
        \textsc{Gao},~R., \textsc{Shen},~W., and \textsc{Zhang},~R.,
        Typicality of periodic optimization over an expanding circle map.
        Preprint, (arXiv:2501.10949), 2025.

        \bibitem[Gr14]{Gr14}
        \textsc{Grafakos,~L.},
        \emph{Classical Fourier Analysis},
        3rd ed.,
        Springer, New York, 2014.

        \bibitem[Haa10]{Haa10}
        \textsc{Haar,~A.},
        Zur Theorie der orthogonalen Funktionensysteme.
        \textit{Math.\ Ann.} {\bf 69} (1910), 331--371.

        \bibitem[Haa33]{Haa33}
        \textsc{Haar,~A.},
        Der Massbegriff in der Theorie der kontinuierlichen Gruppen.
        \textit{Ann.\ of Math.} {\bf 34} (1933), 147--169.

        \bibitem[Hal43]{Hal43}
        \textsc{Halmos,~P.R.},
        On automorphisms of compact groups.
        \textit{Bull.\ Amer.\ Math.\ Soc.} {\bf 49} (1943), 619--624.

        \bibitem[HHJL25]{HHJL25}
		\textsc{Hao},~Z., \textsc{Huang},~Y., \textsc{Jenkinson,~O.}, and~\textsc{Li},~Z.,
		Joint typical periodic optimization.
		Preprint, (arXiv:2502.12269), 2025.

        \bibitem[HHJL26]{HHJL26}
		\textsc{Hao},~Z., \textsc{Huang},~Y., \textsc{Jenkinson,~O.}, and~\textsc{Li},~Z.,
        Joint typical periodic optimization: systems with stable hyperbolicity.
		Preprint, (arXiv:2605.01550), 2026.

        \bibitem[HZ66]{HZ66}
        \textsc{Hewitt,~E.} and \textsc{Zuckerman,~H.S.},
        Singular measures with absolutely continuous convolution squares.
        \textit{Math.\ Proc.\ Cambridge Philos.\ Soc.} {\bf 62} (1966), 399--420.

        \bibitem[HJXZ26]{HJXZ26}
        \textsc{Huang},~W., \textsc{Jenkinson},~O.,  \textsc{Xu},~L., and \textsc{Zhang},~Y.,
         Typical periodic optimization for dynamical systems: symbolic dynamics.
         \textit{Invent.\ Math.} {\bf 245} (2026), 1--63.

         \bibitem[HLMXZ25]{HLMXZ25}
		\textsc{Huang},~W., \textsc{Lian},~Z., \textsc{Ma},~X., \textsc{Xu},~L., and \textsc{Zhang},~Y.,
		Ergodic optimization theory for a class of typical maps.
		To appear in \textit{J.\ Eur.\ Math.\ Soc.},
\url{https://doi.org/10.4171/JEMS/1652}.

        \bibitem[Je06a]{Je06a}
        \textsc{Jenkinson,~O.},
        Ergodic optimization.
        \textit{Discrete Contin.\ Dyn.\ Syst.} {\bf 15} (2006), 197--224.

        \bibitem[Je06b]{Je06b}
        \textsc{Jenkinson,~O.},
        Every ergodic measure is uniquely maximizing.
        \textit{Discrete Contin.\ Dyn.\ Syst.} {\bf 16} (2006), 383--392.

        \bibitem[Je19]{Je19}
		\textsc{Jenkinson,~O.}, 
		Ergodic optimization in dynamical systems.
		\textit{Ergodic Theory Dynam.\ Systems} {\bf 39} (2019), 2593--2618.

        \bibitem[Ka23]{Ka23}
        \textsc{Kauers,~M.},
        \textit{D-Finite Functions}, vol.~30 in \textit{Algorithms and Computation in Mathematics},
        Springer, Cham, 2023.

        \bibitem[LZ25]{LZ25}
		\textsc{Li},~Z. and \textsc{Zhang},~Y.,
		Ground states and periodic orbits for expanding Thurston maps.
		\textit{Math.\ Ann.} {\bf 391} (2025), 3913--3985.

        \bibitem[Li82]{Li82}
        \textsc{Lind,~D.A.},
        Ergodic group automorphisms are exponentially recurrent.
        \textit{Israel J.\ Math.} {\bf 41} (1982), 313--320.

        \bibitem[LT03]{LT03}
        \textsc{Lopes,~A.O.} and \textsc{Thieullen,~Ph.},
        Sub-actions for Anosov diffeomorphisms, 
        in Geometric Methods in Dynamics (II): Volume in Honor of Jacob Palis, \textit{Ast\'{e}risque} {\bf 287} (2003), 135--146.  

        \bibitem[Ma99]{Ma99}
        \textsc{Math\'e,~P.},
        Approximation of H\"older continuous functions by Bernstein polynomials.
        \textit{Amer.\ Math.\ Monthly} {\bf 106} (1999), 568--574.

        \bibitem[PWZ96]{PWZ96}
        \textsc{Petkov\v{s}ek,~M.}, \textsc{Wilf,~H.S.}, and \textsc{Zeilberger,~D.},
        \textit{$A=B$}, 
        A~K~Peters, Wellesley, MA, 1996.

        \bibitem[Ph03]{Ph03}
        \textsc{Phillips,~G.M.},
        \emph{Interpolation and Approximation by Polynomials},
        Springer, New York, 2003.

        \bibitem[PER89]{PER89}
        \textsc{Pour-El,~M.B.} and \textsc{Richards,~J.I.},
        \emph{Computability in Analysis and Physics},
        Perspect.\ Math.\ Logic, Springer, Berlin, 1989.



        \bibitem[Ri18]{Ri18}
        \textsc{Riesz,~F.},
        \"Uber die Fourierkoeffizienten einer stetigen Funktion von beschr\"ankter Schwankung.
        \textit{Math.\ Z.} {\bf 2} (1918), 312--315.

        \bibitem[Ro49]{Ro49}
        \textsc{Rohlin,~V.A.},
        On endomorphisms of compact commutative groups (Russian). 
        \textit{Izv.\ Akad.\ Nauk SSSR Ser.\ Mat.} {\bf 13} (1949), 329--340.

        \bibitem[Sa99]{Sa99} 
        \textsc{Savchenko,~S.V.},
        Cohomological inequalities for finite topological Markov chains (Russian). 
        \textit{Funktsional.\ Anal.\ i Prilozhen} {\bf 33} (1999), 91--93; translation in
        \textit{Funct.\ Anal.\ Appl.} {\bf 33} (1999), 236--238.

               \bibitem[Sc66]{Sc66}
        \textsc{Schwartz,~L.},
        \textit{Th\'eorie des distributions},
        Hermann, Paris, 1966.


        \bibitem[SR88]{SR88}
        \textsc{Shirvani,~M.} and \textsc{Rogers,~T.D.},
        Ergodic endomorphisms of compact abelian groups.
        \textit{Comm.\ Math.\ Phys.} {\bf 118} (1988), 401--410.

        \bibitem[St80]{St80}
        \textsc{Stanley,~R.P.},
        Differentiably finite power series.
        \textit{European J.\ Combin.} {\bf 1} (1980), 175--188.

        \bibitem[Sv81]{Sv81}
        \textsc{Svensson,~S.L.},
        Singular differential operators and distributions.
        \textit{Israel J.\ Math.} {\bf 38} (1981), 131--153.

        \bibitem[Te12]{Te12}
        \textsc{Teschl,~G.},
        \textit{Ordinary Differential Equations and Dynamical Systems},
        Amer.\ Math.\ Soc., Providence, RI, 2012.

        \bibitem[Wals23]{Wals23}
        \textsc{Walsh,~J.L.},
        A closed set of normal orthogonal functions.
        \textit{Amer.\ J.\ Math.} {\bf 45} (1923), 5--24.

        \bibitem[Walt78]{Walt78}
        \textsc{Walters,~P.},
        Invariant measures and equilibrium states for some mappings which expand distances.
        \textit{Trans.\ Amer.\ Math.\ Soc.} {\bf 236} (1978), 121--153.

        \bibitem[Walt82]{Walt82}
        \textsc{Walters,~P.},
        \emph{An Introduction to Ergodic Theory}, Springer, New York, 1982.

        \bibitem[Was65]{Was65}
        \textsc{Wasow,~W.},
        \textit{Asymptotic Expansions for Ordinary Differential Equations},
        Interscience Publishers John Wiley \& Sons, New York, 1965.

        \bibitem[Ze90]{Ze90}
        \textsc{Zeilberger,~D.},
        A holonomic systems approach to special functions identities.
        \textit{J.\ Comput.\ Appl.\ Math.} {\bf 32} (1990), 321--368.

\end{thebibliography}
\end{document}